\documentclass[11pt, reqno]{amsart}
\usepackage[colorlinks,citecolor=red,pagebackref,hypertexnames=false]{hyperref}

\usepackage[T1]{fontenc}

\usepackage{textcomp}

\large

\makeatletter
\def\@tocline#1#2#3#4#5#6#7{\relax
  \ifnum #1>\c@tocdepth 
  \else
    \par \addpenalty\@secpenalty\addvspace{#2}%
    \begingroup \hyphenpenalty\@M
    \@ifempty{#4}{%
      \@tempdima\csname r@tocindent\number#1\endcsname\relax
    }{%
      \@tempdima#4\relax
    }%
    \parindent\z@ \leftskip#3\relax \advance\leftskip\@tempdima\relax
    \rightskip\@pnumwidth plus4em \parfillskip-\@pnumwidth
    #5\leavevmode\hskip-\@tempdima
      \ifcase #1
       \or\or \hskip 1em \or \hskip 2em \else \hskip 3em \fi%
      #6\nobreak\relax
    \dotfill\hbox to\@pnumwidth{\@tocpagenum{#7}}\par
    \nobreak
    \endgroup
  \fi}
\makeatother

\usepackage{tikz,amsthm,amsmath,amstext,amssymb,amscd,epsfig,euscript, mathrsfs, dsfont,pspicture,multicol, mathtools,graphpap,graphics,graphicx,times,enumerate,subfig,sidecap,wrapfig, color,array}

\usepackage{ hyperref}
\usepackage{enumerate}
\usepackage{esint}
\usepackage{verbatim} 
\usepackage[many]{tcolorbox}
\usepackage{dsfont}

\numberwithin{equation}{section}

\def\bB{{\mathbb{B}}}

\def\R{{\mathbb{R}}}

\newcommand{\dA}{\mathcal{A}}
\newcommand{\dG}{\mathcal{G}}
\newcommand{\dM}{\mathcal{M}}
\newcommand{\dD}{\mathcal{D}}
\newcommand{\dH}{\mathcal{H}}
\newcommand{\dW}{\mathcal{W}}
\newcommand{\dI}{\mathcal{I}}
\newcommand{\dT}{\mathcal{T}}
\newcommand{\dN}{\mathcal{N}}
\newcommand{\dS}{\mathcal{S}}
\newcommand{\dV}{\mathcal{V}}

\newcommand{\dB}{\mathcal{B}}

\newcommand{\Nw}{N_{\text{\tiny{W}}}}
\newcommand{\bV}{\mathbb{V}}
\newcommand{\bY}{\mathbb{Y}_{{\text{\tiny{W}}}}}
\newcommand{\Gaw}{\overline{\mathbb{Y}}_{{\text{\tiny{W}}}}}

\def\cC{{\mathscr{C}}}

\def\cH{{\mathscr{H}}}
\def\cI{{\mathscr{I}}}

\def\one{\mathds{1}}

\newcommand{\Stop}{\mathsf{Stop}}
\newcommand{\Tree}{\mathsf{Tree}}
\newcommand{\BA}{\mathsf{BA}}
\newcommand{\BSF}{\mathsf{BSF}}
\newcommand{\BG}{\mathsf{BG}}
\newcommand{\SSL}{\mathsf{SSL}}
\newcommand{\Child}{\mathsf{Child}}

\newcommand{\Grad}{{\rm Grad}}

\newcommand{\wh}[1]{\widehat{#1}}
\newcommand{\wt}[1]{\widetilde{#1}}
\def\ve{\varepsilon}
\def\vp{\varphi}

\renewcommand{\d}{{\partial}}
\def\lec{\lesssim}
\def\gec{\gtrsim}

\newcommand{\vvv}{{\vspace{2mm}}}

\newcommand{\pom}{{\partial \Omega}}

\newcommand{\om}{{\Omega}}

\newcommand{\WW}{{\mathcal W}}
\newcommand{\DD}{{\mathcal D}}
\newcommand{\s}{{\sigma}}

\newcommand{\NN}{{\mathcal N}}

\newcommand{\bwgl}{\mathsf{BWGL}}
\newcommand{\cw}{\mathsf{U}}

\DeclareMathOperator{\diam}{diam}
\def\Lip{\mathop\mathrm{Lip}} 						
\def\dist{\mathop\mathrm{dist}} 						
\def\supp{\mathop\mathrm{supp}}					
\DeclareMathOperator*{\esssup}{ess\,sup}			

\newcommand{\ps}[1]{\left( #1 \right)}

\newcommand{\av}[1]{\left| #1 \right|}

\def\warrow{\rightharpoonup}

\newcommand{\dMi}{\dM^{s}}

\def\Xint#1{\mathchoice
{\XXint\displaystyle\textstyle{#1}}%
{\XXint\textstyle\scriptstyle{#1}}%
{\XXint\scriptstyle\scriptscriptstyle{#1}}%
{\XXint\scriptscriptstyle\scriptscriptstyle{#1}}%
\!\int}
\def\XXint#1#2#3{{\setbox0=\hbox{$#1{#2#3}{\int}$ }
\vcenter{\hbox{$#2#3$ }}\kern-.58\wd0}}

\def\avint{\Xint-}
\def\grad{\nabla}
\newcommand{\gammaq}{\gamma_f^1}
\newcommand{\wtgamma}{\wt{\gamma}_f}

\theoremstyle{maintheorem}
\newtheorem{maintheorem}{Theorem}

\theoremstyle{plain}
\newtheorem{theorem}{Theorem}

\newtheorem{corollary}[theorem]{Corollary}

\newtheorem{lemma}[theorem]{Lemma}

\newtheorem{proposition}[theorem]{Proposition}

\newtheorem{claim}[theorem]{Claim}
\newtheorem{mainlemma}[theorem]{Main Lemma}

\theoremstyle{definition}
\newtheorem{example}[theorem]{Example}
\newtheorem{definition}[theorem]{Definition}
\newtheorem{remark}[theorem]{Remark}

\numberwithin{equation}{section}
\numberwithin{theorem}{section}

  \DeclareFontFamily{U}{mathb}{\hyphenchar\font45} 
\DeclareFontShape{U}{mathb}{m}{n}{
      <5> <6> <7> <8> <9> <10> gen * mathb
      <10.95> mathb10 <12> <14.4> <17.28> <20.74> <24.88> mathb12
      }{}
\DeclareSymbolFont{mathb}{U}{mathb}{m}{n}

\DeclareMathSymbol{\toitself}{3}{mathb}{"FD}  

\def\salpha{\alpha_{\sigma}^{d}}

\makeatletter
\def\@tocline#1#2#3#4#5#6#7{\relax
  \ifnum #1>\c@tocdepth 
  \else
    \par \addpenalty\@secpenalty\addvspace{#2}%
    \begingroup \hyphenpenalty\@M
    \@ifempty{#4}{%
      \@tempdima\csname r@tocindent\number#1\endcsname\relax
    }{%
      \@tempdima#4\relax
    }%
    \parindent\z@ \leftskip#3\relax \advance\leftskip\@tempdima\relax
    \rightskip\@pnumwidth plus4em \parfillskip-\@pnumwidth
    #5\leavevmode\hskip-\@tempdima
      \ifcase #1
       \or\or \hskip 1em \or \hskip 2em \else \hskip 3em \fi%
      #6\nobreak\relax
    \dotfill\hbox to\@pnumwidth{\@tocpagenum{#7}}\par
    \nobreak
    \endgroup
  \fi}
\makeatother

\renewcommand{\Sigma}{E}

\begin{document}
\setcounter{tocdepth}{1}
\setcounter{secnumdepth}{4}

\title[Quantitative differentiability on UR sets]{Quantitative differentiability on uniformly rectifiable sets}

\author[Azzam, Mourgoglou and Villa]{Jonas Azzam, Mihalis Mourgoglou and Michele Villa}
\address{Jonas Azzam\\
School of Mathematics \\ University of Edinburgh \\ JCMB, Kings Buildings \\
Mayfield Road, Edinburgh,
EH9 3JZ, Scotland.}
\email{j.azzam "at" ed.ac.uk}

\address{Mihalis Mourgoglou\\
Departamento de Matem\'aticas, Universidad del Pa\' is Vasco, Barrio Sarriena s/n 48940 Leioa, Spain and\\ Ikerbasque, Basque Foundation for Science, Bilbao, Spain.} \email{michail.mourgoglou "at" ehu.eus}

\address{Michele Villa\\
	Departamento de Matem\'aticas, Universidad del Pa\' is Vasco, Barrio Sarriena s/n 48940 Leioa, Spain and\\ Ikerbasque, Basque Foundation for Science, Bilbao, Spain.}
\email{michele.villa "at" ehu.eus}

\thanks{J.A. was partially supported by Basque Center for Applied Mathematics (BCAM) while on a research visit to M.M. 
	M.M. was supported  by IKERBASQUE and partially supported by the grant PID2020-118986GB-I00 of the Ministerio de Econom\'ia y Competitividad (Spain) and by the grant IT-1615-22 (Basque Government). M.V. was supported by a starting grant of the University of Oulu and by the Academy of Finland via the project
	“Higher dimensional Analyst’s Traveling Salesman theorems and Dorronsoro estimates on non-smooth sets”, grant
	No. 347828/24304228}



\begin{abstract}
We prove $L^p$ quantitative differentiability estimates for functions defined on uniformly rectifiable subsets of the Euclidean space. More precisely, we show that a Dorronsoro-type theorem holds in this context: the $L^p$ norm of the gradient of a Sobolev function $f: E \to \R$ is comparable to the $L^p$ norm of a new square function measuring both the affine deviation of $f$ and how flat the subset $E$ is. A corollary dealing with extensions and traces of Sobolev functions may be found in \cite{AMV2}.

\end{abstract}
\maketitle
\vspace{-1cm}
\tableofcontents

\section{Introduction and main results}

   \def\M{\mathcal{M}}
\def\G{\mathcal{G}}
 \def\Mt{\tilde{\mathcal{M}}}
  \def\Gt{\tilde{\G}}

A basic fact of Lipschitz functions is that they are differentiable almost everywhere. This is Rademacher's theorem. 
For a real valued Lipschitz function $f$, a point $x \in \R^d$ and a scale $r>0$, define
\begin{align}\label{e:def-Omega-infty}
    \Omega_f(x,r):= \inf_{A} \avint_{B(x,r)} \frac{|A(y)-f(y)|}{r} \, dy,
\end{align}
where the infimum is taken over all affine maps $A: \R^d \to \R$. Then, by Rademacher's theorem, $\lim_{r\to 0} \Omega_f(x,r) =0$ almost everywhere. Can we quantify this? That is to say, how many scales can effectively be \textit{bad} in the sense that $\Omega_f(x,r) \gtrsim 1$? To illustrate\footnote{We take this example from \cite{semmes2001real}.}, for $\delta>0$, consider $f_\delta(x) := \delta \sin(x/\delta)$. Then, when $r \gg \delta$, $\Omega_{f_\delta}(x,r)$ is small, simply because $f_\delta$ is small compared to $r$ (and we can always take $A \equiv 0$); when $r \approx \delta$, then $\Omega_{f_\delta}(x,r)$ is large, because the oscillations will be of `height' $\delta$; when $r \ll \delta$, then $\Omega_{f_\delta}$ will be small again, because now the smoothness of $\sin$ is felt. Thus, there is essentially \textit{just one} bad scale for $f_\delta$. This cannot hold in general (take a sum of $f_\delta$ for different $\delta's$), though one might hope that the number of bad scales is bounded. This is in fact what Dorronsoro proved (see \cite[Theorem 6]{dorronsoro1985characterization}).  If we fix  $\ve>0$ and let $\#{\rm Bad}_\ve (x)$ be the cardinality of integers $j$ so that $\Omega_f(x,2^{-j}) > \ve$, then $\avint_B \#{\rm Bad}_\ve(x) dx \lesssim C(\ve) \|f\|_{\Lip}$. Dorronsoro's theorem implies this estimate; it is in fact stronger and, importantly, it extends the above discussion to Sobolev functions. For $x\in \R^d$ and $f \in L^q$, define
\begin{equation}
	G_q f(x) := \left( \int_0^\infty \Omega_f^q (x,r)^2 \, \frac{dr}{r} \right)^{\frac{1}{2}},
\end{equation}
where $\Omega^q_f$ is an $L^q$ averaged version of the coefficients in \eqref{e:def-Omega-infty}, see the definition \eqref{e:def-omega} below. Recall that $p^*=\tfrac{pd}{d-p}$.

\begin{theorem}\label{t:dorronsoro-classic}
Fix $1 \leq d \in \mathbb{N}$ and $1<p<\infty$. Let $f \in L^p(\R^d)$ be a real valued function. Then $f \in W^{1,p}(\R^d)$ if and only if $\|\dG_q f\|_{L^p}$ is finite, where $q$ is in the following range:
\begin{itemize}
    \item If $d=1$, then $1 \leq q \leq \infty$.
    \item If $d\geq 2$, then\footnote{We interpret $2^*=\infty$ when $d=2$.} $1 \leq q < p^*$ whenever $1<p<2$; and $1 \leq q < 2^*$ whenever $2 \leq p <\infty$.
\end{itemize}
In all these cases, we have
\begin{align}\label{e:dorronsoro-classic}
	\|G_q f\|_{L^p} \approx \|\grad f\|_{L^p},
\end{align}
where the implicit constant depend on $d,p,q$. 
\end{theorem}
\noindent
Note that an immediate consequence of this theorem is the following estimates for compactly supported Lipschitz functions:
\begin{equation}\label{e:carleson-measure-Omega}
	\int_B \int_0^{r_B} \Omega_f^q(x,r)^2 \, \frac{dr}{r} \, dx \lesssim_{{\rm Lip}(f)} r_B^d, 
\end{equation}
whenever $B$ is a ball with $r_B \leq \diam({\rm spt}(f))$ and $q$ is in the appropriate range. The bound on `bad scales' mentioned above is a consequence of \eqref{e:carleson-measure-Omega}. \\

\noindent
The purpose of this article is to prove a version of Theorem \ref{t:dorronsoro-classic} for subsets $E \subset \R^n$ which are uniformly $d$-rectifiable. 
	Recall that a $d$-Ahlfors regular set $\Sigma \subset \R^{n}$
	is {\it $d$-uniformly rectifiable} or {\it $d$-UR} if and only if it contains ``Big Pieces of
	Lipschitz Images" (``BPLI"). This means that there are positive constants $c$ and
	$L$, such that for each
	$x\in E$ and each $r\in (0,\diam E)$, there is an
	$L$-Lipschitz mapping $f= \rho_{x,r}: \R^d\rightarrow \R^{n}$ such that
	\[
	\mathcal{H}^d\Big(E\cap B(x,r)\cap f (B_{d}(0,r))\Big) \geq cr^{d}.
	\]
\noindent
Recall that a set $E$ is $d$-Ahlfors regular if there is a constant such that 
\begin{equation}\label{e:ADR}
	C^{-1} r^d \leq \mathcal{H}^d(E \cap B(x,r)) \leq Cr^d
\end{equation}
 for any $x \in E$ and $0<r \leq \diam(E)$.
David and Semmes proved an estimate like \eqref{e:carleson-measure-Omega} for \textit{Lipschitz functions} on UR sets, see \cite[Proposition III.4.2]{david-semmes93}.

\begin{theorem}
	\label{t:DS-Dorronsoro}
	Let $E$ be UR and $f:E\rightarrow \R$ be $1$-Lipschitz. If $1\leq q < \frac{2d}{d-2}$ (or $1\leq q \leq \infty$ if $d=1$). Then $\Omega_{f}^{q,N}(x,r)^{2}\frac{dxdr}{r}$ is a $C(N)$-Carleson measure, where $N>1$ and $\Omega_{f}^{q,N}(x,r)=\inf_{A}\Omega_{f}^{q}(x,r,A)$ where the infimum is over all affine maps with $|\grad A|\leq N$. 
\end{theorem}

\noindent
Here, we show that, in fact, a version of Theorem \ref{t:dorronsoro-classic} holds for $L^p$ Sobolev spaces on this class of subsets - that is, uniformly $d$-rectifiable sets. Before stating our result, let us introduce the notion of Sobolev spaces we will use. 
By $M^{1,p}(E)$ we denote the Haj\l{}asz-Sobolev space on $E$.
For $1\leq p <\infty$, we let $M^{1,p}(E)$ the set of functions $u \in L^{p}(E)$ (the measure here is $\sigma=\dH^d|_E$) for which there exists a nonnegative $g \in L^p(E)$ so that
\begin{align}\label{e:upper-gradients-intro}
	|u(x)-u(y)| \leq |x-y| (g(x)+g(y)) \mbox{ for } \mu\mbox{-a.e.} \, x,y \in X.
\end{align}
We refer to any $g \in L^p(E)$ satisfying \eqref{e:upper-gradients-intro} as a \textit{Hai\l{}asz upper gradient}.
The function $g$ satisfying \eqref{e:upper-gradients-intro} and with mininimal $L^p$ norm is denoted by $|\grad_H f|$, and is called the \textit{Haj\l{}asz gradient}. See Definition \ref{def-Hajlasz-grad}.

\begin{maintheorem}\label{t:main-A}
	Let $n \geq 2$, fix $1 \leq d \leq n-1$ and $1<p<\infty$. 
	Suppose that $E \subset \R^n$ is a uniformly $d$-rectifiable set and $f \in M^{1,p}(E)$. Then, if $q$ is satisfies:
	\begin{itemize}
		\item If $d=1$, then $1 \leq q \leq \infty$.
		\item If $d\geq 2$, then\footnote{We interpret $2^*=\infty$ when $d=2$.} $1 \leq q < p^*$ whenever $1<p<2$; and $1 \leq q < 2^*$ whenever $2 \leq p <\infty$,
	\end{itemize}
	we have the bound
	\begin{equation}
		\|\dG^q f\|_{L^p(E)}  \lesssim \|\grad_H f\|_{L^p(E)}.
	\end{equation}
	where $|\grad_H f| \in L^p(E)$ is the minimal Haj\l{}asz upper gradient of $f$. The theorem holds with $\widetilde{\dG}f$, a variant of $\dG^q f$. See \eqref{e:intro-dG-tilde}.
\end{maintheorem}
\noindent We refer to a $q$ satisfying the constraints in Theorem \ref{t:main-A} as being in the \textit{Dorronsoro range}. 

\begin{maintheorem}\label{t:main-B}
	Let $n \geq 2$ and fix $1 \leq d \leq n-1$ and $1<p<\infty$. Let $E \subset \R^n$ be a uniformly $d$-rectifiable set and $f: E \to \R$ be Lipschitz. Then for $1 \leq q \leq \infty$,
	\begin{equation}
		\|\grad_t f\|_{L^p(E)} \lesssim \| \dG^q f\|_{L^p(E)},
	\end{equation}
	where $\grad_t f$ is the tangential gradient of $f$. The theorem also holds with $\widetilde{\dG}f$, a variant of $\dG^q f$. See \eqref{e:intro-dG-tilde}.
\end{maintheorem}
\noindent

\noindent
We define $\dG^q$ and $\wt \dG$. For a ball $B$, we introduce the quantities $\gamma_f^q(B)$ and $\widetilde{\gamma}_f$: for $1\leq q \leq \infty$, a ball $B$ centered on $E$ and an affine map $A$, set
\begin{equation}\label{e:def-gamma}
	\gamma_f^q(x,r) := \inf_A\left\{ \Omega^q_f(B(x,r),A) + |\grad A| \beta_E^{d,q}(x,r)\right\}.
\end{equation}
where the infimum is taken over all affine maps $A: \R^n \to \R$. Here $\Omega_f^q(x,r; A)$ is an $L^q$-averaged version of \eqref{e:def-Omega-infty}, where the difference is taken with respect to the fixed affine map $A$. Also, put
\begin{equation}\label{e:def-gamma-tilde}
	\widetilde{\gamma}_f(B) := \inf_{A} \left\{ \Omega_f^1(B(x,r),A) + |\grad A| \alpha_\sigma^d(x,r) \right\},
\end{equation}
where $\alpha_\sigma^d(x,r)$ is Tolsa's coefficient defined in terms of Wasserstein distance between measures, see \eqref{e:alpha-def}. See also Definition \eqref{d:gamma}.
Now we set
\begin{equation}\label{e:intro-dG-tilde}
	\dG^qf (x) = \left(\sum_{Q \ni x} \gamma^q_f(B_Q)^2 \right)^{\frac{1}{2}} \,\,\mbox{ and }\,\, 	\widetilde{\dG} f (x) = \left(\sum_{Q \ni x} \widetilde{\gamma}_f(B_Q)^2 \right)^{\frac{1}{2}}.
\end{equation}
Here $Q$ is a Christ-David cube (see Section \ref{s:cubes}).\\

\noindent
We list two first applications of our results. The first one is immediate, and it's an application of Theorems \ref{t:main-A} and \ref{t:main-B} in the special case where not only $E$ is $d$-rectifiable, but moreover it supports a \textit{Poincar\'{e} inequality}.
	\begin{maintheorem}\label{t:main-C}
		Let $n \geq 2$, let $2 \leq d \leq n-1$ and $1<p<\infty$. Suppose that $E \subset \R^n$ is uniformly $d$-rectifiable and that it supports a weak $(1,p')$-Poincar\'{e} inequality, for $1\leq p'<p$. Let $f:E\to \R$ be Lipschitz. Then for $q$ in the Dorronsoro range, we have 
		\begin{equation}
			\|\dG^q f\|_{L^p(E)} \approx \| \grad_t f\|_{L^p(\pom)} \approx \| \grad_{H} f \|_{L^p(\pom)}.
		\end{equation}
		The implicit constants are independent of the Lipschitz norm of $f$.
	\end{maintheorem}
	\noindent
	See Definition \ref{def-poincare} below for the precise definition of a set supporting a Poincar\'{e} inequality. The proof of Theorem \ref{t:main-C} is immediate from Theorem \ref{t:main-A}, \ref{t:main-B} and \cite[Lemma 6.5]{mourgoglou2021regularity}, the latter result stating that, when the hypotheses of Theorem \ref{t:main-C}, $\|\grad_t f\|_p \approx \|\grad_H f\|_p$.\\
	
	\noindent
	The second one has to do with extensions of Sobolev functions on the boundary and may be found in the paper \cite{AMV2}.

\subsection{Remarks on Theorems \ref{t:main-A} and \ref{t:main-B}}
	Recall that if $g \in L^p(E)$ is an Haj\l{}asz upper gradient of $f$, then the pair $(f,g)$ supports a weak $(1,p)$-Poincar\'{e} inequality (PI), that is, for each ball $B$ centered on $E$ we have
	\begin{equation}
		\avint_B |f-f_B| \, d \sigma \leq C r_B \left( \avint_{B} g^p \, d \sigma \right)^{\frac{1}{p}},
	\end{equation}
	where from now on $d\sigma=d\dH^d|_E$. See also Proposition \ref{t:Hajlasz-poincare} in Section \ref{s:sobolev}.
	That a Poincar\'{e} inequality holds for $f$ and whatever objects one might consider as `gradient', is fundamental. For example, in a statement like Theorem \ref{t:main-A}, one might be tempted to use the tangential gradient $\grad_t f$ (see Definition \ref{d:tan-grad}), instead of $g$. But consider the following example:\\

	\noindent
	\textit{Example.} Let $P \subset \R^3$ be a $2$-dimensional affine plane. Let $Q_1$ and $Q_2$ be two unit squares in $P$ that are $\ve>0$ apart from each other. Join $Q_1$ and $Q_2$ by an $\ve$-thin and $\ve$-long strip. So now $E$ is connected. We define $f$ so that, as before, $f\equiv 0$ on $Q_1$ and $f\equiv 1$ on $Q_2$. We then define $f$ on the strip so that the resulting function is Lipschitz continuous, with constant $\ve^{-1}$. Note that at all points $x \in E$ and at all scales $1 \geq r \geq \ve$ we will have $\Omega_f(x,r) \gtrsim 1$, while $\nabla f \equiv 0$ in $Q_1, Q_2$. It can then be checked that 
	\begin{align*}
		\int_E \left( \int_{0}^\infty \Omega_f^1(x,r)^2 \frac{dr}{r} \right) d\mathcal{H}^2|_E  \gg \|\grad f\|_2^2.
	\end{align*} 
	The real issue is that in general UR sets do not support a PI between the pair $(f, \grad_t f)$. As we will see below, if we assume this \textit{a priori}, Theorem \ref{t:main-A} holds for $\grad_t f$ as well. It is true, however, that if an Ahlfors $d$-regular set supports a $(1,d)$-PI, then it is uniformly rectifiable, by a result of the first author \cite{azzam2021poincare}.

\begin{remark}[The square function $\dG^q$ and the new coefficients $\gamma_f^q$ and $\widetilde{\gamma}_f$]
	The square functions appearing in Theorems \ref{t:main-A} and \ref{t:main-B} and defined in \eqref{e:def-gamma} and \eqref{e:def-gamma-tilde} are \textit{not} the same as that of Dorronsoro's theorem. Let us see why our results do not hold if we were to use Dorronsoro's coefficients as they are. \\
	
	\noindent
	\textit{Example.}
	Let $g: \R^2 \to \R$ be compactly supported in $B(0,1)$, smooth and with $\grad g \in L^\infty(\R^2)$. Let $G_0$ be the subset of ${\rm spt}(g)$ where $|\grad g| >0$ and assume that $\mathcal{L}^2(G_0)>0$, where $|\cdot|$ is the Lebesgue measure in $\R^2$. We let $E \subset \R^3$ be the graph of $g$. Then clearly $E$ is Ahlfors $2$-regular and it supports a weak $(1,2)$-Poincar\'{e} inequality. Now let $f: E \to \R$ be given by $f(x)= \langle x, e_3 \rangle$, where $(e_1,e_2,e_3)$ is the standard basis of $\R^3$. Set
	\begin{align*}
		G:= \{ x \in E \, |\, x=(p, g(p)) \mbox{ and } p \in G_0 \}.
	\end{align*} 
	Take $p_0\in G$; we can assume that $p_0=0$. Let $\gamma_i$, $i=1,2$ be the curves given by $t \mapsto (t e_i, g(t e_i))$. Then note that $\tau_i:=\dot{\gamma}_i(0)$ is a basis for $T_0 E$, the tangent plane of $E$ at $0$. Then we see that (see \cite[4.16]{simon1983lectures}) 
	\begin{align*}
		\grad_t f (0) = \sum_{i=1,2} (D_{\tau_i} f(0))\tau_i, \mbox{ where } D_{\tau_i}f(0) = \tfrac{d}{dt} f (\gamma_i(t))|_{t=0}.
	\end{align*}
	Now, $\tfrac{d}{dt} f(\gamma_i(t))|_{t=0}=\tfrac{d}{dt} g(t e_i)|_{t=0}= \langle \grad g(0), e_i\rangle$, $i=1,2$. We conclude that it might very well be that $\|\grad f \|_{L^2(E)}^2 >0$. However, note that for all $x \in E$, $r>0$, $\Omega_f(x,r) =0$.\\
	
	\noindent
	The message here is that we need to use coefficients that `see' the geometry of $E$. Hence the definition of $\gamma^q$ and $\wt \gamma$ in \eqref{e:def-gamma} and \eqref{e:def-gamma-tilde}. Note on the other hand that if $E$ is flat, then our coefficients are just Dorronsoro's original ones.
\end{remark}

\subsection{Uniformly rectifiable sets}
There is a roughly analogous story for rectifiable sets. Recall that a set $E \subset \R^n$ with $\mathcal{H}^d(E)<+\infty$ is $d$-rectifiable\footnote{We refer the reader to the comprehensive recent review on rectifiability by Mattila \cite{mattila2021rectifiability}.} if there is a countable family of Lipschitz maps $f_i:\R^d \to\R^n$ so that 
\begin{equation*}
	\dH^d \left(E \setminus \bigcup_{i} f_i(\R^d) \right) = 0.
\end{equation*}
The analogy goes as follows: if we set\footnote{These are the well-known $\beta$ numbers, the first version of which is due to \cite{jones90}.}
\begin{equation}\label{e:betaq}
	\beta_E^{d,q}(x,r) := \inf_{P} \left(\frac{1}{r^d} \int_{B(x,r)}\left( \frac{\dist(y, P)}{r} \right)^q \, d\dH^d|_E(x)\right)^{\frac{1}{q}},
\end{equation}
which is the natural version of $\Omega_f^q$ for sets, it might be checked that 
 if $E$ is $d$-rectifiable, then $\beta_E^1(x,r) \to 0$ as $r \to 0$. In fact, this is essentially a consequence of Rademacher's theorem. We then ask: can we quantify this asymptotic information and obtain a version of Theorem \ref{t:dorronsoro-classic} for sets? Yes, if we are willing to restrict our attention to  \textit{uniformly} rectifiable sets. Uniform rectifiability is a strengthening of the
 qualitative hypothesis that a set E is rectifiable: in any ball centered on a uniformly rectifiable set $E$, a
 quantitatively large part of E can be covered with just one Lipschitz image. This was mentioned above. Another important characterisatio of UR sets is the so-called \textit{strong geometric lemma} (see \cite[C3]{david-semmes91}); that is,  a set $E$ is uniformly $d$-rectifiable if and only if 
 \begin{equation}\label{e:carleson-measure-beta}
 	\int_B \int_0^{r_B} \beta_E^{d,q}(x,r)^2 \, \frac{dr}{r} d \dH^d|_E(x) \lesssim r_B^d,
 \end{equation}
for any ball $B$ centered on $E$.
 This is one of the most influential characterisations of uniform rectifiability (see Section \ref{sec6}). It's the natural counterpart of \eqref{e:carleson-measure-Omega} for sets; in fact, \eqref{e:carleson-measure-Omega} is essential to the proof of \eqref{e:carleson-measure-beta}. In this sense, 
Dorronsoro's result is a cornerstone of the David-Semmes theory.

Now, it is shown in Federer's book \cite{federer2014geometric} that Rademacher's theorem holds for Lipschitz functions defined on $d$-rectifiable sets, in the sense that $f$ is (tangentially\footnote{See Section \ref{s:detour} for a definition.}) differentiable at $\dH^d$-almost all points. The corresponding quantitative result is the above mentioned Proposition \ref{t:DS-Dorronsoro}, which, we recall, say that if $f$ is Lipschitz on a uniformly rectifiable set, then
\begin{equation}\label{e:carleson-measure-Lip-on-UR}
	\int_B \int_0^{r_B} \Omega_f^q(x,r)^2 \, \frac{dr}{r} d \dH^d|_E(x) \lesssim r_B^d,
\end{equation}
where $\Omega_f^q$ is defined as in \eqref{e:def-Omega-infty} except that the integral is with respect to $\dH^d|_E$. To summarise, we have the following table: 
\setlength{\tabcolsep}{10pt}
\renewcommand{\arraystretch}{1.5}
\begin{center}
	\begin{tabular}{ | m{17em} | m{17em} | } 
		\hline
		Qualitative & Quantitative \\ 
		\hline
		$f$ Lipschitz: $\lim_{r \to 0} \Omega_f^1(x,r)=0$ & $\|\grad f\|_p \approx \|\dG^q\|_p$. A consequence: if $f:\R^d \to \R$ Lipschitz,  Carleson measure estimate \eqref{e:carleson-measure-Omega}. \\ 
		\hline
		$E$ rectifiable: $\lim_{r\to 0} \beta_E^{d,1}(x,r) =0$ & $E$ uniformly rectifible: Carleson measure estimate \eqref{e:carleson-measure-beta}.  \\ 
		\hline
		$E$ rectifiable and $f:E\to \R$ Lipschitz: Rademacher's holds. & $E$ UR and $f:E \to \R$ Lipschitz: Carleson measure estimate \eqref{e:carleson-measure-Lip-on-UR}.\\
		\hline
	\end{tabular}
\end{center}
\vspace{0.5cm}

\noindent
Theorems \ref{t:main-A} and \ref{t:main-B} show that the (Euclidean space) $L^p$ estimates shown  the top right cell also holds on UR sets (i.e. in the bottom right cell).

\subsection{Remarks on the literature}
The result closest to Theorem \ref{t:main-A} is, to the authors' knowledge, the Carleson measure estimate by David and Semmes Propostion \ref{t:DS-Dorronsoro}. As already pointed out, the novelty of our Theorems \ref{t:main-A} and \ref{t:main-B} compared to Proposition \ref{t:DS-Dorronsoro}, is that we have $L^p$ estimates for Sobolev functions. This is of course fundamental in applications, for example Theorem A of \cite{AMV2}.\\

\noindent As far as we know, there are four proofs of the original Dorronsoro's result. 
\begin{itemize}
	\item The original one by Dorronsoro in \cite{dorronsoro1985characterization}, which uses fractional Sobolev spaces and complex interpolation.
	\item A second one might be found in the appendix of \cite{azzam2016bi}, but it was known since the '90s; for example it appears as an exercise in \cite{christ1991lectures}. It works for a smaller range of $q$'s, since it is based on an $L^2$ Fourier calculation.
	\item A third one is by Hyt\"{o}nen, Li, and Naor \cite{hytonen2016quantitative}, who work in the much more general setting of Banach spaces and focus on estimates for Lipschitz functions (no Sobolev spaces involves).
	\item Finally, a fourth one is by Orponen \cite{orponen2021integralgeometric}, where he obtains the Carleson measure estimate \eqref{e:carleson-measure-Omega} via integralgeometric methods. Notably, he can extend his proof to parabolic spaces. Dorronsoro's theorem was extended to Heisenberg groups by F\"{a}ssler and Orponen \cite{fassler2020dorronsoro}. 
\end{itemize}
   
\noindent
Dorronsoro's theorem is not the only quantification of Rademacher's one may think of. Of course, a standard reference on differentiability properties of functions, and quantifications thereof, is the book by Stein \cite{stein2016singular}. A more recent result is that of Alabern, Mateu, and Verdera \cite{alabern2012new}. They essentially prove Theorem \ref{t:dorronsoro-classic} using
\begin{equation*}
	C_f(x,r):=\left| \avint_{B(x,r)} \frac{f(x)-f(y)}{r} \, dy\right|
\end{equation*} 
instead of $\Omega_f$ numbers.
These coefficients measure the symmetry properties of $f$ and the cancellations happening around $f(x)$. In fact, it is important that the absolute value remains \textit{outside} the integral (a difference to Dorronsoro's coefficients). The result in \cite{alabern2012new} was recently proved on the sphere by Barcel\'{o}, Luque, and Per\'{e}z-Esteva \cite{barcelo2020characterization}. It would be of interest to prove versions of Theorem \ref{t:main-A} and \ref{t:main-B} with variants of the $C_f$ coefficients, since they are perhaps more natural quantities to consider in settings where there is no linear structure. To this end, the role played by $\alpha$ and $\beta$ number in the definition of $\gamma_f$ will be taken by the center of mass, see \cite{villa2019square}. Another open problem is whether a converse of Proposition \ref{t:DS-Dorronsoro} is true. It is known for one-dimensional sets, see \cite{david-semmes93}. \\

\noindent
Finally, let us mention that variants of the $\gamma^q_f$ coefficients, inspired by this work, recently appeared in a paper by the second author and Tolsa \cite{mourgoglou2021regularity} on the $L^p$-regularity problem for the Laplacian (resolving an old conjecture by C. Kenig).
An $L^\infty$ variant of the coefficients $\gamma^q_f$ has also been used recently in a very interesting upcoming work of Bate, Hyde and Schul in the context of metric spaces, see \cite{BHS23}.\\


\subsection{Overview of the proof}
A few words about the proof of Theorem \ref{t:main-A}. We first reduce matters to a good-$\lambda$ inequality (Lemma \ref{l:main-I}): for all $\alpha>0$ there is an $\ve>0$ so that for all $\lambda>0$ we have
\begin{equation}\label{e:good-lambda-intro}
\big|\big\{x \in E \, |\, \dG^q f(x) > \alpha\lambda ,\; \dM (\grad_H f)(x) \leq \ve \lambda\big\}\big|
< 0.9 \big|\big\{ x\in Q_0 \, |\, \dG^q f(x) > \lambda\big\}\big|.
\end{equation}
Here $\dM (\grad_H f)$ is some maximal function of the Haj\l{}asz upper gradient. Theorem \ref{t:main-A} follows almost immediately. To show \eqref{e:good-lambda-intro}, we define $E_0= \{\dM (\grad_H f) \leq \ve \lambda\}$, and note that we may assume that $|E_0|>0.5 |E|$, for otherwise there is nothing to prove. Using the fact that $\grad_H f$ is the Haj\l{}asz upper gradient, we conclude that $f$ is approximately $ \ve \lambda$-Lipschitz on $E_0$. We extend $f$ to $F$ over all $E$ with the same Lipschitz constant. We now see, using David and Semmes' Proposition \ref{t:DS-Dorronsoro}, that the square function of $F$ is small. With some rather delicate estimates we transfer this to $f$; so $\dG^q f$ is small on $E_0$, which has large measure. Since the left hand side of \eqref{e:good-lambda-intro} is a subset of the complement of $E_0$, we conclude.\\

\noindent
The proof of Theorem \ref{t:main-B} also goes through a good-$\lambda$ inequality, but it's more technical and involves a stopping time algorithm. Again, we want to show that 
\begin{equation}\label{e:lambda-ineq-B}
  \big|\big\{x \in E \, |\, \dM (\grad_t f) \geq \alpha\lambda ,\; \dG^q f (x) \leq \ve \lambda\big\}\big|
  < 0.9 \big|\big\{ x\in Q_0 \, |\, \grad_t f(x) |> \lambda\big\}\big|.
\end{equation}
Consider a maximal cube $R$ of the set $\{|\grad_t f|> \lambda\}$. We define a greedy algorithm, where we stop whenever we meet a cube $Q$ for which the best approximating affine function $A_Q$ in $\Omega_f^q(B_Q)$ has `bad gradient', meaning that $|\grad A_Q| \gtrsim \alpha \lambda$. If we call $\Stop(R)$ the family of the stopped cubes, then it suffices to show that $\sum_{Q \in \Stop(R)} |Q| \leq \tfrac{1}{2} |R|$: indeed, the set in the left hand side of \eqref{e:lambda-ineq-B} is contained in the union of $\Stop(R)$ over all maximal cubes $R$'s. Showing the packing condition is done by building an approximating Lipschitz function $F$ at the level of $\Stop(R)$ which has small Lipschitz constant and small square function $\dG^q F$. This construction is similar to that of David-Semmes-L\`{e}ger.

\subsection{Structure of the paper}
The paper is structured as follows: Sections \ref{s:sobolev}, \ref{s:cubes} and \ref{s:alpha} contain the preliminaries on Sobolev Space and Poincar\'{e} inequalities, Christ-David cubes and the various coefficients used, respectively. The remainder of the paper is split into two parts. Part 1 \ref{p:Part1} contains the proof of Theorem \ref{t:main-A}. In Section \ref{s:tA-good-lambda} we show how to prove it via a the good-$\lambda$ inequality mentioned above, in Section \ref{s:tA-proof-lambda} we prove the good-$\lambda$ inequality via a square function estimate and finally in Section \ref{s:tA-square-function} we prove this estimate. Part \ref{p:part2} is devoted to the proof of Theorem \ref{t:main-B}. Section \ref{s:detour} contains some preliminaries on tangential gradients. Section \ref{s:tB-proof-via-good} we prove Theorem \ref{t:main-B} via the good-$\lambda$ estimate mentioned above. Section \ref{s:tB-stopping-time} we define our stopping time procedure. In Section \ref{s:tB-approx-Lipschitz} we construct the approximating Lipschitz graph. In Section \ref{s:tB-proof-main-lemma-via-square} we prove the packing condition on `bad gradient' cubes via a square function estimate and finally in Section \ref{s:tB-proofLemmaOmegaF-sim} we prove the square function estimate. 

\subsection*{Acknowledgments} In a first draft of this paper, Theorem \ref{t:main-B} was proved only for $\widetilde{\dG} f$. We thank X. Tolsa for suggesting that the current version might be possible. We also thank M. Hyde for suggestions which improved the readability of the manuscript.

\noindent
 Let us survey some recent literature, (but mind that we will just skim the surface of a very broad and well studied area). In fact, we will mostly focus on the literature from the `UR world'.\\

\noindent
Motivated by the corona problem in higher dimension, Varopoulos \cite{varopoulos1977bmo, varopoulos1978remark} proved that ${\rm BMO}(\R^d)$ can be characterised by the fact that each $f \in {\rm BMO}$ in this space can be extended to a function $F$ on $\R^{d+1}_+=\R^d \times \R_+$, so that $|\grad F| dt dx$ is a Carleson measure. A main tool in Varopoulos argument was an \textit{$\ve$-approximability} result, stating that a bounded analytic function in the upper half plane can be $\ve$-approximated by a $C^\infty$ function whose gradient defines a Carleson measure (see also Theorem 6.1, Chapter VIII in \cite{garnett2007bounded}). If we fast forward forty years, we find out that, in fact, the $\ve$-approximability of bounded harmonic functions actually characterise corkscrew domains with UR boundary \cite{hofmann2016uniform, garnett2018uniform}. In 2018, Hyt\"{o}nen and Ros\`{e}n introduced an $L^p$ version of Varopoulos' $\ve$-approximability: they showed that \textit{any} weak solution to certain elliptic partial differential equations on $\R^{d+1}_+$ is $\ve$-approximable in their $L^p$ sense (\cite[Theorem 1.3]{hytonen2018bounded}) - Varopoulos notion concerned \textit{harmonic} functions. They show the same for dyadic martingales (see \cite[Theorem 1.2]{hytonen2018bounded}). Shortly after, it was shown that the $L^p$ notion of $\ve$-approximability (of harmonic functions) characterises corkscrew domains with UR boundary (\cite{hofmann2020uniform, bortz2019approx}, see also \cite{hofmann2021uniform} and \cite{mourgoglou2023varopoulos}). Back to $\R^{d+1}_+$, Hyt\"{o}nen and Ros\`{e}n used their $\ve$-approximability to construct a bounded and surjective trace map onto $L^p(\R^d)$ from a space of functions $u$ of locally bounded variation on the half space $\R^{d+1}_+$, so that $\|C (\grad u)\|_p$ and $\|\mathcal{N}u\|_p$ are finite.  Here $C$ is the Carleson functional
\begin{equation*}
	C\mu(x) := \sup_{Q \ni x} \frac{1}{|Q|}\int_{\widehat{Q}} d |\mu| (x,t),
\end{equation*}
the supremum is over dyadic cubes in $\R^d$ and $|\mu|$ is a locally finite measure on $\R^{d+1}_+$; $\mathcal{N}$ denotes the non-tangential maximal function. 

Finally, to our corollary. There we show that the trace map is surjective onto the \textit{Sobolev space} $M^{1,p}(\pom)$, from the space of functions $u$ on $\om$ so that $\|\mathcal{N}(\grad u)\|_p$ and the non-tangential square function of the Hessian of $u$ are finite. Note that we don't work in $\R^{d+1}_+$ but, rather, in the more general case of a corkscrew domain with UR boundary. A similar extension was constructed by the second author and Tolsa in \cite{mourgoglou2021regularity} to solve the regularity problem for the Laplacian - however only for $p$ close to one. Remark also that the extension in \cite{mourgoglou2021regularity} was in fact borrowed from the current work. 
 
\section{Notation}\label{s:notation}
We write $a \lesssim b$ if there exists a constant $C$ such that $a \leq Cb$. By $a \sim b$ we mean $a \lesssim b \lesssim a$.
In general, we will use $n\in \mathbb{N}$ to denote the dimension of the ambient space $\R^n$, while we will use $d \in \mathbb{N}$, with $d\leq n-1$, to denote the dimension of a subset $E \subset \R^n$.
For two subsets $A,B \subset \R^n$, we let
$
    \dist(A,B) := \inf_{a\in A, b \in B} |a-b|.
$
For a point $x \in \R^n$ and a subset $A \subset \R^n$, 
$
    \dist(x, A):= \dist(\{x\}, A)= \inf_{a\in A} |x-a|.
$
We write 
$
    B(x, r) := \{y \in \R^n \, |\,|x-y|<r\},
$
and, for $\lambda >0$,
$
    \lambda B(x,r):= B(x, \lambda r).
$
At times, we may write $\mathbb{B}$ to denote $B(0,1)$. When necessary we write $B_n(x,r)$ to distinguish a ball in $\R^n$ from one in $\R^d$, which we may denote by $B_d(x, r)$. 
We denote by $\dG(n,d)$ the Grassmannian, that is, the manifold of all $d$-dimensional linear subspaces of $\R^n$. A ball in $\dG(n,d)$ is defined with respect to the standard metric
\begin{align*}
	d_{\dG}(V, W) = \|\pi_V - \pi_W\|_{{\rm op}}.
\end{align*}
Recall that $\pi_V: \R^n \to V$ is the standard orthogonal projection onto $V$.
With $\dA(n,d)$ we denote the affine Grassmannian, the manifold of all affine $d$-planes in $\R^n$. 
 \noindent
The set of all affine maps $A: \mathbb{R}^n \to \R$ will be denoted as $\dM(n,1)$. Finally, $\dH^d|_E$ denotes the $d$-dimensional Hausdorff measure restricted to $E \subset \R^n$.

\section{Preliminaries: Sobolev spaces and Poincar\'{e} inequalities}\label{s:sobolev}
We use this section to mention the results we will need about Sobolev spaces in metric setting and Poincar\'{e} inequalities. 
\begin{definition}\label{def-Hajlasz-grad}
Let $(X,\mu)$ be a metric measure space. For $1\leq p <\infty$, we let $M^{1,p}(X)$ the set of functions $u \in L^{p}(X)$ for which there exists a $g \in L^p(X)$ so that
\begin{align}\label{e:upper-gradients}
	|u(x)-u(y)| \leq |x-y| (g(x)+g(y)) \mbox{ for } \mu\mbox{-a.e.} \, x,y \in X.
\end{align}
For $f \in L^{p}(X)$, denote by $\Grad_p(f)$ the set of $L^p(X)$ functions $g$ which satisfy \eqref{e:upper-gradients}. We also denote by $|\grad_H f|$ the function $\in \Grad_p(f)$ so that 
\begin{equation}\label{e:Hajlasz-grad}
	\|\grad_H f\|_{L^p(X)} = \inf_{g \in \Grad_p(f)} \|g\|_{L^p(X)}.
\end{equation}
We call $\grad_H f$ the Haj\l{}asz gradient. If $g \in \Grad_p(f)$, we will refer to it as a \textit{Haj\l{}asz upper gradient}.
\end{definition}
\noindent We refer the reader to \cite[Section 5.4]{heinonen2005lectures} for an introduction to Haj\l{}asz-Sobolev spaces. A very useful fact about $M^{1,p}(X)$ is that pairs $(f, g)$, where $f \in M^{1,p}(X)$ and $g \in \Grad_p(X)$, always admit a Poincar\'{e} inequality.
\begin{proposition}\label{t:Hajlasz-poincare}
	Let $(X,\mu)$ be a metric measure space. Let $1\leq p <\infty$, $f \in M^{1,p}(X)$ and $g \in \Grad_p(f)$. Then for each $1\leq p' \leq p$, 
	\begin{equation}
		\left(\avint_B |f-f_B|^{p'} d\mu \right)^{\frac{1}{p'}} \leq 2 r_B \left( \avint_B g^{p'} d\mu \right)^{\frac{1}{p'}}.
	\end{equation}
\end{proposition}
\noindent
See \cite[Theorem 5.15]{heinonen2005lectures} or \cite[Proposition 2.1]{mourgoglou2021regularity} for a proof.\\

\noindent
Haj\l{}asz upper gradients should not be confused with what are commonly referred to simply as \textit{upper gradients}.
\begin{definition}\label{d:upper-grad}
 Given a metric measure space $X$ and a function $f: X \to \R$ measurable, we say that $\rho: X \to [0, \infty]$ is an upper gradient of $f$ if, for $x,y\in X$,
$
|u(x)-u(y)|\leq \int_{\gamma} \rho
$
for any rectifiable curve $\gamma$ connecting $x$ to $y$ in $X$.
\end{definition}
\noindent
Now, if the space $X$ is so that a Poincar\'{e} holds for $f$ and \textit{all of its upper gradients} (something that comes for free when using Haj\l{}asz upper gradients), then we say that $X$ admits a Poincar\'{e} inequality. More precisely:
\begin{definition}\label{def-poincare}
	For $p\geq 1$, a metric measure space $(X, d,\mu)$ admits a  {\it weak $(1,p)$-Poincar\'{e} inequality} for all measurable functions $f$ with constants $C_1,\Lambda \geq 1$ if $\mu$ is locally finite and
	\begin{equation}
		\label{e:poincare}
		\avint_{B} |f-f_{B}|d\mu \leq C_1 r_B \ps{\avint_{\Lambda B} \rho^{p}d\mu}^{\frac{1}{p}}
	\end{equation}
	where $\rho$ is any upper gradient for $f$. 
\end{definition}
\noindent
Spaces supporting a weak Poincar\'{e} inequality enjoy quantitative connectivity properties in the sense that subsets of $X$ which are disjoint continua are connected by quantitatively many rectifiable curves. See \cite{heinonen1998quasiconformal}. In general, these spaces can be geometrically quite irregular and lack any (Euclidean) rectifiable structure. The Heisenberg group is one such standard example. When they are Ahlfors regular subsets of Euclidean space, however, we have the following result, due to the first author. 
\begin{theorem}[\cite{azzam2021poincare}]
	Let $n>d\geq 2$ be integers and $X\subseteq \R^{n}$ be an Ahlfors $d$-regular set with constant $c_0\geq 1$ supporting a weak $(1,d$)-Poincar\'{e} inequality with respect to $\dH^{d}|_{X}$ with constants $C_1,\Lambda \geq 1$. Then $X$ is uniformly $d$-rectifiable \textup{(}with constants depending on $C_1$ and $\Lambda$\textup{)}.
\end{theorem}

\noindent
 The following theorem says that, if $X$ admits a Poincar\'{e} inequality in the sense that we just described, then the notion of upper gradients and Haj\l{}asz upper gradient essentially coincide. 
 \begin{theorem}\label{t:max-funct-est}
 Suppose $(X,\mu)$ is a locally compact doubling space admitting a weak $(1,p)$-Poincar\'{e} inequality with constant $\Lambda$. Then for every $u\in L^{p}(X,\mu)$, if $g$ is an upper gradient for $u$, then for almost every $x,y\in X$,
 \begin{equation}
 \label{e:maximal-lip}
 |u(x)-u(y)| \lec |x-y| ((\M_{\Lambda |x-y|}g(x)^{p})^{\frac{1}{p}}+\M_{\Lambda |x-y|}g(y)^{p})^{\frac{1}{p}})
 \end{equation}
 \end{theorem}
\noindent
 This follows from \eqref{e:poincare} and \cite[Theorem 3.2]{hajlasz2000sobolev}. One other remark is that, if $X$ is a $d$-rectifiable subset of $\R^n$, then the \textit{tangential gradient}, which will be defined later in Section {\ref{s:detour} is an upper gradient.\\
 	
 \noindent
 Recall that a metric measure space $X$ is doubling if for any ball $B$, $\mu(2B) \leq C \mu(B)$, where $C$ depends on the metric. Given $s>0$, consider the following condition:
 \begin{equation}\label{e:condition-s}
     \frac{\mu(B)}{\mu(B_0)} \geq C \left(\frac{r_B}{r_{B_0}} \right)^s.
 \end{equation}
 where the center of $B$ is in $B_0$, and $r_B \leq r_{B_0}$. 
 
 \begin{theorem}\label{t:sobolevmetpoincare}
 Let $(X, d,\mu)$ be a doubling metric measure space such that $\mu$ satisfies \eqref{e:condition-s} for some $s>0$. Assume that the pair $f,g$ satisfies a $p$-Poincar\'{e} inequality, $p>0$. 
 \begin{itemize}
     \item If $p<s$, then for all $0<q<\frac{ps}{s-p}$, we have
     \begin{align}\label{e:HK00-1}
    \left(\avint_B \left( \frac{|f-f_B|}{r_B}\right)^{q}d \mu \right)^{\frac{1}{q}} \leq C \left( \avint_{5 \Lambda B} g^p \, d\mu \right)^{\frac{1}{p}}
\end{align}
Moreover, for any $p<q<s$, we have
\begin{equation}
    \ps{ \avint_B\left(\frac{|f-f_B|}{r(B)}\right)^{q^*(s)} \, d\mu }^{\tfrac{1}{q^*(s)}} \leq C \ps{\avint_{5\Lambda B} g^q \, d\mu}^{\tfrac{1}{q}},
\end{equation}
where $q^*(s)=\frac{sq}{s-q}$.
\item If $p=s$, then 
\begin{equation}\label{e:HK00-p=s}
    \avint_B \exp\left\{ \frac{C_1 |f(x)-f_B|}{r_B \left(\avint_{5\Lambda B} g^s \right)^{\frac{1}{s}}} \right\}\, d \mu(x) \leq C_2.
\end{equation}
\item If $s>d$, then 
\begin{equation}\label{e:HK-p>s}
    \esssup_{x \in B}\frac{|f(x)-f_B|}{r_B} \leq C \left(\avint_{5\Lambda B} g^p \, d\mu\right)^{\frac{1}{p}}. 
\end{equation}
\end{itemize}
The constants $C, \,\Lambda \geq 1$ depends on $s$, the doubling constant, $p$ and $q$.

\end{theorem}

\begin{remark}\label{r:sobolevmetpoincare}
  Of course, \eqref{e:HK00-p=s} implies that
  \begin{equation*}
      \left(\avint_{B} \left(\frac{|f-f_B|}{r_B}\right)^q \, d\mu \right)^{\frac{1}{q}} \leq C \left(\avint g^p \, d\mu\right)^{\frac{1}{p}}
  \end{equation*}
  for any $1\leq q < \infty$. 
\end{remark}

\section{Preliminaries: dyadic lattices}\label{s:cubes}

Given an Ahlfors $d$-regular measure $\mu$ in $\R^{n}$, we consider 
the dyadic lattice of ``cubes'' built by David and Semmes in \cite[Chapter 3 of Part I]{david-semmes93}. The properties satisfied by $\DD_\mu$ are the following. 
Assume first, for simplicity, that $\diam(\supp\mu)=\infty$). Then for each $j\in\mathbb{Z}$ there exists a family $\DD_{\mu,j}$ of Borel subsets of $\supp\mu$ (the dyadic cubes of the $j$-th generation) such that:
\begin{itemize}
\item[$(a)$] each $\DD_{\mu,j}$ is a partition of $\supp\mu$, i.e.\ $\supp\mu=\bigcup_{Q\in \DD_{\mu,j}} Q$ and $Q\cap Q'=\varnothing$ whenever $Q,Q'\in\DD_{\mu,j}$ and
$Q\neq Q'$;
\item[$(b)$] if $Q\in\DD_{\mu,j}$ and $Q'\in\DD_{\mu,k}$ with $k\leq j$, then either $Q\subset Q'$ or $Q\cap Q'=\varnothing$;
\item[$(c)$] for all $j\in\mathbb{Z}$ and $Q\in\DD_{\mu,j}$, we have $2^{-j}\lesssim\diam(Q)\leq2^{-j}$ and $\mu(Q)\approx 2^{-jd}$;
\item[$(d)$] there exists $C>0$ such that, for all $j\in\mathbb{Z}$, $Q\in\DD_{\mu,j}$, and $0<\tau<1$,
\begin{equation}\label{small boundary condition}
\begin{split}
\mu\big(\{x\in Q:\, &\dist(x,\supp\mu\setminus Q)\leq\tau2^{-j}\}\big)\\&+\mu\big(\{x\in \supp\mu\setminus Q:\, \dist(x,Q)\leq\tau2^{-j}\}\big)\leq C\tau^{1/C}2^{-jd}.
\end{split}
\end{equation}
This property is usually called the {\em small boundaries condition}.
From (\ref{small boundary condition}), it follows that there is a point $x_Q\in Q$ (the center of $Q$) such that $\dist(x_Q,\supp\mu\setminus Q)\gtrsim 2^{-j}$ (see \cite[Lemma 3.5 of Part I]{david-semmes93}).
\end{itemize}
We set $\DD_\mu:=\bigcup_{j\in\mathbb{Z}}\DD_{\mu,j}$. \\

\noindent
In case that $\diam(\supp\mu)<\infty$, the families $\DD_{\mu,j}$ are only defined for $j\geq j_0$, with
$2^{-j_0}\approx \diam(\supp\mu)$, and the same properties above hold for $\DD_\mu:=\bigcup_{j\geq j_0}\DD_{\mu,j}$.
Given a cube $Q\in\DD_{\mu,j}$, we say that its side length is $2^{-j}$, and we denote it by $\ell(Q)$. Notice that $\diam(Q)\leq\ell(Q)$. 
We also denote 
\begin{equation}\label{defbq}
B(Q):=B(x_Q,c_1\ell(Q)),\qquad B_Q = B(x_Q,\ell(Q)),
\end{equation}
where $c_1>0$ is some fix constant so that $B(Q)\cap\supp\mu\subset Q$, for all $Q\in\DD_\mu$. Clearly, we have $Q\subset B_Q$.
For $\lambda>1$, we write
$$\lambda Q = \bigl\{x\in \supp\mu:\, \dist(x,Q)\leq (\lambda-1)\,\ell(Q)\bigr\}.$$

\noindent
The side length of a ``true cube'' $P\subset\R^{n}$ is also denoted by $\ell(P)$. On the other hand, given a ball $B\subset\R^{n}$, its radius is denoted by $r_B$ or $r(B)$. For $\lambda>0$, the ball $\lambda B$ is the ball concentric with $B$ with radius $\lambda\,r(B)$.

\vvv

\section{Preliminaries: uniform rectifiability; the $\alpha$, $\beta$, $\gamma$ and $\Omega$ coefficients}\label{s:alpha}
We gather in this section some basic about the various coefficients that will appear in the proofs below. We assume throughout that $E\in \R^{n}$ is an Ahlfors $d$-regular set and that $\sigma=\dH^{d}|_{E}$. 

\subsection{Ahlfors regularity; uniform rectifiability}
Recall that we introduced Ahlfors $d$-regularity in \eqref{e:ADR}. Since $d$ is fixed throughout the paper, we will often simply say Ahlfors regular or AR. The following fact about Ahlfors regular sets will come handy over and over again. 
\begin{lemma}\label{l:balanced-balls}
	Let $1\leq d\leq n-1$ and $E \subset \R^n$ be Ahlfors $d$-regular. There is a constant $0<c<1$, depending on the AR constant, so that for any ball $B$ centered on $E$ we can find balls $B_{0},...,B_{d}$ centered on $E$ and with radii $cr_{B}$, so that $2B_{i}\subseteq B$ and
	\begin{equation}
		\label{e:d-dim-points-1}
		\dist(x_{B_{i+1}},\, {\rm span}\{x_{B_{1}},...,x_{B_{i}}\})\geq 4cr_{B}.
	\end{equation}
\end{lemma}
\noindent
This is a standard fact. See \cite{david-semmes91}, Lemma 5.8.\\

\noindent
We briefly recalled the definition of uniform rectifiability in the introduction. Let us be more precise here. 
\begin{definition}
  We say that an Ahlfors $d$-regular set $E \subset \R^n$ is uniformly $d$-rectifiable if it contains "big pieces of Lipschitz images" (BPLI) of $\R^d$. That is to say, if there exist constants $\theta, L>0$ so that for every $x \in E$, and $0<r<\diam(E)$, there is a Lipschitz map $\rho: \R^d \to \R^n$ (depending on $x,r$), with Lipschitz constant $\leq L$, such that
  \begin{equation*}
      \dH^d\left(E \cap B(x,r) \cap \rho(B(0, r))\right) \geq \theta r^d.
  \end{equation*}
\end{definition}
We might often simply say uniformly rectifiable or UR sets. There is a well developed theory of uniformly rectifiable sets. We refer the interested reader to the original monographs \cite{david-semmes91} and \cite{david-semmes93}.

\subsection{The geometric coefficients $\alpha$ and $\beta$}\label{sec6}

\subsubsection{Tolsa's $\alpha$ numbers}
We first define Tolsa's $\alpha$ numbers. They first appeared in the area in \cite{tolsa2009uniform} in connection to singular integral operators, and have been heavily used since then. The $\alpha$ quantify the closeness of a Radon measure $\mu$ to a multiple of $d$-dimensional Hausdorff measure on some plane. 
Let $\mu$ and $\nu$ be Radon measures. For an open ball $B$ define
\begin{equation*}
F_B(\sigma,\nu):= \sup\Bigl\{ \Bigl|{\textstyle \int \phi \,d\sigma  -
	\int \phi\,d\nu}\Bigr|:\, \phi\in \Lip(B) \Bigr\},
\end{equation*}
where 
\begin{equation*}
    \Lip(B)=\{\phi:{\rm Lip}(\phi) \leq1,\,\supp f\subseteq
B\}
\end{equation*}
and $\Lip(\phi)$ stands for the Lipschitz constant of $\phi$. See \cite[Chapter 14]{Mattila} for the properties of this distance. Next, set
\begin{gather}
\alpha_\sigma^{d}(B,P) := \frac1{r_{B}\,\sigma(2B)}\,\inf_{c\geq0} \,F_{2B}(\sigma,\,c\dH^{d}|_{P}),\label{e:alpha-def0}\\
 \alpha_{\sigma}^{d}(B) = \inf_{P\in \dA(d,n)}\alpha_\sigma^{d}(B,P)\label{e:alpha-def}.
\end{gather}
Note that the right hand side of \eqref{e:alpha-def} is computed over $2B$ (rather than $B$). This is simply for notational convenience. 

\begin{remark}
	We denote by $c_B$ and $P_B$ a constant and a plane that infimise $\alpha_\sigma(B)$. That is, we let $c_B>0$ and $P_B \in \dA(n,d)$ be such that, if we set
	\begin{equation}\label{e:def-L}
		\mathcal{L}_{B} :=c_B\dH^{d}|_{P_B}, 
	\end{equation}
	then
	\begin{equation}\label{cL-ppts}
		\alpha_\sigma^{d}(B) =  
		\alpha_{\sigma}^{d}(B,\mathcal{L}_{B})=\frac1{r_{B}^{d+1}}\,F_{2B}(\sigma,\,\mathcal{L}_B)
	\end{equation}
\end{remark}
\noindent
We will need the following properties of these coefficients. 
\begin{lemma} \cite[Lemma 3.1]{tolsa2009uniform}
	\label{l:alpha-lemmas}
	For any ball $B\subset \R^n$, 
	\begin{enumerate}[\textup{(}a\textup{)}]
		\item $\alpha_{\sigma}^{d}(B)\lec 1$,
		\item If $B\subseteq B'$ and $r_{B}\approx_c r_{B'}$, then $\alpha_{\sigma}^{d}(B)\lec_c \alpha_{\sigma}^{d}(B')$. 
		\item $c_{B}\approx 1$.
	\end{enumerate}
\end{lemma}
\noindent
We recall a characterisation of uniform rectifiable sets by Tolsa. 
\begin{theorem}
	\cite[Theorem 1.2]{tolsa2009uniform}
	An Ahlfors $d$-regular set $\Sigma\subseteq \R^{n}$ is UR if and only if  for all $R\in \dD(E)$ we have 
	\[
	\sum_{Q\subseteq R}\salpha(Q)^{2}\ell(Q)^d \lec \ell(R)^d.
	\]
	\label{t:alpha-UR}
\end{theorem}

\subsubsection{Jones' $\beta$ numbers}
The second quantity we introduce are the well-known Jones' $\beta$ numbers. For a ball $B$ centered on $\Sigma$, a $d$-plane $P \in \dA(n,d)$, and $p>0$, put
\begin{equation*}
\beta_{\sigma}^{d,p}(B,P)  =  \left(\frac{1}{r_B^{d}} \int_{B}\ps{\frac{\dist(y,L)}{r_B}}^{p} d\sigma(y)\right)^{\frac{1}{p}}.
\end{equation*}
The Jones' $\beta$-number of $\Sigma$ in the ball $B$ is defined as the infimum over all $d$-affine planes $P \in \dA(n,d)$:
\begin{equation*}
\beta_{\sigma}^{d,p}(B) =\inf_{P\in \dA(d,n)} \beta_{\sigma}^{d,p}(B,P).
\end{equation*}

\begin{remark}[Infimising planes I]\label{r:pi_B}
	In some situations, we will be dealing with planes $P_B$ which infimise certain coefficients (e.g. $\beta_E^{d,q}(B)$) in a ball $B$. Then we call $P_B$ one such plane and denote by $\pi_B$ the orthogonal projection onto $P_B$.
\end{remark}

\begin{remark}[Infimising planes II]\label{r:P_B}

  We adopt the following convention. Below, it will often happen that, while working with with $p$ in the range $[1, \infty]$, we will use $\alpha_\sigma$ for $p=1$, and $\beta_\Sigma^{d,p}$ for $p>1$. In this situation, we will abuse notation and also let $P_{B}$ be the $d$-plane that infimises $\beta_{\Sigma}^{d,p}(B,P)$.  Then, given $P_B$, whether this infimises $\beta_{\Sigma}^{d,p}$ or $\salpha$ will be clear from context, that is, in a theorem that is stated for $p\geq 1$, in the proof we will assume any $P_{B}$ that appears is defined for $\salpha$ if $p=1$ and for $\beta_{\Sigma}^{d,p}$ if $p>1$.
\end{remark}

\noindent
In the lemma below we gather some basic properties of $\beta$ numbers which will be used throughout the paper.

\begin{lemma}\label{l:beta-properties}\
		
	\begin{enumerate}
		\item (Well known).
Suppose that $B, B'$ are two balls centered on $\Sigma$ such that $B \subset B'$ and $r(B) \approx_c r(B')$. Then, if $P \in \dA(n,d)$,
\begin{equation}
\beta_{\sigma}^{d,p}(B,P)\lec_c \beta_{\sigma}^{d,p}(B',P).
\end{equation} 
In particular, 
\begin{equation}
\label{e:beta-monotone}
\beta_{\sigma}^{d,p}(B)\lec_c \beta_{\sigma}^{d,p}(B').
\end{equation}

\item \textup{(}Equation 5.4, \cite{david-semmes91}\textup{)}.
For any ball $B$ centered on $\Sigma$, 
 \begin{equation}
\label{e:beta<beta}
\beta_{\Sigma}^{d,\infty}\left(\tfrac{1}{2}B,P_B\right)^{d+1}
\lec \beta_{\Sigma}^{d,1}(B,P_B)
\end{equation}
The constant in $\lesssim$ is dimensional.
\item (Well known). For any ball $B$, we have
\begin{align}\label{e:beta1<betap}
	\beta_E^{d,1}(B) \lesssim \beta_E^{d,q}(B) \,\,\mbox{ for any } \,\, q \geq 1.
\end{align}
\end{enumerate}
\end{lemma}

\begin{proposition}[{\cite[Condition C]{david-semmes91}}]\label{p:strong-geometric-lemma}
	A $d$-Ahlfors regular set $E \subset \R^n$ is UR if and only if
	\begin{equation*}
			\sum_{Q \subset R} \beta_{E}^{d,q}(Q)^2 \ell(Q)^d \lesssim \ell(R)^d.
	\end{equation*}
	Here $1\leq q \leq \infty$ when $d=1$ and $1\leq q < 2^*$ when $d\geq 2$. 
\end{proposition}

 \subsubsection{Relation between $\beta$, $\alpha$ and angle between planes}
 We can relate $\alpha_\sigma$ and $\beta_E^1$ via the following lemma.
\begin{lemma}\cite[Lemma 3.2]{tolsa2009uniform}
For any ball $B$ centered on $\Sigma$ \textup{(}and recalling how we defined $\salpha(B)$\textup{)},
\begin{equation}
\label{e:beta<alpha}
\beta_{\Sigma}^{d,1}(B,P_B)
\lec \alpha_{\sigma}^{d}(B,\mathcal{L}_B)= \alpha_{\sigma}^{d}(B) .
\end{equation}
\end{lemma}
\noindent
The following lemma is originally stated in more generality than this, but it is implied by the original. It says that the $\beta$ number control the angles between best approxiamting planes at different scales. Recall that $\pi_{B}$ denote the orthogonal projection onto $P_{B}$.
\begin{lemma}\cite[Lemma 2.16]{azzam2018analyst} Suppose $\Sigma$ is Ahlfors $d$-regular and $B\subseteq B'$ are centered on $\Sigma$ with $r_{B}\approx r_{B'}$. Then
\begin{equation}
\label{e:PBPB'}
d(P_{B}\cap B',P_{B'}\cap B')\lec \beta_{\Sigma}^{d,1}(B')r_{B'}
\end{equation}
where $d(\cdot,\cdot)$ denotes Hausdorff distance. In particular, if $P_{B}'$ and $P_{B'}'$ denote the planes passing through the origin parallel to $P_{B}$ and $P_{B'}$ respectively, then
\begin{equation}
\label{e:beta-angle}
\angle(P_{B},P_{B'})
:= \angle(P_{B}',P_{B'}')
:=||\pi_{{B}}-\pi_{{B'}}||\lec  \beta_{\Sigma}^{d,1}(B')r_{B'}.
\end{equation}
\end{lemma}
\noindent

\subsection{The coefficients $\Omega$ and $\gamma$} In this subsection we introduce the quantities relevant to Dorronsoro's estimates.
Let $q \geq 1$ and consider a function $f: \Sigma \to \R$ so that $f\in L^{q}(\Sigma)$. For a each ball $B$ centered on $\Sigma$, and an affine map $A : \R^n \to \R$, let 
\begin{equation}\label{e:def-omega}
\Omega_{f}^{q}(B,A)=\ps{ \avint_{B}\ps{\frac{|f-A|}{r_B}}^{q}d\sigma}^{\frac{1}{q}}
\quad \mbox{ and } \quad  \Omega_{f}^{q}(B)=\inf_{A \in \dM(n, 1)}\Omega_{f}^{q}(B,A).
\end{equation}
\noindent
In the following lemma we gather some basic properties of the $\Omega_f$ numbers. 
\begin{lemma}\
	
	\begin{enumerate}
	\item (Monotonicity). If $B\subseteq B' \subset \Sigma$, and $r_{B'} \approx_c r_{B}$, for $q\geq 1$,
\begin{equation}
\label{e:omega-monotone}
\Omega_{f}^{1}(B)\leq \Omega_{f}^{q}(B)\lec_c \Omega_{f}^{q}(B').
\end{equation}
	\item (Affine invariance). Let $f \in L^1(\sigma)$, and suppose that $A_0: \R^n \to \R$ is affine. Then
	\begin{align}\label{e:omega-affine-invariant}
		\Omega^1_{f-A_0} (B) = \Omega^1_{f}(B).
	\end{align}
\end{enumerate}
\end{lemma}
\begin{proof}\
	
	\begin{enumerate}
		\item The proof of this is an easy exercise with Jensen's inequality and we leave the details to the reader.
		\item It suffices to show that if $A \in \dM(n,1)$ infimises $\Omega^q_f(B)$, then $A_1:=A-A_0$ infimises $\Omega_{f-A_0}^q(B)$. Suppose there existed $A_2 \in \dM(n,1)$ with $\Omega_{f-A_0}^q(B, A_2) < \Omega_{f-A_0}^q(B, A_2)$. But then $\Omega_f^q(B, A_0 + A_2) < \Omega_f^q(B, A)$, which is a contradiction. 
	\end{enumerate}
\end{proof}
\noindent

\noindent
We now come to the definition of the quantity $\gamma_f$. Let $f$ be a real valued function defined on $E \subset \R^n$.
\begin{definition}\label{d:gamma}\
\begin{itemize}
    \item For $1\leq q \leq \infty$, $f \in L^q(\Sigma)$ and $A \in \dM(n,1)$, set
\begin{equation}\label{e:def-pgamma0}
\gamma_{f}^{q}(B,A) = \Omega_{f}^{q}(B,A)+|\grad A|\beta_{\Sigma}^{d,q}(B). 
\end{equation}
Then set
\begin{equation}\label{e:def-pgamma}
    \gamma_{f}^{q}(B)=\inf_{A \in \dM(n,1)} \gamma_{f}^{q}(B,A).
\end{equation}
\item If $f \in L^1(\Sigma)$ and $A \in \dM(n,1)$ let 
\begin{equation}\label{e:def-1gamma0}
\widetilde{\gamma}_{f}(B,A)= \Omega_{f}^{1}(B,A)+|\grad A|\salpha(B),
\end{equation}
and then
\begin{equation}\label{e:def-1gamma}
\widetilde{\gamma}_{f}(B)=\inf_{A\in \dM(n,1)} \widetilde\gamma_{f}(B,A).
\end{equation}
\end{itemize}
It is immediate from the definitions that for $q\geq 1$, $\gamma_f^q(B) \geq \Omega_f^q(B)$. Recalling \eqref{e:beta1<betap} and \eqref{e:beta<alpha} we also have that, for any $q \geq 1$,
\begin{align}
	& \gamma^1_f(B)\lesssim \gamma^q_f(B); \label{e:gamma1<gammaq}\\
	& \gamma^1_f(B) \lesssim \widetilde{\gamma}_f (B). \label{e:gamma1<gammatilde}
\end{align}
\end{definition}

\subsection{Less basic facts about the $\gamma$ coefficients}
 Recall that $\pi_B$ is the projection onto $P_B$. Lemma \ref{l:A_B} below essentially says that we can always use a specific map $A_B$ as a minimiser of $\gamma_f^p$. Note that we do not require that $\Sigma$ is UR nor that it supports a Poincar\'{e} inequality.
\begin{lemma}
\label{l:A_B}
Let $\Sigma \subset \R^n$ be a Ahlfors $d$-regular subset, $B$ a ball centered on $\Sigma$ and $q\geq 1$. Let $f \in L^q(\Sigma)$. There is an affine map in $\dM(n,1)$, denoted by $A_{B}$, so that 
\begin{align}
\label{e:ABbeta<gammaAB<gamma}
|\grad A_{B}|\beta_{\Sigma}^{d,q}(B)\lec \gamma_{f}^{q}(B,A_B) 
\lec \gamma_{f}^{q}(B), \\
|\grad A_{B}|\beta_{\Sigma}^{d,1}(B)\lec \widetilde{\gamma}_{f}(B,A_B) 
\lec \widetilde{\gamma}_{f}(B), \label{e:ABbeta<gammatilde}
\end{align}
and
\begin{equation}
\label{e:ABpiAB}
A_{B}\circ\pi_{B}=A_{B}.
\end{equation}
\end{lemma}
\begin{proof}
Fix a $q \geq 1$ and $f \in L^q(\Sigma)$. Let $A \in \dM(n,1)$ such that it attains the infimum in the definition of $\gamma_{f}^{q}(B)$. Define
\begin{equation}\label{e:def-A_B}
A_B(x) := A\circ \pi_{B}(x).
\end{equation}
Here $\pi_B$ is the orthogonal projection onto $P_B$, the infimising plane for $\beta_E^{d,q}(B)$ in \eqref{e:ABbeta<gammaAB<gamma}, or for $\alpha_\sigma^d(B)$ in \eqref{e:ABbeta<gammatilde}.
Since $\pi_B$ is linear, $A_B \in \dM(n,1)$. Note that, by Taylor's theorem, $A(x) = \grad A \cdot x + b$, $b \in \R$, and $A_B(x) = \grad A\cdot \pi_B(x) + b$. So we have that
\begin{align}\label{e:form1}
    |A(x) - A_B(x)| & = |\nabla A \cdot x - \nabla A \cdot \pi_B(x)|\nonumber \\ & \leq |\grad A||x-\pi_{B}(x)|
\leq |\grad A|\dist(x,P_{B}).
\end{align}
Thus using Minkowski's inequality, for $q\geq 1$ we have
\begin{align*}
\Omega_{f}^{q}(B,A_B) & = \ps{\avint_{B}\ps{\frac{f-A_B}{r_{B}}}^{q}\, d \sigma}^{\frac{1}{q}}\\
& \leq \Omega_{f}^{q}(B,A)+\ps{\avint_{B}\ps{\frac{A-A_B}{r_{B}}}^{q} d\sigma}^{\frac{1}{q}}\\
& \stackrel{\eqref{e:form1}}{\leq}  \Omega_{f}^{q}(B,A)+|\grad A|\ps{\avint_{B}\ps{\frac{\dist(x,P_{B})}{r_{B}}}^{q}d \sigma}^{\frac{1}{q}}\\
& \leq  \Omega_{f}^{p}(B,A)+|\grad A|\beta_{\Sigma}^{d,q}(B)
= \gamma_{f}^{q}(B).
\end{align*}
The inequality for $\widetilde{\gamma}_f$ follows by using \eqref{e:beta<alpha} in the last line.
This proves the second inequality in \eqref{e:ABbeta<gammaAB<gamma} and \eqref{e:ABbeta<gammatilde}. The first inequalities are immediate, since $|\grad A_{B}|\leq|\grad A||\grad \pi_B|\leq |\grad A|$ by definition and the fact that $\pi_B$ is $1$-Lipschitz.
\end{proof}
\noindent
The quantity $\gamma_f^p$ also enjoys some quasi-monotonicity properties.
\begin{lemma}
 Let $q \geq 1$ and $f \in L^q(\Sigma)$. If $B\subseteq B'$ are two balls centered on $\Sigma$ such that $r_{B'}\approx_c r_{B}$, then
\begin{equation}
\label{e:gamma-monotone}
\gamma_{f}^{q}(B)\lec_{c} \gamma_{f}^{q}(B').
\end{equation}
\end{lemma}

\begin{proof}
Let $A_B, A_{B'} \in \dM(n,1)$ as in Lemma \ref{l:A_B}, \eqref{e:ABpiAB}. If $q\geq 1$, then 
\begin{align*}
\gamma_f^q(B) & \leq \gamma_{f}^{q}(B,A_{B'})
 =\Omega_{f}^{q}(B,A_{B'})
+|\grad A_{B'}|\beta_{\sigma}^{d,q}(B)\\
& \stackrel{ \eqref{e:beta-monotone}\atop \eqref{e:omega-monotone}}{\lec_c} \Omega_{f}^{q}(B',A_{B'})
+|\grad A_{B'}|\beta_{\sigma}^{d,q}(B')
=\gamma_{f}^{q}(B',A_{B'})\stackrel{\eqref{e:ABbeta<gammaAB<gamma}}{\lesssim_c} \gamma_{f}^{q}(B').
\end{align*}
Similarly, to estimate $\wt{\gamma}_f$, we use Lemma \ref{l:alpha-lemmas}(b) to compute
\begin{align*}
\wtgamma(B) & \leq \wtgamma(B,A_{B'})
=\Omega_{f}^{1}(B,A_{B'})
+|\grad A_{B'}|\salpha(B)\\
& \lec_c \Omega_{f}^{1}(B',A_{B'})
+|\grad A_{B'}|\salpha(B')
=\wtgamma(B',A_{B'})\lesssim_c \wtgamma(B').
\end{align*}

\end{proof}

\begin{lemma}\label{l:gammacf=cgammaf}
Let $q \geq 1$ and $f \in L^q(\Sigma)$. If $B$ is a ball centered on $\Sigma$, and $c>0$, then
\begin{equation*}
\gamma_{cf}^{q}(B)=c\gamma_{f}^{q}(B),
\end{equation*}
and similarly for $\widetilde{\gamma}_f$
\end{lemma}

\begin{proof}
We first show that $\gamma_{cf}^{q}(B)\leq c\gamma_{f}^{q}(B).$ Let $\ve>0$. Suppose $A$ is such that $\gamma_{f}^{q}(B,A)<\gamma_{f}^{q}(B)+\ve$. Then by definition, 
\begin{equation*}
\gamma_{cf}^{q}(B)\leq \gamma_{cf}^{q}(B,cA)
=c\gamma_{f}^{q}(B,A)
<c(\gamma_{f}^{q}(B)+\ve)
\end{equation*}
and letting $\ve\rightarrow 0$ gives $\gamma_{cf}^{q}(B)\leq c\gamma_{f}^{q}(B).$ To obtain the converse inequality, we claim that if $A \in \dM(n,1)$ attains the infimum in $\gamma_f^q(B)$, then $cA$ attains the infimum in $\gamma_{cf}^q(B)$. If that were the case, then 
\begin{align*}
    c \gamma_{f}^q(B) = c\gamma_f^q(B, A) = \gamma_{cf}^q(B, cA) = \gamma_{cf}^q(B).
\end{align*}
To prove the claim, assume the contrary. Then there exists $A_1 \in \dM(n,1)$ so that $\gamma_{cf}(B, A_1) < (1-c_1)\gamma_{cf}(B, cA)$, for some small $c_1>0$. But then 
\begin{align*}
    \gamma_f^q(B, A_1/c) = c^{-1}\gamma_{cf}^q(B, A_1) < \tfrac{1-c_1}{c} \gamma_{cf}^q(B, cA) =(1-c_1) \gamma_f^q(B, A).
\end{align*}
This contradicts that $A$ infimises $\gamma_f^q(B)$. The same proof works for $\widetilde{\gamma}_f$.
\end{proof}

\begin{lemma}\label{l:gamma<gradf}
Let $1< p < \infty$ and $f \in M^{1,p}(\Sigma)$. Assume that $q$ is in the Dorronsoro range\footnote{Which, recall, is given by 
\begin{itemize}
    \item If $d=1$, then $1 \leq q \leq \infty$.
    \item If $d\geq 2$, then $1 \leq q < p^*$ whenever $1<p<2$; and $1 \leq q < 2^*$ whenever $2 \leq p <\infty$.
\end{itemize}
}.
Then there exist constants $C,\,\Lambda \geq 1$ so that 
\begin{equation}
\label{e:gamma<av-grad}
\gamma_{f}^{q}(B)\leq C \ps{\avint_{5\Lambda B} g^{p}d\sigma }^{\frac{1}{p}},
\end{equation}
whenever $g \in \Grad_p(f)$ and $B$ is a ball centered on $\Sigma$.
The constants $C, \, \Lambda$ depend on $d$, $n$, the Ahlfors regularity constant of $\Sigma$, $p$ and $q$. The same holds for $\wtgamma$.
\end{lemma}

\begin{proof}
Let $A_0 \in \dM(n,1)$ be so that $A_0 \equiv f_{B}$ (so, in particular, $|\grad A_0|=0$). For $q \geq 1$, we have
\begin{align*}
\gamma_{f}^{q}(B) \leq \gamma_{f}^{q}(B,A_0) =\Omega_{f}^{q}(B,A_0)+0\cdot \beta_{\sigma}^{d,q}(B) = \ps{\avint_{B}\ps{\frac{|f-f_{B}|}{r_{B}}}^{q}d\sigma}^{\frac{1}{q}}.
\end{align*}
 The same estimate holds for $\wtgamma$ by replacing $\beta_{\Sigma}^{d,q}(B) $ with $\salpha(B)$. By Proposition \ref{t:Hajlasz-poincare}, if $g \in \Grad_p(f)$, the pair $(f,g)$ satisfy a $(p',p')$-Poincar\'{e} inequality, for all $1< p' \leq p$. By Jensen's inequality, the same pair satisfies a $(1,p')$-Poincar\'{e} inequality for the same range of $p'$. \\
 
 \noindent
 We want to apply Theorem \ref{t:sobolevmetpoincare} with the metric measure space given by $E$, with the Euclidean distance and measure $\mu=\sigma=\dH^d|_E$. Note that since $E$ is Ahlfors $d$-regular, then \eqref{e:condition-s} is satisfied with $s=d$.
 We consider two distinct cases. 
 \begin{itemize}
     \item If $d=1$, then $p'>d$ always, and thus we can apply \eqref{e:HK-p>s} for all $1 \leq q \leq \infty$; this gived \eqref{e:gamma<av-grad} for all $1\leq q \leq \infty$ on the left hand side, and all $1<p'\leq p$ on the right hand side.
     \item $d \geq 2$. In this case, we distinguish two subcases.
     \begin{itemize}
         \item $1<p<2$. In this case, we have to consider the range of $q$'s $1\leq q < p^*$, as given in the hypotheses of the Lemma. Since $d\geq 2$, then $d>p$, and thus Theorem \ref{t:sobolevmetpoincare}(1) is applicable. Theorem \ref{t:sobolevmetpoincare}(1) applies with the same range of $q$'s, i.e. $1\leq q < p^*$. So, we obtain \eqref{e:gamma<av-grad} in this case.
         \item $2\leq p <\infty$. The hypotheses of the Lemma tells us that $1 \leq q < 2^*$. If $p<d$ (which forces $d>2$) then $2^*\leq p^*$, and thus we apply again Theorem \ref{t:sobolevmetpoincare}(1), and obtain \eqref{e:gamma<av-grad}. For $p \geq d$, \eqref{e:gamma<av-grad} is immediate using Theorem \ref{t:sobolevmetpoincare}(2) and (3) and Remark \ref{r:sobolevmetpoincare}.
     \end{itemize}
 \end{itemize}
The estimate for $\wtgamma$ follows from the above, since in this case $q=1$.

\end{proof}

\begin{lemma}\label{l:A-A}
Let $1<p<\infty$, $f \in M^{1,p}(\Sigma)$ and $B,B'$ two balls centered on $E$ so that $B\subset B'$ and $r_B \approx r_{B'}$. Then, whenever $q$ is in the Dorronsoro range, 
\begin{align}
& |\grad A_{B}-\grad A_{B'}|\lec \gamma_{f}^{q}(B'); \label{e:gradA-gradA}\\
& |\grad A_{B'}|\lec_{c} \gamma_{f}^{q}(B')+\ps{\avint_{5\Lambda B'}g^{p}}^{\frac{1}{p}} \mbox{ for } \, g \in \Grad_p(f); \label{e:A<gamma+av}\\
& |A_{B}(x)-A_{B'}(x)|
\lec_{c}  \gamma_{f}^{q}(B')(\dist(x,B')+r_{B'}), \label{e:AB-AB'}
\end{align}
where $A_B$ and $A_{B'}$ are the maps from Lemma \ref{l:A_B}.
\end{lemma}

\begin{proof}
First we show \eqref{e:AB-AB'} assuming \eqref{e:gradA-gradA}. Let $x_0$ be a point in $B'\cap \Sigma$ and $x\in \R^{n}$. We compute
\begin{align}
& |A_{B}(x)-A_{B'}(x)| \nonumber \\
& \leq |\grad A_{B}(x-x_{0}) - \grad A_{B'}(x-x_{0})|+|A_{B}(x_{0})-A_{B'}(x_{0})| \nonumber \\
& \leq |\grad A_{B}-\grad A_{B'}||x-x_{0}|+|A_{B}(x_{0})-f(x_{B})|+|f(x_{B})-A_{B'}(x_{0})| \label{e:form2}
\end{align}
Using \eqref{e:gradA-gradA}, we can bound the first term in \eqref{e:form2} by $\gamma_f^q(B') |x-x_0| \lesssim \gamma_f^q(B') \dist(x, B')$. Using Chebyshev's inequality, we can choose $x_0 \in B \cap \Sigma$ such that the last two terms in \eqref{e:form2} are bounded by $\gamma_f^q(B)r_B$ and $\gamma_f^q(B')r_{B'}$, respectively. We then obtain \eqref{e:AB-AB'} using the quasi-monotonicty of $\gamma_f^q$, as per \eqref{e:gamma-monotone}.\\

\noindent
We now focus on proving \eqref{e:gradA-gradA} and \eqref{e:A<gamma+av}. Let $\ve>0$ to be chosen later. We look at two cases.\\

\noindent
\textit{\underline{Case 1.}} First, we assume that $\beta_{\Sigma}^{d,q}(B')\geq \ve$. Let $A_B$, $A_{B'} \in \dM(n,1)$ as in Lemma \ref{l:A_B}. Using the bound $|\nabla A_B|\beta_\Sigma^{d,q}(B) \lesssim \gamma_f^q(B)$ (and the equivalent for $B'$, see \eqref{e:ABbeta<gammaAB<gamma}) and the quasi-monotonicity of $\gamma_f^q$, see \eqref{e:gamma-monotone}, we compute
\begin{multline*}
|\grad A_{B}-\grad A_{B'}|
\leq |\grad A_{B}|+|\grad A_{B'}|\\
\leq \ve^{-1}  |\grad A_{B}| \beta_{\Sigma}^{d,q}(B)+\ve^{-1}|\grad A_{B'}| \beta_{\Sigma}^{d,q}(B')
\stackrel{\eqref{e:ABbeta<gammaAB<gamma}\atop \eqref{e:gamma-monotone}}{\lec} \ve^{-1} \gamma_{f}^{q}(B').
\end{multline*}
Moreover, we clearly have 
 \begin{equation*}
 |\grad A_{B'}|\leq \ve^{-1}  |\grad A_{B'}| \beta_{\Sigma}^{d,q}(B')
 \stackrel{\eqref{e:ABbeta<gammaAB<gamma}}{\lec} \ve^{-1} \gamma_{f}^{q}(B').    
 \end{equation*}
This proves \eqref{e:gradA-gradA} and \eqref{e:A<gamma+av} in this case. Note that there are no constraints on $q$ other than $q \geq 1$.\\

\noindent
Before looking into Case 2, we need the following auxiliary claim, which we will prove below. 
\begin{claim}\label{c:balanced-points}
	Keep the notation and assumptions form Lemma \ref{l:A-A}. Let $\{B_i\}_{i=0}^d$ be the $d+1$ balls found by applying Lemma \ref{l:balanced-balls} to the ball $B$ (as in the statement of Lemma \ref{l:A-A}). We can find $d+1$ points $x_{i}\in B_{i}\cap \Sigma$, $i=0,...,d$, so that if $y_{i}=\pi_{B}(x_i)$ and $y_i'=\pi_{B'}(x_{i})$,
	\begin{equation}
		\label{e:xyy'}
		|x_{i}-y_{i}| + |x_{i}-y_{i}'| \lec_{c} \beta_{\Sigma}^{d,q}(B)r_{B} +\beta_{\Sigma}^{d,q}(B')r_{B'},
	\end{equation}
	and
	\begin{equation}
		|f(x_{i})-A_{B}(y_{i})|  +|f(x_{i})-A_{B'}(y_{i}')|
		\lec_{c} \Omega_{f}^{q}(B,A_B)+\Omega_{f}^{q}(B',A_{B'}).
		\label{e:fxi-ABxi}
	\end{equation}
\end{claim}

\noindent
\textit{\underline{Case 2.}} Now suppose
\[
\beta_{\Sigma}^{d,q}(B)\lec_{c}\beta_{\Sigma}^{d,q}(B')< \ve.
\] 
Let $x_i$, $0\leq i\leq d$ be the points found in Claim \ref{c:balanced-points}, and recall the notation $\{y_i\}_{i=0}^d \subset P_B$ and $\{y'_i\}_{i=0}^d \subset P_{B'}$. Since $P_B \in \dA(n,d)$ and $y_0 \in P_B$, then $P_B-y_0 \in \dG(n,d)$ and $P_B^\perp-y_0$ is its orthogonal complement. Any $v \in \bB$ can be written as $\alpha_1 v_1 + \alpha_2 v_2$, where $v_1 \in P_B-y_0$, $v_2 \in P_B^\perp -y_0$, $|v_i|=1$, and $\alpha_i \leq 1$. Thus
\begin{align}
    |\nabla A_B - \nabla A_{B'}| & = \sup_{v \in \bB} \left|(\nabla A_B - \nabla A_{B'})v\right| \nonumber \\
    & \leq \sup_{\substack{\alpha_i\leq 1, \, i=1,2\\ |v_i|=1, i=1,2}} \sum_{i=1}^2 \left|(\nabla A_B - \nabla A_{B'})\alpha_i v_i\right| \nonumber \\
    & \leq \sup_{v \in \bB \cap P_B-y_0}\left|(\nabla A_B - \nabla A_{B'})v\right| \nonumber \\
    & \quad \quad \quad + \sup_{v \in \bB \cap P_B^\perp-y_0}\left|(\nabla A_B - \nabla A_{B'})v\right|. \label{e:form4}
\end{align}
We will bound the last two terms in \eqref{e:form4} separately.\\

\noindent
If we choose $\ve>0$ sufficiently small with respect to the constant $c$ in \eqref{e:d-dim-points-1}, Lemma \ref{l:balanced-balls}, then
\begin{equation}
\label{e:y-spread}
\dist(y_{i+1},{\rm span}\{y_{0},...,y_{i}\})\geq c'r_{B} \;\;\mbox{ and }\;\; \dist(y_{i+1}',{\rm span}\{y_{0}',...,y_{i}'\})\geq c'r_{B},
\end{equation}
where $c' \approx c$. We immediately relabel $c'$ as $c$ to keep a manageable notation.
In particular, the vectors $\{u_{i}=y_{i}-y_{0}:i=1,...,d\}$ form a basis for $P_{B}-y_{0}$, and also $|u_i| \approx r_B$. Hence there exists a $1 \leq j \leq d$ so that 
\begin{align}
\sup_{v\in P_{B}-y_{0}\atop |v|=r_{B}} |(\grad A_{B}-\grad A_{B'})v|
\lec_d |(\grad A_{B}-\grad A_{B'})(y_{j}-y_{0})|. \label{e:form3}
\end{align}
We continue computing:
\begin{align*}
\eqref{e:form3} &  \approx_d \max_{1 \leq i\leq d}|(A_{B}(y_{j})-A_{B}(y_{0}))-(A_{B'}(y_{j})-A_{B'}(y_{0}))|\\
& \quad \quad \quad \lec_d \left|\left(A_{B}(y_{j})-A_{B}(y_{0})\right)+ \left(A_{B'}(y_{0}')-A_{B'}(y_{j}'\right))\right| \\ &  \quad \quad \quad \quad  \quad \quad \quad+ \left|\grad A_{B'}\right|\left(|y_{j}-y_{j}'|+|y_{0}-y_{0}'|\right) \\
& \quad \quad \quad =: T_1 + T_2.
\end{align*}
Now, 
\begin{align*}
    T_1 &\leq \left| \left(A_B(y_j)-A_B(y_0)\right) - \left(f(x_j)-f(x_0)\right)\right| \\
    & \quad \quad \quad + \left|(f(x_j)-f(x_0)) - \left(A_{B'}(y'_j)- A_{B'}(y'_0) \right)\right|.
\end{align*}
Recall now that we chose the points $\{x_i\}$ appropriately, and so, by Claim \ref{c:balanced-points}, they satisfy \eqref{e:fxi-ABxi}. This gives the bound
\begin{align*}
    T_1 & \lesssim_{d} \Omega_f^q(B, A_B)r_B + \Omega_f^q(B', A_{B'})r_{B'} \\ & \stackrel{\eqref{e:gamma-monotone}}{\lesssim_{d,c}} \Omega_f^q(B', A_{B'})r_{B'} \stackrel{\eqref{e:ABbeta<gammaAB<gamma}}{\lesssim_{d,c}} \gamma_f^q(B')r_{B'}.
\end{align*}
On the other hand, again by the choice of $\{x_i\}$ \eqref{e:xyy'},
\begin{align*}
T_2 \lesssim |\grad A_{B'}|\beta_{\Sigma}^{d,q}(B')r_{B'}
 \stackrel{\eqref{e:ABbeta<gammaAB<gamma}}{\lec} \gamma_{f}^{q}(B')r_{B'}.
\end{align*}
This shows that
\begin{equation*}
    \sup_{\substack{v \in P_B-y_0 \\ |v|=1}} \left|( \nabla A_B - \nabla A_{B'})v\right| \lesssim \gamma_f^q(B').
\end{equation*}
We now bound the last term in \eqref{e:form4}. If $v \in P_B^\perp - y_0$, then 
\begin{equation*}
|\pi_{B'}(v)|
= |\pi_{B'}(v)-\pi_{B}(v)|
\leq \|\pi_{B'}-\pi_{B}\|_{{\rm op}}\cdot |v|\stackrel{\eqref{e:beta-angle} \atop {{\rm Lemma}\,  \ref{l:beta-properties}}}{\lec} \beta_{\Sigma}^{d,q}(B')|v|,
\end{equation*}
and so if $|v|=r_{B}$,
\[
|(\grad A_{B}-\grad A_{B'})v|
=|\grad A_{B'}v|
\leq |\grad A_{B'}|\cdot \beta_{\Sigma}^{d,q}(B')r_{B}.
\]
Combining the above estimates gives \eqref{e:gradA-gradA} also in Case 2 (small $\beta$'s).\\ 

\noindent
What we have left, is to prove \eqref{e:A<gamma+av} for Case 2 (small $\beta$'s).
By Chebychev, we can find points $z_i\in B_i\cap \Sigma$ so that 
\[
\frac{|A_{B'}(z_i)-A_{B'}(z_0)|}{|z_i-z_0|}
\lec 
\avint_{B_{i}}\avint_{B_{0}}\frac{|A_{B'}(x)-A_{B'}(y)|}{r_{B'}}d\sigma(x)d\sigma(y).
\]
Just as before with the $y_i$, \eqref{e:d-dim-points} implies that 
\[
|\grad A_{B'}|
\lec \max_{i=1,...,d}\frac{|A_{B'}(z_i)-A_{B'}(z_0)|}{|z_i-z_0|}.
\]
Combining these inequalities gives

\begin{align*}
|\grad A_{B'}|
& \lec  \max_{i}\avint_{B_{i}}\avint_{B_{0}}\frac{|A_{B'}(x)-A_{B'}(y)|}{r_{B'}}d\sigma(x)d\sigma(y)\\
& \lec \Omega_{f}^{1}(B,A_{B'})+\max_{i}\avint_{B_{i}}\avint_{B_{0}}\frac{|f(x)-f(y)|}{r_{B'}}d\sigma(x)d\sigma(y).
\end{align*}
Now, using \eqref{e:omega-monotone} and \eqref{e:ABbeta<gammaAB<gamma}, we obtain 
\begin{equation}
    \Omega_f^1(B', A_B') \lesssim \gamma_f^q(B', A_B') \lesssim \gamma_f^q(B'). 
\end{equation}
To deal with the other term, we compute, for each $0\leq i \leq d$, 
\begin{align*}
& \avint_{B_{i}}\avint_{B_{0}}\frac{|f(x)-f(y)|}{r_{B'}}d\sigma(x)d\sigma(y) \\
&\leq \avint_{B_{i}}\avint_{B_{0}}\frac{|f(x)-f_{B'}|}{r_{B'}}d\sigma(x)d\sigma(y) + \avint_{B_{i}}\avint_{B_{0}}\frac{|f_{B'}-f(y)|}{r_{B'}}d\sigma(x)d\sigma(y)\\
& \lesssim \avint_{B'}\frac{|f(x)-f_{B'}|}{r_{B'}} \, d \sigma(x) \lesssim \left(\avint_{5\Lambda B'} g^p d\sigma\right)^{\frac{1}{p}},
\end{align*}
where we used the fact that $g \in \Grad_p(f)$ and Theorem \ref{t:Hajlasz-poincare} together with Jensen's inequality.
This proves \eqref{e:A<gamma+av} in this case as well, and finishes the proof of the lemma.
\end{proof}

\begin{proof}[Proof of Claim \ref{c:balanced-points}]
It is easy to see that we can assume that the quantities in the upper bounds of \eqref{e:xyy'} and \eqref{e:fxi-ABxi} are positive, for otherwise the claim follows trivially. Let $\{B_i\}_{i=0}^d$ be the family of balls as in the statement of the Claim. Let $C>0$ and set
\begin{align*}
E_{i}=\big\{ & x\in B_{i}: \dist(x,P_{B})\geq C \beta_{\Sigma}^{d,q}(B)r_{B}, \\
&  \dist(x,P_{B'})\geq C \beta_{\Sigma}^{d,q}(B')r_{B'}, \\
& |f(x)-A_{B}(x)|>C\Omega_{f}^{q}(B,A_B)r_{B},\;\; \mbox{ or }\\  & |f(x)-A_{B'}(x)|>C\Omega_{f}^{q}(B',A_{B'})r_{B'}\big\}.
\end{align*}
If $E_{i}=B_{i}$, then by Chebychev's inequality,
\begin{align*}
1
& \leq \avint_{B_i} \frac{\dist(x,P_{B})}{ C \beta_{\Sigma}^{d,q}(B)r_{B}}d\sigma(x) + \avint_{B_i} \frac{\dist(x,P_{B'})}{ C \beta_{\Sigma}^{d,q}(B')r_{B'}}d\sigma(x) \\
& \qquad + \avint_{B_{i}} \frac{|f(x)-A_{B}(x)|}{C\Omega_{f}^{q}(B,A_B)r_{B}}d\sigma(x) +\avint_{B_{i}} \frac{|f(x)-A_{B'}(x)|}{C\Omega_{f}^{q}(B',A_{B'})r_{B'}}d\sigma(x)  \\
& \lec_{c} C^{-1}
\end{align*}
which is a contradiction for $C>0$ large enough (depending on $c$), thus we can find $x_{i}\in B_{i}\cap \Sigma\backslash E_{i}$ which satisfies the above properties (and recall that $A_{B}(x_i)=A_{B}\circ\pi_{B}(x_i)=A_{B}(y_{i})$). This proves the claim. 
\end{proof}

%
%
%
%
%

\newpage
\begin{center}
\part{}\label{p:Part1}

\Large\bfseries{Proof of Theorem \ref{t:main-A}.\\ The Haj\l{}asz upper gradients control the square function.} 
\end{center}

 \section{Proof of Theorem \ref{t:main-A} via a good $\lambda$-inequality}\label{s:tA-good-lambda}
 In the following two sections, we prove Theorem \ref{t:dorronsoro-I} below, of which Theorem \ref{t:main-A} is an immediate consequence. To state it, let us introduce some notation.
  For $q\geq 1$, and $Q_0 \in \dD(\Sigma)$, set
  \begin{equation}\label{e:G_q}
\G^q_{Q_0}f(x) =\ps{\sum_{Q_0 \supseteq Q\ni x} \gamma_{f}^{q}(B_{Q})^{2}}^\frac{1}{2} \,\, \mbox{ and } \,\, \widetilde{\G}_{Q_0}f(x) =\ps{\sum_{Q_0 \supseteq Q\ni x} \widetilde{\gamma}_{f}(B_{Q})^{2}}^\frac{1}{2}
 \end{equation} 
 \begin{theorem}
 \label{t:dorronsoro-I}
Fix $2 \leq d \leq n-1$ and $1<p<\infty$. Suppose that $\Sigma \subset \R^n$ is a uniformly $d$-rectifiable set and $Q_0\in \dD(\Sigma)$. Let $f \in M^{1,p}(\Sigma)$. Then, if $C_1\geq 1$ is chosen sufficiently large and if $1 \leq q< \min\{2^*, p^{*}\}$,
 \begin{equation}
 \|\G^q_{Q_{0}}  f\|_{L^p(Q_0)} \lec \left(\int_{2C_1 B_{Q_{0}}}g^{p} \, d \sigma \right)^{\frac{1}{p}}
 \end{equation}
 whenever $g \in \Grad_p(f)$. The same statement holds if we replace $\G^q_{Q_0}$ with $\widetilde{\dG}_{Q_0}$.
\end{theorem}

\noindent
Before getting started, we need the following corollary of Proposition \ref{t:DS-Dorronsoro}.
 \begin{lemma}
 \label{c:DS-Dorronsoro}
 Let $E \subset \R^n$ be uniformly $d$-rectifiable and $f:E\rightarrow \R$ be $L$-Lipschitz with $L>0$. If $1\leq q < \frac{2d}{d-2}$ \textup{(}or $1\leq q \leq \infty$ if $d=1$\textup{)}. Then $\Omega_{f}^{q,2L}(x,r)^{2}\frac{dxdr}{r}$ is a $CL^2$-Carleson measure. 
 \end{lemma}
 
 \begin{proof}
 We apply Lemma \ref{t:DS-Dorronsoro} to the function $\frac{f}{L}$, which is now $1$-Lipschitz, and with $N=2$. Then any ball $B$ centered on $E$ with $r_{B}<\diam E$,
 \[
 \int_{B}\int_{0}^{r_{B}} \Omega_{f/L}^{q,2}(x,r)^{2}\frac{dr}{r}d\sigma(x)\lec r_{B}^{d}
 \]
 and so for 
\[
 \int_{B}\int_{0}^{r_{B}} \Omega_{f}^{q,2L}(x,r)^{2}\frac{dr}{r}d\sigma(x)
 = L^{2}\int_{B}\int_{0}^{r_{B}} \Omega_{f/L}^{q,2}(x,r)^{2}\frac{dr}{r}d\sigma(x)\lec L^{2}  r_{B}^{d} .
\]
 
 \end{proof}

 \noindent
\begin{remark}\label{r:notation-dorronsoro-I}
  We set some notation. Fix $n,d,p, \Sigma, Q_0$ and $f \in M^{1,p}(\Sigma)$ as in Theorem \ref{t:dorronsoro-I}. Fix also $C_1 > 2\Lambda$. These will remain fixed throughout the current section. Let $|\grad_H f| \in \Grad_p(f)$ be the minimal Haj\l{}asz upper gradient. For $1 \leq s<p$, set
  
\begin{align*}
    \dG^qf(x) :=\dG_{Q_0}^q f(x) \quad \mbox{ and } \quad \dM^{s} f(x) := \left(M_{C_1\ell(Q_0)} |\grad_H f|^{s} (x) \right)^{\frac{1}{s}},
\end{align*}
where for a function $u$
\begin{equation*}
    M_R u(x)= \sup_{0<r<R} \avint_{B(x,r)} |u| \, d\sigma.
\end{equation*}
\end{remark}
 \begin{lemma}\label{l:main-I}
Let $q$ be in the Dorronsoro's range and
\begin{equation}\label{e:s-range-pqd}
	\frac{dq}{d+q}< s <\min\{q,p, d\}
\end{equation}  
if $q>1$ or $s=1$ if $q=1$. For each $\alpha>1$ there exists an $\ve>0$ so that for each $\lambda >0$,
 \begin{equation}
 \label{e:G>a,M<vea}
 \big|\big\{x \in Q_0 \, |\, \dG^q f(x) > \alpha\lambda ,\; \dMi f(x) \leq \ve \lambda\big\}\big|
 < \frac{15}{16} \big|\big\{ x\in Q_0 \, |\, \dG^q f(x) > \lambda\big\}\big|.
 \end{equation}
 \end{lemma}
 \begin{remark}
 Before going any further we check that the choice of $s$ in the hypotheses of Lemma \ref{l:main-I} is in fact possible. It suffices to show that, with $1<p<\infty$ and $q$ is in the Dorronsoro range, then $dq/(d+q) < \min\{p, q, d\}$. Recall that the Dorronsoro range is given by 
 \begin{itemize}
 	\item If $d=1$, then $1 \leq q \leq \infty$.
 	\item If $d\geq 2$, then $1 \leq q < p^*$ whenever $1<p<2$; and $1 \leq q < 2^*$ whenever $2 \leq p <\infty$.
 \end{itemize}
That $dq/(d+q) < \min\{d, q\}$, is true in all cases.
 On the other hand, we see that
 \begin{itemize}
 	\item If $d=1$, then $q/(1+q)<1<p$, so that \eqref{e:s-range-pqd} is satisfied in this case.
 	\item If $d \geq 2$, we have the usual two cases: if $1<p<2$, then $q<p^*$. But note that $dp/(d+q)< p$ if and only if $q< p^*$. On the other hand, if $2\leq p <\infty$, then $q<2^*$, and $2^* \leq p^*$ since $p\geq 2$; thus \eqref{e:s-range-pqd} is satisfied in this case, too. 
 \end{itemize}    
 \end{remark}

\noindent
Now let us show how Lemma \ref{l:main-I} proves Theorem \ref{t:dorronsoro-I}. 
\begin{proof}[Proof of Theorem \ref{t:dorronsoro-I}]
Let $\alpha>1$ and let $\ve>0$ be as in Lemma \ref{l:main-I}. We compute
\begin{align*}
\int_{Q_0} \left(\dG^q f\right)^p \, d\sigma(x)
& =\int_{0}^{\infty} \left|\{x \in Q_0: \dG^q f(x) >\lambda\}\right|\lambda^{p-1}\, d\lambda\\
& =  \alpha^{p} \int_{0}^{\infty} \left|\{x \in Q_0:  \dG^q f(x)>\alpha \lambda\}\right|\lambda^{p-1}d\lambda\\
& \leq \alpha^{p}  \int_{0}^{\infty} \left|\left\{x \in Q_0:  \dG^q f(x)>\alpha \lambda \,\,  {\rm and} \,\, \dMi f(x) \leq \ve\lambda \right\}\right| \lambda^{p-1}d \, \lambda \\
& \quad \quad \quad \quad + \alpha^{p} \int_{0}^{\infty} \left|\left(\{ x \in Q_0: \dMi f(x) > \ve\lambda \}\right)\right| \lambda^{p-1}\, d\lambda \\
& := I_1 + I_2.
\end{align*}
Because Lemma \ref{l:main-I} holds for each $\alpha>1$, we apply it with $\alpha$ satisfying $\alpha^p \tfrac{15}{16}< \tfrac{16}{17}$. Then 
\begin{align}\label{e:Dorr-I-a}
	I_1 \stackrel{\eqref{e:G>a,M<vea}}{<} \alpha^p \frac{15}{16} \int_0^\infty \left|\left( \{ x \in Q_0: \dG^q f(x) > \lambda \} \right)\right| \lambda^{p-1} \, d\lambda \leq \frac{16}{17} \int_{Q_0} \left( \dG^q f \right)^p \, d \sigma.
\end{align}
On the other hand we trivially have
\begin{align}\label{e:Dorr-I-b}
	I_2 \leq a^p \ve^{-p} \int_{Q_0} \left( \dMi f\right)^p  \, d \sigma.
\end{align} 
Putting \eqref{e:Dorr-I-a} and \eqref{e:Dorr-I-b} together (and 'hiding' the upper bound for $I_1$ on the left hand side), 
\begin{align*}
\int_{Q_0} (\dG^q f)^{p}\,  d\sigma
& <\frac{34}{\ve^{p}} \int_{Q_0} \left(\dMi f\right)^{p} \,d\sigma
\lec_{\ve} \int (M_{C_1\ell(Q_0)}(\one_{2C_1 B_{Q_0}} \,g^{s}))^{\frac{p}{s}}\, d\sigma\\
& \lec_{p/s} \int_{2C_1B_{Q_{0}}} (|\grad_H f|^{s})^{\frac{p}{s}} d\sigma
 =\int_{2C_1 B_{Q_{0}}}  |\grad_H f|^{p} d\sigma
\end{align*}
where, in the last inequality, we used that the Hardy-Littlewood maximal function is bounded from $L^{p/s}(\sigma)$ to $L^{p/s}(\sigma)$, which holds since $s<p$ by assumption.
\end{proof}

\section{The proof of the good-$\lambda$ inequality}\label{s:tA-proof-lambda}
In this section we prove the good-$\lambda$ inequality in Lemma \ref{l:main-I}.
We start with some set up. Given $\lambda>0$, let $\cC_\lambda$ be the set of maximal cubes in 
 \begin{equation*}
\{x\in Q_0 \, :\,  \dG^q f (x) > \lambda\}.
\end{equation*}
To prove Lemma \ref{l:main-I}, it suffices to show that it holds for each $R \in \cC_\lambda$. That is, we aim to prove that for each $R\in \cC_\lambda$, the following is true: for each $\alpha>1$ there is $\ve>0$ so that
 \begin{equation}
 \label{e:G>a,M<vea-in-R}
  |\{x\in R \, :\,  \dG^q f(x)> \alpha\lambda ,\;\; \dMi f(x)\leq \ve \lambda\}|
  <\frac{15}{16} |R|.
  \end{equation}
 \noindent
Let now $R\in \cC_\lambda$, $\alpha>0$, and $\ve>0$ be a constant to be determined later and will only depend on $\alpha$. Let
 \begin{align}\label{e:E_R-I}
 E_R= \{x\in R \, :\,  \dMi f(x)\leq \ve \lambda\}.    
 \end{align}
 If $|E_R|<\frac{1}{2}|R|$, then
 \begin{equation*}
     |\{x\in R \, :\,  \dG^q f(x)> \alpha\lambda ,\;\; \dMi g(x)\leq \ve \lambda\}| \leq |E_R| \leq \frac{15}{16}|R|
 \end{equation*}
and we're done. We then assume that $|E_R|\geq \frac{1}{2}|R|$. Recall that the defining property of $|\grad_H f| \in \Grad_p(f)$ is that for $\sigma$-a.e. $x, y \in \Sigma$  we have that 
 \begin{align*}
|f(x)-f(y)| &\leq  |x-y| (|\grad_H f(x)| + |\grad_H f(y)|.
\end{align*}
Let $\delta>0$ be arbitrary, and pick $\delta_1$ so small so that $|x-y|\delta_1 < \delta$. Since $|\grad_H f| \in L^p$, we know that, if $\phi_t$ is an approximate identity, $\lim_{t \to 0} \phi_t * |\grad_H f|(x) = |\grad_H f|(x)$ for $\sigma$-a.e. $x$ (see \cite{duoandikoetxea2001fourier}, Corollary 2.9). Then we may pick $t$ so small, so that the following inequalities hold:
\begin{align*}
	|f(x)-f(y)| &\leq 2\delta + |x-y|(\phi_t*|\grad_H f|(x) + \phi_t * |\grad_H f|(y))\\
	& \leq 2 + \delta \left[ \left(\avint_{B(x,t)} |\grad_H f|^s \right)^{\frac{1}{s}} + \left(\avint_{B(y,t)} |\grad_H f|^s \right)^{\frac{1}{s}} \right]\\
	& \leq 2 \delta + |x-y|\left[ \left(\sup_{r\leq \Lambda |x-y|} \avint_{B(x,r)} |\grad_H f|^s \right)^{\frac{1}{s}} + \left(\sup_{r\leq \Lambda |x-y|} \avint_{B(y,r)} |\grad_H f|^s \right)^{\frac{1}{s}} \right]
\end{align*}
 If $x,y \in E_R$, then $|x-y| \leq \ell(Q_0)$, and so, since $C_1 \geq 2 \Lambda$, $M_{\Lambda|x-y|} g^s \leq M_{C_1\ell(Q_0)} |\grad_H f|^s$. Hence $f$ is $C\ve\lambda$-Lipschitz on $E_R$. By Kirszbraun's extension theorem, we can extend $f$ to a $C\ve\lambda$-Lipschitz function $F$ defined on all of $\Sigma$. Recall that $q < \tfrac{2d}{d-2}$. Thus we can apply Lemma \ref{c:DS-Dorronsoro} and with Chebychev's inequality we find $E^1_R\subseteq E_R$ with $|E^1_R|\geq \frac{1}{4}|R|$ so that,
 \begin{equation}
 \label{e:sum-over-E'}
\sum_{R\supseteq Q\ni x} \Omega_{F}^{q,2C\ve\lambda }(B_{Q})^{2}  \lec ( \ve \lambda)^{2} \;\; \mbox{ for }x\in E^1_R.
\end{equation}
Let $A_{Q}^{F}$ be the affine map minimizing $\Omega_{F}^{q,2C\ve \lambda }(B_{Q})$ (not to be confused with the affine map from Lemma \ref{l:A_B}). So, in particular, 
\begin{equation}
\label{e:AQF<2Ce}
|\grad A_{Q}^{F}|\leq 2C\ve\lambda\;\;\mbox{ if  $Q\subseteq R$ intersects $E_R^1$}.
\end{equation}
\begin{lemma}\label{l:aux-lemma-D-I}
Assume the notation of Remark \ref{r:notation-dorronsoro-I}. Suppose that $1 \leq q < \min\{2^*, p^*\}$ and $1\leq s < \min\{q,p,d\}$ if $q>1$ or $s=1$ if $q=1$. Then
\begin{equation}
\label{e:intE'Omega}
\avint_{E^1_R}\sum_{R\supseteq Q\ni x} \Omega_{f}^{q}(B_{Q},A_{Q}^{F})^{2}
\lec (\ve \lambda)^2.
\end{equation}
\end{lemma}

\noindent
Let's finish the proof of the good-$\lambda$ inequality assuming Lemma \ref{l:aux-lemma-D-I}.
\begin{proof}[Proof of Lemma \ref{l:main-I}]
Lemma \ref{l:aux-lemma-D-I} together with Chebychev's inequality implies that there is $E^2_R\subseteq E^1_R$ with $|E^2_R|\geq \frac{1}{2}|E^1_R|\geq \frac{1}{8} |R|$ so that,  for all $x\in E^2_R$,
\begin{equation}
\label{e:sumomegaQAQF<elambda}
\sum_{R\supseteq Q\ni x} \Omega_{f}^{q}(B_{Q}, A_{Q}^{F})^{2}
\lec (\ve\lambda)^{2}
.\end{equation}
Since $\Sigma$ is uniformly $d$-rectifiable we can use Proposition \ref{p:strong-geometric-lemma}: there is $E^3_R\subseteq E^2_R$ with $|E^3_R|\geq \frac{1}{16}|R|$ so that for all $x\in E^3_R$
 \begin{equation}
 \label{e:sumbetaQ<1}
\sum_{R \supseteq Q \ni x}\beta_{\Sigma}^{d,q}(B_Q)^{2} \lec 1,
\end{equation}
where the constant behind the symbol $\lesssim$ depends on the uniform rectifiability constants of $\Sigma$. 
Denote by $\wh{R}$ the unique parent cube of $R$. Then, for $x\in E^2_R$ and $0<r\leq \ell(\widehat{R})$, if $x\in Q\subseteq \widehat{R}$ and $q\geq1$,
\begin{multline}\label{e:form10}
\gamma_{f}^{q}(B_{Q})
\leq \gamma_{f}^{q}(B_{Q},A_{Q}^{F})
= \Omega_{f}^{q}(B_{Q},A_{Q}^{F})+|\grad A_{Q}^{F}|\cdot \beta_{\sigma}^{d,q}(B_{Q})\\
\stackrel{\eqref{e:AQF<2Ce}}{\lec}  \Omega_{f}^{q}(B_{Q},A_{Q}^{F})+\ve\lambda \cdot \beta_{\sigma}^{d,q}(B_{Q}),
\end{multline}
which gives 
\begin{align}
 \sum_{R\supseteq Q\ni x}
\gamma_{f}^{q}(B_{Q})^2
  & \stackrel{\eqref{e:form10}}{\lesssim}  \sum_{R \supseteq Q \ni x} \Omega_f^q(B_Q, A^F_Q)^2 + (\ve\lambda)^2 \sum_{R \subseteq Q \ni x} \beta_\Sigma^q(B_Q)^2 \nonumber \\
& \quad \quad \quad \quad \quad \quad \quad  \stackrel{\eqref{e:sumomegaQAQF<elambda}, \eqref{e:sumbetaQ<1}}{\lec} (\ve\lambda)^{2} \;\; \mbox{ for all} \;\; x\in E^3_R. \label{e:form12}
\end{align}
For $\widetilde \gamma_f$ we obtain the very same estimate, since \eqref{e:sumbetaQ<1} also holds for $\alpha_\sigma^d$.
Because $R \in \cC_\lambda$ is a maximal cube of $\{x \in Q_0 \,|\, \dG^g f(x) > \lambda\}$, then
\begin{align}\label{e:form11}
    \sum_{\wh{R}\subseteq Q\subseteq Q_0}
\gamma_{f}^{q}(B_{Q})^2\leq \lambda^2.
\end{align}
Hence,
\begin{align*}
    \dG^q f(x)^2=\sum_{x\in Q\subseteq Q_0}
\gamma_{f}^{q}(B_{Q})^2
\stackrel{\eqref{e:form12}\atop \eqref{e:form11}}{\leq} \lambda^2+(C\ve\lambda)^2.
\end{align*}
Note that the constant $C$ in the last display depends only on $d,n$ and on $\Lambda$, which in turn depends only on $d$, the Ahlfors regularity constant, $p$ and $q$. Thus we can pick $\ve>0$ sufficiently small so that $C\ve +1 < \alpha$, to obtain
\begin{equation}
    \dG^q f(x)^2 \leq (\alpha \lambda)^2 \, \,  \mbox{ for each }  x \in E^3_R, \, R \in \cC_\lambda, 
\end{equation}
which implies that
\begin{equation*}
|\{x\in R: \dG^q f(x)\leq \alpha \lambda\}|
\geq |E^3_R|\geq \frac{1}{16} |R|,
\end{equation*}
and so
\begin{align*}
& |\{x\in R: \dG^q f(x)>\alpha \lambda,\;\; \dMi f(x)\leq \ve\lambda  \}| \\
&  \quad \quad \qquad \qquad 
\leq 
|\{x\in R: \dG^q f(x)>\alpha \lambda \}|
\leq \frac{15}{16}|R|.
\end{align*}
\end{proof}
\noindent 
This proves \eqref{e:G>a,M<vea-in-R} and finishes the proof. The next section is dedicated to proving the claim in \eqref{e:intE'Omega}.

\section{Proof of the square function estimate \eqref{e:intE'Omega}}\label{s:tA-square-function}

\begin{proof}[Proof of Lemma \ref{l:aux-lemma-D-I}]
By Jensen's inequality, it suffices to prove Lemma \ref{l:aux-lemma-D-I} assuming $q\geq 2$. Remark that this is consistent with the upper bound $q< \min\{2^*, p^*\}$, since, if the smaller number is $2^*$, we can find $2 \leq q < 2^*$; if $p^*$ is the smaller number, it holds nonetheless that $p^*>2$, since $p>1$. Recall also that we picked $s$ so that $dq/(d+q)<s<\min\{d,p, q\}$.
Now let 
\begin{align*}
    h=f-F,
\end{align*}
where recall that $F$ is a $C\ve \lambda$-Lipschitz function on $\Sigma$ which coincide with $f$ on $E_R$. Minkowski's inequality  gives
\begin{align*}
\Omega_{f}^{q}(B_{Q},A_{Q}^{F})
& =\ps{\avint_{B_{Q}} \ps{\frac{|f-A_{Q}^{F}|}{\ell(Q)}}^{q}d\sigma}^{\frac{1}{q}}
\\
& \leq \ps{\avint_{B_{Q}} \ps{\frac{|F-A_{Q}^{F}|}{\ell(Q)}}^{q}d\sigma}^{\frac{1}{q}} + \ps{\avint_{B_{Q}}\ps{\frac{h}{\ell(Q)}}^{q}}^{\frac{1}{q}},
\end{align*}
and so we get
\begin{align*}
& \int_{E^1_R}\sum_{\wh{R} \supseteq Q\ni x}   \Omega_{f}^{q}(B_{Q},A_{Q}^{F})^{2} d\sigma (x)
 \\ & \lesssim \int_{E^1_R} \sum_{\wh{R}\supseteq Q\ni x}   \Omega_{F}^{q}(B_{Q},A_{Q}^{F})^{2}d\sigma(x)   + \int_{E_R^1}\sum_{\wh{R}\supseteq Q\ni x}  \ps{\avint_{B_{Q}}\ps{\frac{|h|}{\ell(Q)}}^{q}}^{\frac{2}{q}}d\sigma(x) \\
& 
\stackrel{\eqref{e:sum-over-E'}}{\lec} (\ve\lambda)^{2}|E_R^1|
+\sum_{Q\subseteq R\atop Q\cap E_R\neq\emptyset} \ps{\avint_{B_{Q}}\ps{\frac{|h|}{\ell(Q)}}^{q}}^{\frac{2}{q}}|Q|\\
& \lec  (\ve\lambda)^{2}|R|+I.
\end{align*}
Now we estimate $I$. Let $\dW_R:=\{Q_j\}_{j \in \mathbb{N}} \subset \dD(\Sigma)$ be a Whitney decomposition of $\Sigma\backslash E_R^1$, that is, a family of maximal cubes $Q_j$ for which $10B_{Q_{j}}\cap E^1_R=\emptyset$. For each $Q_j \in \dW_R$, let $x_j \in E_R^1$ be so that 
\begin{align*}
    |x_{Q_j}-x_j| \approx \dist(Q_j, E_R^1).
\end{align*}
(Recall that $x_{Q}$ denotes the center of the cube $Q$). This in particular implies that we can find a ball $B_j$ centered on $x_j$ such that $r(B_j) \approx \ell(Q_j)$ and so that $Q_j \subset B_j$. We now compute
\begin{align*}
    \int_{Q_j} |f(y)- F(y)|^q \, d \sigma(y) & \lesssim \int_{B_j} |f(y) - f_{B_j}|^q \, d\sigma(y) \\
    & \qquad \qquad + \int_{B_j} |f_{B_j} - F(y)|^q \, d \sigma(y) := I_1(Q_j) + I_2(Q_j).
\end{align*}
We first estimate $I_2(Q_j)$.
We have \begin{align*}
   \int_{B_j} \left| f_{B_j} - F(y)\right|^q & = \int_{B_j} \left|\avint_{B_j} f(z) - F(y) \, d\sigma(z) \right|^q d \sigma(y) \\
    & \lesssim \int_{B_j}\left| \avint_{B_j} f(z)-F(z) d\sigma(z) \right|^q \, d \sigma(y)+ \int_{B_j} \left| \avint_{B_j} F(z)-F(y) d \sigma(z)\right|^q\, d \sigma(y)\\
    & =: I_{2,1}(Q_j) + I_{2,2}(Q_j).
\end{align*}
Since $F$ is $C\ve \lambda$-Lipschitz and $|B_j| \approx |Q_j| \approx \ell(Q_j)^d$, we have 
\begin{equation}\label{e:I22Q_j}
I_{2,2}(Q_j) \lesssim (\ve \lambda)^q \ell(Q_j)^{q+d}.
\end{equation}
On the other hand, since $F(x_j)=f(x_j)$ as $x_j \in E_R$, we have
\begin{align*}
    I_{2,1}(Q_j) & \lesssim \ell(Q_j)^d\left| \avint_{B_j} f(z) - f(x_j) \, d\sigma(z) \right|^q + \ell(Q_j)^d\left| \avint_{B_j} F(x_j) - F(z) \, d \sigma(z) \right|^q \\
    & :=I_{2,1,1}(Q_j) + I_{2,1,2}(Q_j).
\end{align*}
Again using that $F$ is $C\ve \lambda$ Lipschitz, we bound
\begin{equation}\label{e:I222Q_j}
    I_{2,1,2}(Q_j) \lesssim (\ve \lambda)^q \ell(Q_j)^{q+d}. 
\end{equation}
On the other hand, using that $g$ is an Haj\l{}asz upper gradient of $f$, we have
\begin{align*}
    \ell(Q_j)^d\left| \avint_{B_j} f(z)-f(x_j) \, d\sigma(z) \right|^q \lesssim \ell(Q_j)^{q+d} \left[\left( \avint_{B_j} g(z)\, d\sigma\right)^q + g(x_j)^q\right].
\end{align*}
By Lebesgue differentiation Theorem and Jensen's inequality, it is easy to see that $g(x_j) \lesssim \dM^s f(x_j)$. On the other hand, we have
\begin{align*}
    \avint_{B_j} g(z) d\sigma(z) \lesssim \left( \avint_{B_j}g(z)^s \, d\sigma(z) \right)^{\frac{1}{s}} \lesssim \dM^s f(x_j).
\end{align*}
Since $x_j \in E_R^1$, $\dM^sf (x_j) \lesssim \ve \lambda$. All in all we obtain
\begin{equation}\label{e:I221Q_j}
    I_{2,1,1}(Q_j) \lesssim (\ve \lambda)^q \ell(Q_j)^{q+d}.
\end{equation}
We estimate $I_1(Q_j)$. Note that since we chose $s<p$ and $f \in M^{1,p}(\Sigma)$, then $(f, g)$ satisfy a $(1,s)$-Poincar\'{e} inequality. This follows from Proposition \ref{t:Hajlasz-poincare}. Furthermore, since also $s<d$, we can apply Theorem \ref{t:sobolevmetpoincare}. By this theorem, since $q < s^*$, which hold because $qd/(q+d)<s$, we get that
\begin{align}\label{e:I1Q_j}
   I_1(Q_j) & \approx \ell(Q_j)^d \avint_{B_j} |f(y)-f_{B_j}|^q \, d\sigma(y) \notag\\
   & \qquad \lesssim \ell(Q_j)^{d+q} \left( \avint_{5 \Lambda B_j} g(z)^s \, d \sigma(z) \right)^{\frac{q}{s}} \lesssim \ell(Q_j)^{d+q} (\ve \lambda)^q,
\end{align}
as $x_j \in E_R^1$ and $C_1 \geq 2 \Lambda$. Gathering together the bounds \eqref{e:I221Q_j}, \eqref{e:I222Q_j}, \eqref{e:I22Q_j} and \eqref{e:I1Q_j}, we obtian that
\begin{equation}\label{e:estimate-h}
    \int_{Q_j} |h|^2 \, d\sigma \lesssim (\ve\lambda)^q \ell(Q_j)^{d+q}.
\end{equation}

\noindent
Recall that, since $q\geq 2$, then for $a_j\geq 0$, 
\begin{equation}\label{e:form20}
    \left(\sum_j a_{j}\right)^{\frac{2}{q}}\leq \sum_j a_{j}^{\frac{2}{q}}.
\end{equation}
Thus we get
\begin{align*}
\mathcal{I} &\approx \sum_{Q \subset R \atop Q \cap E_R \neq \emptyset} \left(\ell(Q)^{-d}\int_{B_Q} \left( \frac{|h(x)|}{\ell(Q)} \right)^q \, d\sigma (x) \right)^{\frac{2}{q}} |Q| \\
& \leq \sum_{Q\subseteq R\atop Q\cap E_R\neq\emptyset} \ps{\sum_{Q_j \in \dW_R \atop Q_{j}\cap B_{Q}\neq\emptyset}\frac{1}{\ell(Q)^{d+q}}\int_{Q_{j}} |h(x)|^q d \sigma(x)}^{\frac{2}{q}}|Q|\\
& \stackrel{\eqref{e:estimate-h}}{\lec} (\ve\lambda)^{2} \sum_{Q\subseteq R\atop Q\cap E_R\neq\emptyset} \ps{\sum_{Q_j \in \dW_R \atop Q_{j}\cap B_{Q}\neq\emptyset}\ps{\frac{\ell(Q_{j})}{\ell(Q)}}^{d+q} }^{\frac{2}{q}}|Q|\\
& \stackrel{\eqref{e:form20}}{\lesssim} (\ve\lambda)^{2}\sum_{Q\subseteq R\atop Q\cap E_R\neq\emptyset} \sum_{Q_j \in \dW_R \atop Q_{j}\cap B_{Q}\neq\emptyset}\ps{\frac{\ell(Q_{j})}{\ell(Q)}}^{\frac{2}{q}(d+q)}|Q|\\
& \lec (\ve\lambda)^{2} \sum_{Q_j \in \dW_R \atop Q_{j}\cap B_{R}\neq\emptyset}\sum_{Q\subseteq R\atop B_Q\cap Q_j\neq\emptyset, Q\cap E_R\neq\emptyset }\frac{\ell(Q_{j})^{2\ps{\frac{d}{q}+1}}}{\ell(Q)^{2\ps{\frac{d}{q}+1}-d}}\\
& \lec(\ve\lambda)^{2}  \sum_{Q_j \in \dW_R \atop Q_{j}\cap B_{R}\neq\emptyset}|Q_j|
 \lec (\ve\lambda)^{2} |R|.
\end{align*}
The last inequality follows from the fact that $q< \frac{2d}{d-2}$ (and recall $d>1$) which implies
\[
2\ps{\frac{d}{q}+1}-d
>2\ps{\frac{d(d+2)}{2d}+1}-d
=2\ps{\frac{d}{2}-1+1}-d=0
\]
and that, if $B_Q\cap Q_{j}\neq\emptyset$ and $Q\cap E_R\neq\emptyset$, then $\ell(Q)\gec \ell(Q_{j})$, and there can only be boundedly many such cubes $Q$ of the same size, so the interior sum on the penultimate line is essentially a geometric series. \\

\noindent
Combining these estimates, we get 
\begin{equation*}
\int_{E^1_R}\sum_{\wh{R} \supseteq Q\ni x}   \Omega_{f}^{q}(B_{Q},A_{Q}^{F})^{2}
\lec   (\ve\lambda)^{2}|R| \lec (\ve\lambda)^{2}|E^1_R|.
\end{equation*}
This finishes the proof of Lemma \ref{l:aux-lemma-D-I}. 
\end{proof}

\newpage

\begin{center}
	\part{}\label{p:part2}
	
	\Large\bfseries{Proof of Theorem \ref{t:main-B}.\\ The square function controls the tangential gradient.}
\end{center}

\section{Detour on tangential gradients and other preliminaries}\label{s:detour}

 We introduce tangential differentiability and the tangential gradient.
	Let $E \subset \R^n$ be a $d$-rectifiable set. If $x \in E$, recall that the approximate tangent space $T_xE$ is the $d$-dimensional linear subspace parallel to the approximate tangent to $E$ at $x$. The tangent space $T_xE$ then exists for $\dH^d$-almost all points in $E$. 
	\begin{definition}[Tangential differentiability and tangential gradient]\label{d:tan-grad}\
		\begin{itemize}
		\item We say that $f: \R^n \to \R$ is \textit{tangentially differentiable} with respect to $E$ at $x$ if the restriction of $f$ to $x+T_x E$ is differentiable at $x$. That is to say, there exists a continuous linear map $d_xf: T_x E \to \R$ so that
		\begin{align}\label{e:differentiability}
			\lim_{y \to x} \frac{ f(x) - f(y) - d_xf (y-x)}{|y-x|} =0.
		\end{align} 
		Then $d_x f$ is uniquely defined by the formula
		\begin{align*}
			\lim_{t \to 0} \frac{ f(x+t v) - f(x)}{t} = d_xf(v) \,\mbox{ if }\, v \in T_xE. 
		\end{align*}
		\item If $E \subset \R^n$ is $d$-rectifiable, $x \in E$, $f:\R^n \to \R$ is tangentially differentiable with respect to $E$ at $x$, its gradient at $x$ is the vector $\grad_t f(x)$ characterised by the condition that
		\begin{align}
			\langle \grad_t f(x), v\rangle = d_x f (v) \mbox{ for each } v \in T_x E. 
		\end{align} 
	\end{itemize}
\end{definition}
\begin{remark}\label{r:tangential-derivative}
	If $f : E \to \R$ is Lipschitz, then, taking a Lipschitz extension $F$ of $f$ to $\R^n$, we define
	\begin{align}
		\grad_t f := \grad_t F.
	\end{align}
	This definition is independent of the extension. This follows from \cite[Lemma 11.5]{maggi2012sets}. 
\end{remark}

\noindent
The following theorem below is the fact that Rademacher's theorem holds for Lipschitz functions defined on a $d$-rectifiable subset of $\R^n$.
\begin{theorem}\label{t:tangential-diffbar}
	Let $E \subset \R^n$ be $d$-rectifiable and $f: E \to \R$ be Lipschitz. Then $d_x f$, and so $\grad_t f$, exists for $\dH^d$-almost every point in $E$. 
\end{theorem}
\noindent
A proof can be found in \cite[3.2.19]{federer2014geometric}, \cite[Theorem 2.90]{ambrosio2000functions} or \cite[Theorem 11.4]{maggi2012sets}. When taking limits, we will need the following two facts.
\begin{lemma}[{\cite[Theorem 1.4(T2)]{villa2020tangent}}]\label{l:tan1}
	Let $E \subset \R^n$ be $d$-rectifiable and lower $d$-regular with also $\dH^d(E)<+\infty$. Let $x\in E$ be so that $E$ has a tangent at $x$, denoted by $V_x$. Then 
	\begin{align}\label{e:beta-tan-0}
		\lim_{r \to 0} \frac{1}{r^d} \int_{B(x,r) \cap E} \frac{\dist(y, V_x)}{r} \, d \dH^d(y) = 0. 
	\end{align}
\end{lemma}
\begin{lemma}[{\cite[Lemma 2.2]{mourgoglou2021regularity}}]
	Let $E \subset \R^{n}$ be $d$-uniformly $d$-rectifiable, and let $f:\R^n \to \R$ be $L$-Lipschitz. Then
	\begin{align}\label{e:gradtf-leq-L}
		|\grad_t f(x)| = \limsup_{y \to x \atop y \in E} \frac{f(y)-f(x)}{|x-y|} \leq L.
	\end{align}
\end{lemma}

\subsection{Some convergence lemmas}
\noindent
Throughout this subsection, assume that $E \subset \R^n$ is Ahlfors $d$-regular. For $x \in \Sigma$, set 
\begin{align}\label{e:def-L_x}
	L_xf(y) := f(x) + \langle \grad_t f(x), y-x\rangle.
\end{align}
\begin{lemma}\label{l:average-Lxf-to-0}
	Let $E \subset \R^n$ be $d$-rectifiable and Ahlfors $d$-regular. Suppose that $f: E \to \R$ is Lipschitz. Then 
	\begin{align*}
		\lim_{r\to 0} \avint_{B(x,r)}  \frac{|f(y)-L_x f(y)|}{r} \, d \sigma(y)=0.
	\end{align*}
\end{lemma}
\begin{proof}
Denote by $\pi_x: \R^n \to x+ T_x E$ the orthogonal projection onto $x+T_x E$. Then we have 
\begin{align}\label{e:form600}
	|f(x) - L_xf(y)|& =|\langle \grad_t f(x), y-x\rangle| \nonumber\\
	& \leq |\langle \grad_t f(x), y-\pi_x(y) \rangle| + |\langle \grad_t f(x), \pi_x(y)-x \rangle|.
\end{align} 
Since $f$ is Lipschitz, say with constant $L$, and $E$ is Ahlfors regular, we have 
\begin{align*}
	\avint_{B(x,r)} \frac{|\langle \grad_t f(x), y-\pi_x(y) \rangle|}{r} \, d\sigma(y) \lesssim_L \avint_{B(x,r)} \frac{\dist(y, x+T_x E)}{r} \, d\sigma(y) \stackrel{\eqref{e:beta-tan-0}}{\to} 0\, \,  \mbox{ as }\, r \to 0.
\end{align*}
On the other hand, since $\pi_x$ is $1$-Lipschitz, we see that for any $\ve>0$, if $r>0$ is sufficiently small, then 
\begin{align*}
	\avint_{B(x,r)} \frac{|\langle \grad_t f(x), \pi_x(y)-x \rangle|}{r} \, d \sigma(y) =\avint_{B(x,r)} \frac{|f(x)-L_xf(\pi_x(y))|}{r} \, d \sigma(y) \stackrel{\eqref{e:differentiability}}{<} \ve.
\end{align*}
\end{proof}
\noindent
Now, recall from Lemma \ref{l:A_B} that, given a ball $B$ centered on $\Sigma$, we can find an affine map $A_B:\R^n \to \R$ so that 
\begin{align}\label{e:gradAB-beta}
	|\grad A_B| \beta_{\Sigma}^{d,q}(B) \lesssim \gamma_f^q(B)
\end{align}
and so that $A_B \circ \pi_B = A_B$, where $\pi_B$ denotes the orthogonal projection onto the plane $P_B$ infimising $\beta_\Sigma^{d,q}(B)$. 
\begin{lemma}\label{l:gradAB-fLip}
	Let $E \subset \R^n$ be Ahlfors $d$-regular and $d$-rectifiable. Let $f: E \to \R$ be $L$-Lipschitz. Then for each ball $B$ centered on $E$, we have
	$
		|\grad A_B| \lesssim L. 
	$
\end{lemma}
\begin{proof}
	Note that, since $A\equiv f_B$ is affine and has $\grad f_B=0$, then, for $1\leq q<\infty$,
	\begin{align*}
		\gamma_f^q(B) \leq \Omega_f^q(B, f_B)\leq L.
	\end{align*}
	Fix some numerical constant $c<1$. If $\beta_\Sigma^{d,q}(B)> c$, then 
	\begin{align*}
		|\grad A_B| \leq c^{-1} |\grad A_B| \beta_{\Sigma}^{d,q} (B) \stackrel{\eqref{e:gradAB-beta}}{\lesssim}_c \gamma_f^q(B) \lesssim L.
	\end{align*}
For the case where $\beta_\Sigma^{d,q}(B)< c$, the proof is as in Lemma \ref{l:A-A}, and specifically for Case 2 of \eqref{e:A<gamma+av}. Note that there, the Poincar\'{e} inequality is only used to bound a term of the form $\avint_B \frac{|f(x)-f_B|}{r_B} \, d \sigma(x)$, which now can be bounded simply by $L$, since $f$ is $L$-Lipschitz. We leave the details to the reader. 
\end{proof}

\begin{lemma}\label{l:gammaBQj-to-0}
	Let $E \subset \R^n$ be $d$-rectifiable and lower $d$-regular. Let $f: E \to \R$ be $L$-Lipschitz and $x \in E$ so that $T_x E$ exists. Suppose that $\{Q_j\} \subset \dD_E$ is a sequence of cubes so that $x \in Q_j$ for each $j$ and $\ell(Q_j) \to 0$ as $j \to \infty$. Then if $M \geq 1$ and $1\leq q\leq \infty$,  
	\begin{align}\label{e:gamma-to-0}
	\lim_{j \to \infty} \gamma_f^q(MB_{Q_j}) =0.
	\end{align}
\end{lemma}
\begin{proof}
	Note that 
	$
		\gamma_f^q(MB_{Q_j}) \leq \gamma_f^q(MB_{Q_j}, L_x f)
		 $.
		Using Lemma \ref{l:average-Lxf-to-0}, it follows that for any $\ve>0$, there is a $j$ sufficiently large so that
		\begin{align*}
			\Omega_f^q(MB_{Q_j}, L_x f) = \left( \avint_{MB_{Q_j}} \left(\frac{|f(y) - L_xf(y)|}{\ell(Q_j)} \right)^q \, d \sigma(y) \right)^{\frac{1}{q}} < \ve.
		\end{align*}
	This is clear if $q=1$. If $q>1$, note that $|f(y)-L_xf(y))|\lesssim_L \ell(Q_j)$. Thus 
	\begin{equation*}
	\Omega^q_f(MB_{Q_j}, L_x f)^q \lesssim \Omega^1_f(MB_{Q_j}, L_x f)^q \lesssim_L \ve^q.
	\end{equation*} 
	On the other hand, note that $|\grad L_x f(y)| \lesssim |\grad_t f(x)| \leq L$. To finish the proof of the lemma is then enough to show that
	\begin{equation}\label{e:beta-alpha-to-0} \beta_E^{d, q}(MB_{Q_j}) \to 0 \quad  \mbox{ and } \quad  \alpha_\sigma^d(MB_{Q_j}) \to 0
\end{equation}
 as $j \to \infty$. To see the first, by \cite[Theorem 1.4 (T2))]{villa2020tangent} (which can be applied because we assume lower regularity), we see that $\beta_{E}^{d, \infty}(MB_{Q_j}) \leq \beta_E^{d, \infty}(MB_{Q_j}, T_x E) \to 0$. Since clearly $\beta_{E}^{d,q} \leq \beta_{E}^{d,\infty}$ for any $q\geq 1$, we are done with the decay of the $\beta$-coefficients. On the other hand, that $x+T_xE$ is the approximate tangent at $x$, implies that $T_{x,r}[\sigma]r^{-d} \warrow c \dH^d|_{x+T_x E}$ (see for example \cite[Theorem 4.8]{delellis2008rectifiable}). But that a sequence of Radon measures $\mu_j \warrow \mu$ implies that $F_B(\mu_j, \mu) \to 0$ for all balls $B$ (see \cite[Lemma 14.13]{Mattila}). Hence the weak convergence of $T_{x,r}[\sigma]r^{-d}$ to $c \dH^d|_{x + T_x E}$ in fact implies that also $\alpha_\sigma^d(MB_{Q_j}) \to 0$.
\end{proof}

\begin{proposition}\label{p:grad-lim-lem}
Let $\Sigma\subset \R^n$ be uniformly $d$-rectifiable. Let $M>1$ be sufficiently large, $f:\Sigma\rightarrow  \R$ an $L$-Lipschitz and $Q_0\in \dD(\Sigma)$. The following holds for $\sigma$-a.e. $x\in Q_0$: if  $\{Q_j\}_{j=1}^\infty \subset \dD(E)$ is a sequence of cubes so that $x \in Q_j$ for each $j$, and $\ell(Q_j^x)\downarrow 0$, then
\begin{equation*}
\lim_{j \to \infty}  \left\| \frac{A_{MB_{Q_{j}^x}}-L_x f}{\ell(Q_{j})}\right\|_{L^{\infty}(MB_{Q_{j}^x})}=0.
\end{equation*}
In particular, $\grad A_{MB_{Q_{j}}}\rightarrow \grad_t f(x)$ for $\sigma$-a.e. $x \in Q_0$.
\end{proposition}
\noindent
Recall that $A_{MB_{Q_j}}$ is the affine map given in Lemma \ref{l:A_B}. To prove this proposition, we will first need the two lemmas below.

\begin{lemma}\label{l:Aj-f}
	Same hypotheses and notation as in Proposition \ref{p:grad-lim-lem}. Let $A_j^x:= A_{MB_{Q_j}}$. For $\sigma$-almost all $x \in E$, we have 
	\begin{equation}\label{e:Aj-f-lim}
		\lim_{j \to \infty} \sup_{z \in B(x, 3\ell(Q_j)) \cap E} \frac{|A_j^x (z) - f(z)|}{\ell(Q_j)} = 0. 
	\end{equation}
\end{lemma}
\begin{proof}
	Since $f$ is Lipschitz by Lemma \ref{l:gradAB-fLip} we have  
	$
	| \grad A_j^x |  \lesssim L. 
	$
	The following estimate is similar to the proof of \cite[Eq. III (4.4)]{david-semmes93}. Set $B_j^x:= B(x,3 \ell(Q_j))$. Let $z \in 3B_j^x \cap E$.
	For each $0<\lambda<1$ and $y \in  B(z, 3\lambda \ell(Q_j))\cap E=: 3\lambda B^z_j\cap E$, we write
	\begin{align*}
		|A_j^x (z) - f(z)| &  \leq |f(z) - f(y)| + |A_j^x(z)- A_j^x(y)| + |A_j^x(y)- f(y)| \\
		& \lesssim L \lambda \ell(Q_j) + |A_j^x(y) - f(y)|.
	\end{align*}
	Averaging over $3\lambda B_j^z$, we get 
	\begin{align*}
		|A_j^x(z) - f(z)| &= \left(\avint_{3\lambda B_j^z} |A_j^x(z) - f(z)|^q\, d \sigma(y) \right)^{\frac{1}{p}}\\
		& \lesssim L\lambda\ell(Q_j) + \left( \frac{1}{\lambda^d \ell(Q_j)^d} \int_{3\lambda B_j^z} |A_j^x (y) - f(y)|^q \, d\sigma(y) \right)^{\frac{1}{q}}\\
		& \lesssim L \lambda \ell(Q_j) + \lambda^{-d/q} \Omega_f^q(3B_j^x, A_j^x) \ell(Q_j).
	\end{align*}
	If $\Omega_f^q(3B_j^x, \ell(Q_j))\leq 1$, then we pick $\lambda=\Omega_f^q(3B_j^x, \ell(Q_j))^{q/(d+q)}$ and so we obtain
	\begin{equation*}
		|A_j^x(z)-f(z)|/\ell(Q_j) \lesssim \Omega_f^q(3B_j^x,A_j^x)^{\frac{q}{d+q}}.
	\end{equation*}
	If, on the other hand, $\Omega_f^q(3B_j^x, A_j^x) > 1$, then 
	$
	\sup_{z \in 3B_j^x \cap \Sigma} |f(x)-A_j^x(x)|/\ell(Q_j) \lesssim 1 \lesssim \Omega_f^q(3B_j^x, A_j^x)^{\frac{q}{d+q}}.
	$
	If $M$ is sufficiently large, we obtain
	\begin{equation}\label{e:Ajx-fx}
		\left(|A_j^x(z) - f(z)|/\ell(Q_j)\right)^{\frac{d+q}{q}} \stackrel{\eqref{e:ABbeta<gammaAB<gamma}}{\lesssim} \gamma_f^q(MB_{Q_j}) \stackrel{\eqref{e:gamma-to-0}}{\to} 0.
	\end{equation}	
\end{proof}

\begin{lemma}\label{l:gradA-gradf}
Same hypotheses and notation as in Proposition \ref{p:grad-lim-lem}. For $x \in Q_0$ and $j \in \mathbb{N}$, set $A_j^x := A_{MB_{Q_j}}$. For $\sigma$-almost all $x \in E$ we have
\begin{equation}\label{e:A_j-gradf}
    \lim_{j \to \infty} \left| \nabla A_j^x - \nabla_t f (x)\right| = 0.
\end{equation}
\end{lemma}
\begin{proof}
	Let $\mathcal{L}_j := \mathcal{L}_{MB_{Q_j}}$, recalling the notation set in \eqref{e:def-L}: here $\mathcal{L}_{MB_{Q_j}}= c \dH^d|_{P_{Q_j}}$, where $c, P_{MB_{Q_j}}$ minimise $\alpha_\sigma^d(MB_{Q_j})$. Set also $P_j :=P_{MB_{Q_j}}$, and define $P_j^x$ to be the affine $d$-plane parallel to $P_j$ containing $x$, and $P_j^0= P_j^x - x$. Because $E$ is uniformly rectifiable, $\alpha_\sigma^d(MB_{Q_j}) \to 0$ as $j \to \infty$, as was established in the proof of Lemma \ref{l:gammaBQj-to-0}. This in particular implies that
	\begin{equation}\label{e:angle-to-0}
		\angle (P_j, x+T_x E) = \angle (P_j^0, T_x E) \to 0 \quad \mbox{ as } \quad  j \to \infty
	\end{equation}
	Note also that 
	\begin{equation}\label{e:form605}
		\dist(P_j, P_j^x)/\ell(Q_j) \leq \beta_E^{d,\infty}(MB_{Q_j}) \to 0
	\end{equation}
	when $j \to \infty$. This was also established in the proof of Lemma \ref{l:gammaBQj-to-0}. For each $j$, we compute
	\begin{align}\label{e:form19213}
		|(\grad A_{j}-\grad_t f(x))| 
	&	= \sup_{|y|=1} |\langle \grad A_{j}-\grad_t f(x), y\rangle |  \notag \\
		& \leq  \left(\sup_{y\in P_{j}^0\atop |y|=1} |\langle \grad A_{j}-\grad_t f(x)), y\rangle | +\sup_{y\in (P_{j}^0)^{\perp}\atop |y|=1} |\langle \grad A_{j}-\grad_t f(x)), y\rangle | \right)\notag  \\
		& =:I_1(j)+I_2(j).
	\end{align}
	We will compute the two limits $\lim_{j \to \infty} I_i(j), i = 1,2$ separately.\\ 
	
	\noindent
	We start with $I_1(j)$.
	Let $\phi$ be a $C^{\infty}$-bump function that is identically $1$ on $\bB$ and $0$ outside $2\bB$ and let 
	\[
	\phi_{j}(y) = \ell(Q_{j})^{-d}\phi_{j}\left(\frac{y-x}{\ell(Q_{j})}\right).\]
	Note that $y \mapsto \langle \grad A_j^x- \grad_t f(x), y\rangle$ is Lipschitz with constant depending only on $L$ - this can be seen using Lemma \ref{l:gradAB-fLip}; in particular the Lipschitz constant is uniform in $j$. Hence
	\begin{align*}
		I_1(j) & \lesssim \int_{\bB} |\langle \grad A_{j}^x  -\grad_t f(x),  y \rangle |  \phi(y)  d\dH^{d}|_{P_{j}^0}(y)\\
		& =  \int \langle \grad A_{j}^x -\grad_t f(x)),   (y-x)/\ell(Q_{j}) \rangle \phi_{j}(y) d \, \dH^{d}|_{P_{j}^x}(y)\\
		&  = \frac{cL}{\ell(Q_j)^{d+1}} \int \Phi_j^x (y)\, d \dH^d|_{P_j^x}(y), 
	\end{align*}
	where
	\begin{equation*}
		\Phi_j^x(y):= \frac{1}{cL}\left\langle \grad A_{j}^x -\grad_t f(x),   y-x \right\rangle \phi\left(\frac{y-x}{\ell(Q_j)} \right)
	\end{equation*} 
	and $c \approx 1$ is a numerical constant appropriately chosen so that $\Phi_j^x$ is $1$-Lipschitz. That this can be done is easily checked. Moreover, for sufficiently large $M$, it is  supported on $MB_{Q_j}$. With this in mind, we compute
	\begin{align*}
		I_1(j) &\lesssim_L \ell(Q_j)^{-d-1}  \int \Phi_j^x(y) \left( d\dH^d|_{P_j^x} - d\dH^d|_{P_j} \right)(y) \\
		& \qquad \qquad + \ell(Q_j)^{-d-1}\int \Phi_j^x (y) \left(d\dH^d|_{P_j^x} - d \sigma \right)(y) + \ell(Q_j)^{-d-1} \int \Phi_j^x(y)\, d \sigma(y) \\
		& := I_{1,1}(j) + I_{1,2}(j) + I_{1,3}(j).
	\end{align*} 
	We see that
	\begin{equation}\label{e:limI11j}
		I_{1,1}(j) \lesssim \frac{1}{\ell(Q_j)^{d+1}} F_{MB_{Q_j}} (\dH^d|_{P_j^x}, \dH^d|_{P_j}) \stackrel{\eqref{e:form605}}{\to} 0 \mbox{ as } j \to \infty.
	\end{equation}
	Similarly
	\begin{equation}\label{e:limI12j}
		I_{1,2}(j) \leq \alpha_{\sigma}^d(MB_{Q_j}) \to 0, 
	\end{equation}
	as seen in the proof of Lemma \ref{l:gammaBQj-to-0}. We are left to estimate $I_{1,3}(j)$. We see that
	\begin{align*}
	I_{1,3}(j)& =	\ell(Q_j)^{-d} \int \ell(Q_j)^{-1}\left| \left\langle \grad A_j^x, \, y-x \right\rangle - \langle  \grad_t f(x), y-x \rangle \right| \phi\left( \tfrac{y-x}{\ell(Q_j)} \right)\,  d \sigma(y)\\
	& \leq \ell(Q_j)^{-d} \int \ell(Q_j)^{-1}\left| \langle \grad A_j^x, y-x \rangle + f(x) - f(y) \right| \phi\left( \tfrac{y-x}{\ell(Q_j)} \right)\, d\sigma(y) \\
	& \qquad \qquad + \ell(Q_j)^{-d} \int \ell(Q_j)^{-1}\left| f(y) - L_x f(y) \right| \phi\left( \tfrac{y-x}{\ell(Q_j)} \right)\, d \sigma(y)   =: I_{1,3,1}(j) + I_{1,3,2}(j). 
	\end{align*} 
By Lemma \ref{l:average-Lxf-to-0}, we immediately see that 
\begin{equation}\label{e:form606}
	\lim_{j \to \infty} I_{1,3,2}(j) = 0. 
\end{equation}
As for $I_{1,3,1}(j)$, we see that $\langle \grad_j^x, y-x \rangle = A_j^x (y) - A_j^x(x)$, and so 
\begin{align}\label{e:lim131j}
	\lim_{j \to 0} I_{1,3,1}(j) & \leq \lim_{j\to \infty} \ell(Q_j)^{-d}\int |f(x) - A_j^x(x)|/\ell(Q_j) \,  \phi\left( \tfrac{y-x}{\ell(Q_j)} \right) \, d\sigma(y) \notag \\
	& \qquad \qquad + \lim_{j\to \infty}\ell(Q_j)^{-d} \int |f(x) - A_j^x(x)|/\ell(Q_j)\, \phi\left( \tfrac{y-x}{\ell(Q_j)} \right) \, d\sigma(y) \stackrel{\eqref{e:Aj-f-lim}}{=} 0
\end{align}
Putting together \eqref{e:limI11j}, \eqref{e:limI12j}, \eqref{e:form606} and \eqref{e:lim131j}, we see that $I_1(j) \to 0$ as $j \to \infty$.  \\

	\noindent
	Now we focus on $I_2(j)$, as defined in \eqref{e:form19213}. Recall from Lemma \ref{l:A_B} that $A_{j}\circ \pi_{P_j}=A_j$, and so $\grad A_{j}\in P_{j}^0$. This and the fact that $\grad_t f\in T_x E$ imply, for $y \in (P_{j}^0)^{\perp}$,
	\begin{align*}
		|\langle A_j(x)-\grad_t f(x),  y \rangle | 
		&  = |\langle \grad_t f(x),  y\rangle|
		= |\langle \grad_t f(x),  \pi_x(y) \rangle|
		\leq |\grad_t f(x)| |\pi_x( y)| \\ 
		& =|\grad_t f(x)| |(\pi_x-\pi_{P_{j}^0})(y)|,
	\end{align*} 
where here $\pi_x$ is the orthogonal projection onto $T_x E$ and $\pi_{P_j^0}$ the orthogonal projection onto $P_j^0$. Then we see that
	\begin{align*}
		I_2(j)
		&  \leq \sup_{y\in (P_{j}^0)^{\perp}\atop |y|=1}|\grad_t f(x)| \cdot  |(\pi_x-\pi_{P_{j}^0})(y)|  \\
		& \qquad \leq |\grad f(x)| \left( \angle(T_x E,P_j^0)\right) \stackrel{\eqref{e:gradtf-leq-L}}{\leq} L \left(\angle(T_x E, P_j^0)\right) \stackrel{\eqref{e:angle-to-0}}{\rightarrow }0.
	\end{align*}
	Thus we see that $\lim_{j \to \infty}I_{i}(j) = 0$, $i=1,2,$. This proves \eqref{e:A_j-gradf} and finishes the proof. 
\end{proof}

\begin{proof}[Proof of Propostion \ref{p:grad-lim-lem}]

By \eqref{e:Ajx-fx}, for any $\ve>0$, we can find a $j$ so that $\tfrac{|A_j^x(x) - f(x)|}{\ell(Q_j)}< \ve$. For this $j$ and an arbitrary $y \in MB_{Q_j}$, and using the Taylor expansion $A_j^x(y) = A_j^x(x) + \langle \grad A_j^x(x), y-x \rangle$, we compute
\begin{align*}
 \frac{| A_{j}^x(y)-L_xf(y)|}{\ell(Q_{j})} & \leq  \frac{\left|\langle\nabla A_j^x(x) - \nabla_t f(x), x-y\rangle\right|}{\ell(Q_j)} + \frac{|A_j^x(x) - f(x)|}{\ell(Q_j)}  \\
& \stackrel{\eqref{e:Ajx-fx}}{<} \ve + \frac{| \langle\grad A_{j}^x-\grad_t f(x), x-y\rangle|}{\ell(Q_{j})}
\end{align*}
Since $x \in Q_j$, $|x-y| \lesssim_M \ell(Q_j^x)$. Thus 
\begin{align*}
     \frac{|A_j(y) - L_x f(y)|}{\ell(Q_j)} \lesssim_M \ve + |\nabla A_j^x - \nabla_t f(x)|.
\end{align*}
Since this holds for any $y \in MB_{Q_j} \cap E$ uniformly, we can take the supremum over such $y$'s. Then letting $j \to \infty$, and applying Lemma \ref{l:gradA-gradf}, we conclude the proof of Proposition \ref{p:grad-lim-lem}.
\end{proof}

\section{The proof of Theorem \ref{t:dorronsoro-II} via a good-$\lambda$ inequality}\label{s:tB-proof-via-good}
Recall that $\dG^1 f \lesssim \dG^q f$ for any $q \geq 1$, and also $\dG^1 f \lesssim \widetilde{\dG} f$. Hence, to verify Theorem \ref{t:main-B}, it suffices to prove the theorem below.
\begin{theorem}
	\label{t:dorronsoro-II}
	For $1<p<\infty$, let $\Sigma\subseteq \R^{n}$ be an uniformly $d$-rectifiable and $f: E \to \R$ be Lipschitz. Then,
	\begin{equation}
		\label{e:||f||<||G||}
		\|\grad_t f\|_{L^p(\Sigma)}
		\lec \|\G^{1}f\|_{L^p(\Sigma)}. 
	\end{equation}
\end{theorem}

\subsection{First reductions and the proof of Thereom \ref{t:dorronsoro-II}}

Let $M$ be a large constant; we will adjust its value as we go along. We denote by $L$ the Lipschitz constant of $f$, and we will be careful that our estimates do not depend on it. For $Q\in\dD(E)$, let $A_{Q} = A_{MB_{Q}}$, where $A_{MB_Q}$ is the affine map from Lemma \ref{l:A_B}.  For $\lambda>0$,
\begin{equation}\label{e:def-Clambda}
    \mbox{let } \cC_\lambda = \cC_{\lambda}(f) \mbox{ be a family of maximal cubes } R\in \dD \mbox{ such that } |\grad A_{R}|>\lambda.
\end{equation} 
Recall from Lemma \ref{l:gradAB-fLip}, that $|\grad A_R| \lesssim L$ for any $R \in \dD(E)$. We conclude that $\cC_\lambda$ is well defined. Now set\footnote{We will usually write just $E_\lambda$ because the function $f$ will be fixed. We write $E_\lambda (f)$ whenever we are dealing with two distinct functions.}
\begin{equation}\label{e:E-lamdba2}
E_{\lambda}: = E_\lambda(f) :=\bigcup_{R\in \cC_{\lambda}}R.
\end{equation}
Let $x \in E$. 
\begin{itemize}
	\item If there is a cube $Q \in \dD(E)$ so that $Q \ni x$ and $|\grad A_Q|> \lambda$, then $x \in E_\lambda$, simply because then either $Q$ or an ancestor of $Q$ belongs to $\cC_\lambda$.
	\item By Lemma \ref{l:gradA-gradf}, it holds for $\sigma$-almost all $x \in E$ that, if we have a sequence of cubes $Q_j$ all containing $x$ and such that $\ell(Q_j)\to 0$ as $j\to \infty$, then $\lim_{j \to \infty} |\grad A_{Q_j}(x) - \grad_t f(x)| = 0$ (whenever $f$ is Lipschitz).
	\item The previous two bullet points imply that
	for $\sigma$-almost every $x\in \Sigma$, if $0<\lambda< |\grad_t f(x)|$, then $x\in E_{\lambda}$. In particular
	\begin{equation}\label{e:form22}
		|\grad_t f(x)|\leq \sup\{\lambda>0: x\in E_{\lambda}\}.
	\end{equation}
\end{itemize} 
Note that $E_{\lambda_2} \subseteq E_{\lambda_1}$ whenever $\lambda_1 \leq \lambda_2$. Thus, if $t>0$, 
\begin{equation}\label{e:form23}
\{x:  \sup\{\lambda>0: x\in E_{\lambda}\}>t\}
\subseteq E_{t}.
\end{equation}
We conclude that
\begin{align}
\int |\grad_t f|^{p}d\sigma
& =\int_{0}^{\infty}|\{x: |\grad_t f|>t\}t^{p-1}dt \notag \\
& \stackrel{\eqref{e:form22}}{\leq} \int_{0}^{\infty}|\{x:  \sup\{\lambda>0: x\in E_{\lambda}\} >t\}t^{p-1}dt \notag \\
& \stackrel{\eqref{e:form23}}{\leq} \int_{0}^{\infty}|E_{t}|t^{p-1}dt. \label{e:form21}
\end{align}
So, to prove Theorem \ref{t:dorronsoro-II}, it suffices to bound the right hand side in \eqref{e:form21}.

 \begin{proposition}[Good-$\lambda$ inequality]
 \label{p:EcapG<E}
Hypotheses as in Theorem \ref{t:dorronsoro-II}. There is $\theta>0$ so that for each $\kappa>1$ sufficiently close to 1, there is $\ve = \ve(\kappa)>0$ small enough \textup{(}depending on $\kappa$\textup{)}, such that the following holds. For any $\lambda>0$, set
 \begin{align*}
 G_{\ve \lambda}^q := G_{\ve \lambda}^q(f) := \left\{x \in E \, |\, \G^q f(x) \leq \ve\lambda\right\} .
 \end{align*}
Then
 \begin{equation}
 \label{e:Ecap G<E}
 |E_{\kappa \lambda}\cap G_{\ve\lambda}^q|<\theta |E_{\lambda}|.
 \end{equation}
 \end{proposition}
 
 \begin{proof}[Proof of Theorem \ref{t:dorronsoro-II}]
 Let us explain how to finish the proof assuming this lemma. For $\lambda>0$,
 \begin{align*}
\kappa^{-p}  \int_{0}^{\infty}|E_{ \lambda}| \lambda^{p-1}d\lambda
& =   \int_{0}^{\infty}|E_{ \kappa \lambda}| \lambda^{p-1}d\lambda
 \\
& \leq  \int_{0}^{\infty} (|E_{\kappa \lambda}\cap G_{\ve\lambda}^q| +|(G_{\ve\lambda}^q)^{c}|)\lambda ^{p-1}d\lambda \\
& \stackrel{\eqref{e:Ecap G<E}}{\leq}   \int_{0}^{\infty} \theta |E_{ \lambda}|\lambda^{p-1}\, d \lambda+ \ve^{-p}\|\G_q f\|_{L^p(\Sigma)}^{p}
\end{align*}
and so for $\kappa>1$ such that, say, $\kappa^{p}-\theta > \tfrac{1}{2}$, 
\[
\|\grad_t f\|_{L^p(\Sigma)}^{p}
\stackrel{\eqref{e:form21}}{\leq}  \int_{0}^{\infty}|E_{ \lambda }| \lambda ^{p-1}d\lambda
\leq \frac{\ve^{-p}}{\kappa^{p}-\theta} \|\G_q f\|_{L^p(\Sigma)}^{p} \leq 2 \ve^{-p} \|\G_q f\|_{L^p(\Sigma)}^p.
\]

This proves the theorem, so now we focus on proving Proposition \ref{p:EcapG<E}.
\end{proof}

\section{The stopping time procedure and the proof of the good-$\lambda$ inequality}\label{s:tB-stopping-time}
In this subsection, we prove Proposition \ref{p:EcapG<E} with Main Lemma \ref{l:m2} (see below). We first prove an easy reduction.
\begin{lemma}
To prove Proposition \ref{p:EcapG<E}, we just need to verify \eqref{e:Ecap G<E} for one value of $\lambda>0$. 
\end{lemma}

\begin{proof}
Suppose we have shown \eqref{e:Ecap G<E} for some value $\lambda >0$.  Now let $\lambda'$ be any positive number. 
Since \eqref{e:Ecap G<E} holds with {\it any} Lipschitz function $f$, it also holds for $f$ replaced with $g=\frac{f\lambda }{\lambda'}$. Let 
$A_{Q}^{g} = \frac{\lambda}{\lambda'}A_{Q}$.
Then all our previous results regarding $f$ and the affine maps $A_{Q}$ also hold for the function $g$ and associated affine maps $A_{Q}^{g}$. Thus, if $E_{\lambda}(g)$ is constructed the same way as $E_{\lambda}(f)$ with $g$ in place of $f$ and $\lambda'$ in place of $\lambda$, we have that $E_{\lambda}(g) = E_{\lambda'}(f)$. Let 
\[
G_{\ve\lambda}^q(g) = \{x\in \Sigma \,|\,  \G^{q} g(x)\leq \ve \lambda\}.
\]
Since $\G^{q}g = \frac{\lambda}{\lambda'}\G^{q}f$ by Lemma \ref{l:gammacf=cgammaf}, we get that 
\[
G_{\ve\lambda}^{q}(g)  = G_{\ve\lambda'}^q(f)
\]
and so
\[
|E_{\kappa \lambda'}(f)\cap G_{\ve\lambda'}^q(f)|
=|E_{\kappa\lambda}(g) \cap G_{\ve\lambda}^{q}(g)|
<\theta |E_{\lambda}(g)|=\theta |E_{\lambda'}(f)|.
\]
This concludes the proof.
\end{proof}

\begin{remark}\label{r:lambda-fixed}
We now assume $\lambda>0$ is fixed, but we will indicate what value we need it to be later in the proof. 
\end{remark}
\noindent
Take $\lambda$ fixed as in Remark \ref{r:lambda-fixed}, let $\cC_\lambda$ the relative maximal family, as in \eqref{e:def-Clambda}, and $R\in \cC_{\lambda}$. A further immediate reduction is that it suffices to find a $\theta>0$ (not depending on $R$) so that for any $\kappa>0$ there is $\ve>0$ such that 
 \begin{equation}
 \label{e:EcapGcapR<thetaR}
 |E_{\kappa \lambda}\cap G_{\ve\lambda}^q\cap R|<\theta |R|.
 \end{equation}
This inequality will be our goal for the remainder of the section. \\

\subsection{Definition of the stopping cubes}\label{ss:stopping-cubes}

The proof of \eqref{e:EcapGcapR<thetaR} will be done via a stopping time argument. Before defining the stopping criteria, let us spell out some simple facts that will be extensively used later. 
\begin{itemize}
	\item First, note that
 \begin{align}
 \label{e:gamma<ve}
 \gammaq(M^2 B_{Q}) \stackrel{\eqref{e:gamma-monotone}}{\lec}_M & \G^q f(x)  \leq C(M) \ve \lambda \;\; \notag \\
  & \mbox{ for }\;\; Q\subseteq R, \;\; x\in Q\cap G_{\ve\lambda}^q\cap R \; \; \mbox{ and } \; R \in \cC_\lambda.
\end{align}
\item 
Let $\rho>0$, to be fixed later. Let us say immediately that the $\theta$ we are looking for will satisfy $\theta\in (1-\rho,1)$. If $|G_{\ve\lambda}^q\cap R|<(1-\rho)|R|$, then \eqref{e:EcapGcapR<thetaR} follows immediately, so assume
\begin{equation}
\label{e:GcapR>1/2R}
|G_{\ve\lambda}^q\cap R |\geq (1-\rho)|R|.
\end{equation}
This is the second fact, which in particular implies that
 $G_{\ve\lambda}^q\neq\emptyset$. 
 \item Now, since $R\in \cC_{\lambda}$, if $\wh{R}$ is the parent of $R$, we know that $|\grad A_{\wh{R}}|\leq \lambda$. Thus, for some $C>0$,
\begin{equation}
\label{e:gradA-simLambda}
\lambda < |\grad A_{R}|
\stackrel{\eqref{e:gradA-gradA}}{\leq} |\grad A_{\wh{R}}|+C\gammaq(M^2B_{\wh{R}})
\stackrel{\eqref{e:gamma<ve}}{<}|\grad A_{\wh{R}}|+C\ve\lambda
\leq (1+C\ve)\lambda.
\end{equation}
And this is the third fact. 
\end{itemize}

\vspace{1cm}

\noindent
We now define the stopping time criteria. Fix $\delta_0,C_0,\ve_0>0$, and set
\begin{align}\label{e:def-alpha_0}
    \alpha_0 := 2\kappa - 1> 1
\end{align}
and $\Stop(R)$ be those maximal cubes $Q$ which contain a child $Q'$ such that either
\begin{itemize}
\item $\sum_{Q'\subseteq T\subseteq R}\gammaq(M^2B_{T})^{2}\geq C_0(\ve\lambda)^{2}$, call these cubes $\BSF(R)$ (`big square function'), or
\item $\angle(P_{Q'},P_{R})\geq \delta_0$, and call these cubes $\BA(R)$ (`big angle'), or
\item  $|\grad A_{Q'}|<\alpha_0^{-1} |\grad A_{R}|$ or $|\grad A_{Q'}|>\alpha_0|\grad A_{R}|$, call these cubes $\BG(R)$ (`bad gradient'), or
\item $\ell(Q')<\ve_0\ell(R)$, call these cubes $\SSL(R)$ (`small side length'). 
\end{itemize}
Let $\Tree(R)$ be those cubes in $R$ that are not properly contained in a cube from $\Stop(Q)$. Then $\Tree(R)$ is a {\it stopping-time region}, meaning that is a collection of cubes satisfying the following:
\begin{enumerate}
\item It has a maximal cube $Q_{\Tree(R)}$ containing all cubes from $\Tree(R)$. In our case, the maximal cube is just $R$.
\item For all $Q\in \Tree(R)$ and $Q\subseteq T\subseteq Q_{\Tree(R)}$ we have $T\in \Tree(R)$. 
\item If $Q\in \Tree(R)$, then all siblings of $Q$ are in $\Tree(R)$.
\end{enumerate}

\noindent
A few remarks. First, note that every $x\in R$ is contained in a cube from $\Stop(R)$, since at the least every $x\in R$ is contained in a cube $Q$ with $\ell(Q)<\ve_0\ell(R)$.
Second, if $Q\in \Tree(R)$, then 
\begin{equation}
\label{e:gamma-small}
\Omega_{f}^{1}(M^{2}B_{Q})^{2}\leq \gammaq(M^2B_{Q})^2
\leq 
\sum_{Q\subseteq T\subseteq R}\gammaq(M^2B_{T})^{2}< C_0(\ve\lambda)^{2}, 
\end{equation}
\begin{equation}
\label{e:angle-S}
\angle(P_{Q},P_{R})< \delta_0,
\end{equation}
and 
\begin{equation}
\label{e:AQ<alphaAR}
\alpha_0^{-1}\lambda \stackrel{R \in \cC_\lambda}{<}\alpha_0^{-1} |\grad A_{R}|\leq |\grad A_{Q}|\leq \alpha_0|\grad A_{R}|
\stackrel{\eqref{e:gradA-simLambda}}{\leq} \alpha_0(1+C\ve)\lambda \leq 2\lambda .
\end{equation}
From \eqref{e:gamma<ve} and \eqref{e:gradA-simLambda}, note that the constant $C$ in \eqref{e:AQ<alphaAR} depends on $M$. Hence we can choose $\ve>0$ sufficiently small (depending on $M$ and $\alpha_0$), so that $\alpha_0(1+C\ve)\leq 2$. 

\subsection{Packing estimates for the stopping cubes, statement of Main Lemma \ref{l:m2}}
In this subsection, we prove that, with the criteria as above, our algorithm doesn't stop too often. Most packing estimates will be fairly easy, except the one right below, whose proof will require more work.
\begin{mainlemma}
\label{l:m2}
With notation as in Subsection \ref{ss:stopping-cubes} and assumptions as in Proposition \ref{p:EcapG<E}, we have the packing estimate
\begin{equation}
\label{e:m2}
\sum_{Q\in \BG(R)}|Q|\leq \frac{1}{4}|R|.
\end{equation}
\end{mainlemma}
\noindent
We now finish the proof of Proposition \ref{p:EcapG<E} via two further easy lemmas, which will be proven immediately. The last sections of Part II will be devoted to the proof of Main Lemma \ref{l:m2}.

\begin{lemma}\label{l:BSF-packing}
\begin{equation*}
 \sum_{Q \in \BSF(R)} |Q| \leq \frac{1}{4} |R|.
 \end{equation*}
\end{lemma}
\begin{proof}
 Let $Q \in \BSF(R)$. Then by definition, for each such a $Q$, there is a cube $Q' \in \Child(Q)$ so that $$\sum_{Q' \subseteq T \subseteq R} \gammaq(M^2 B_T)^2 \geq C_0 (\ve \lambda)^2.$$
 If we pick $C_0$ sufficiently large with respect to $M$, this gives us
\begin{equation*}
(\ve\lambda)^{2} 
\leq C_0^{-1} \sum_{Q'\subseteq T\subseteq R}\gammaq(M^2B_{T})^{2}
<  \sum_{T \ni x \atop T \subseteq R} \gammaq(B_{T})^{2} 
=\G_{q}f(x)^2.
\end{equation*}
This says that $Q'\subseteq R\backslash G_{\ve\lambda}$, and hence, since $E$ is Ahlfors regular and the cubes in $\Stop(R)$ are disjoint,  
\begin{equation*}
\sum_{Q\in \BSF(R)}|Q|
\lec \sum_{Q\in \BSF(R)}|Q'| \lec \left|\bigcup_{Q \in \BSF(R) } Q' \right| 
\lec |R\backslash G_{\ve\lambda}|
\stackrel{\eqref{e:GcapR>1/2R}}{\lec} \rho|R|
\end{equation*}
so that, for $\rho$ smaller than some absolute constant,
\[
\sum_{Q\in \BSF(R)}|Q|<\frac{1}{4}|R|.
\]
\end{proof}
\begin{lemma}\label{l:BA-packing}
 \begin{equation*}
     \sum_{Q \in \BA(R)} |Q| \leq \frac{1}{4} |R|.
 \end{equation*}
\end{lemma}
\begin{proof}

For $Q\in \Tree(R)$, and assuming $\alpha\leq 2$,
\begin{align}\label{e:a<2/lambdagamma}
& \beta_E^{d,1}(MB_{Q})
\stackrel{R \in \cC_\lambda}{\leq} \frac{|\grad A_{R}|}{\lambda} \beta_E^{d,1}(MB_{Q}) \notag \\
& \qquad \qquad \qquad
\stackrel{\eqref{e:AQ<alphaAR}}{\leq} \frac{\alpha |\grad A_{Q}|}{\lambda} \beta_E^{d,1}(MB_{Q})
\leq \frac{2}{\lambda} \gamma(MB_{Q}),
\end{align}
and so for any $T\in \Tree(R)$,
\begin{equation}
\label{e:jones-small}
\sum_{T\subseteq Q\subseteq R}\beta_E^{d,1}(MB_{Q})^2
\leq \frac{4}{\lambda^2} \sum_{T\subseteq Q\subseteq R} \gamma(MB_{Q})^2
\stackrel{\eqref{e:gamma-small}\atop \eqref{e:gamma-monotone}}{\lec_{C_0, M}} \ve^{2}.
\end{equation}
Thus,
\begin{align*}
\sum_{Q\in \Tree(R)}\beta_E^{d,1}(MB_Q)^2|Q|
& = \sum_{Q\in \Tree(R)} \beta_E^{d,1}(MB_{Q})^2\int_{R}\sum_{T\in \Stop(R)\atop T\subseteq Q}\one_{T}(x) \, d\sigma(x) \\
& =\int_{R}\sum_{T\in \Stop(R)}\sum_{Q\in \Tree(R) \atop Q\supseteq T}\beta_E^{d,1}(MB_{Q})^2\one_{T}(x)\, d\sigma(x)\\
& \stackrel{\eqref{e:jones-small}}{\lec_{C_0, M}} \ve^{2} \int_{R}\ps{\sum_{T\in \Stop(R)}\one_{T}(x)}d\sigma(x)\\
& \lec_{C_0,M} \ve^{2}|R|.
\end{align*}
By the results in \cite{david-semmes91} (specifically the claim at the beginning of Section 14 and equation (14.1)) and the above inequality, for $\ve>0$ small enough depending on $\delta_0>0, C_0$ and $M$, we can guarantee that 
\[
\sum_{Q\in \BA(R)}|Q|<\frac{1}{4}|R|.
\]
\end{proof}

\subsection{Proof of the good-$\lambda$ inequality \ref{p:EcapG<E}}
We can now end the proof of Proposition \ref{p:EcapG<E}
assuming Main Lemma \ref{l:m2}.
\begin{proof}[Proof of Proposition \ref{p:EcapG<E}]
	Recall that it suffices to prove \eqref{e:EcapGcapR<thetaR}. First, gathering Main Lemma \ref{l:m2}, Lemma \ref{l:BSF-packing} and Lemma \ref{l:BA-packing} We have that
\begin{align}
 & \left|\left(\bigcup_{Q \in \BSF(R)} Q\right) \cup \left(\bigcup_{Q \in \BA(R)} Q \right) \cup \left(\bigcup_{Q \in \BG(R)} Q \right)\right| \leq  \sum_{Q\in \Stop(R)\setminus \SSL(R)} |Q| 
\leq \frac{3}{4} |R| .
\end{align}
Since our stopping-time region $\Tree(R)$, and hence our minimal cubes $\Stop(R)$, depend on the parameter $\ve_0$. To highlight this, we denote $\Tree(R)$ as $\Tree_{\ve_0}(R)$ and also we adopt the notation $\BSF_{\ve_0}(R), \BA_{\ve_0}(R),$ $\BG_{\ve_0}(R)$ and $\Stop_{\ve_0}(R)$.

Let $\kappa>1$ be sufficiently close to $1$, as in the statement of Proposition \ref{p:EcapG<E}. We claim the following: suppose that $T \in \cC_{\kappa \lambda}= \cC_{\kappa \lambda} (f)$ (as defined in \eqref{e:def-Clambda}), that $T \subseteq R$ for some $R \in \cC_\lambda$ and $\ell(T)> \ve_0 \ell(R)$. Then $T \in \Stop(R)$. Let us prove this. By definition, if $T \in \cC_{\kappa \lambda}$ then $| \grad A_T| > \kappa \lambda$. But recall from \eqref{e:gradA-simLambda} that $\lambda < |\grad A_R| \leq (1+C\ve)\lambda$. Hence we can estimate $|\grad A_T|$ from below as follows:
\begin{equation*}
	|\grad A_T| > \kappa \lambda = \kappa \lambda \, \frac{1+C\ve}{1+C\ve} \geq \frac{\kappa}{1+C\ve} |\grad A_R|.
\end{equation*}
Now, choosing $\ve>0$ sufficiently small, $\kappa>1$ sufficiently close to $1$, and recalling that $\alpha_0=2\kappa-1$, we can insure that $\kappa/(1+C\ve) > \alpha_0$, and then we obtain 
\begin{equation*}
	|\grad A_T| > \alpha_0 |\grad A_R|.
\end{equation*}
 Since $\ell(T)>\ve_0\ell(R)$, this implies that $T\subseteq Q$ for some $Q\in\Stop_{\ve_0}(R)\setminus \SSL_{\ve_0}(R)$. Moreover, notice that the collections $\BSF_{\ve_0}(R),\BA_{\ve_0}(R),$ and $\BG_{\ve_0}(R)$ increase as $\ve_0\rightarrow 0$. Thus,
\[
|E_{\kappa\lambda}\cap R|
 \leq \left|\bigcup_{\ve_{0}>0}\bigcup_{Q\in \Stop_{\ve_0}(R) \setminus \SSL_{\ve_0}(R)} Q\right|
 =\lim_{\ve_0\rightarrow 0} \left|\bigcup_{Q \in \Stop_{\ve_0}(R)\setminus \SSL_{\ve_0}(R)} Q \right|
  \leq \frac{3}{4} |R| 
\]
We now pick $\max\{3/4,1-\rho\}<\theta<1$ and \eqref{e:EcapGcapR<thetaR} is proven. This concludes the proof of Lemma \ref{e:Ecap G<E} assuming Lemma \ref{l:m2}, which will be the focus of the next section. 

\end{proof}
\noindent
For the rest of the proof, we assume $\ve_0$ to be fixed and suppress it from the notation for $\Tree(R)$ and $\Stop(R)$.
\section{Proof of Main Lemma \ref{l:m2}: construction of the approximating Lipschitz graph}\label{s:tB-approx-Lipschitz}

We will finish the proof of Theorem \ref{t:dorronsoro-II} once we prove Main Lemma \ref{l:m2}, that is, once we prove the estimate
\begin{equation*}
\sum_{Q\in \BG(R)}|Q|\leq \frac{1}{4}|R|.
\end{equation*}

In this section we set up the standard scheme that is typically used in stopping time arguments to construct a Lipschitz graph (with small constant) which approximates our set $E$ well from the scale of the root cube of a stopping time region (in our case $\Tree(R)$, $R \in \cC_\lambda$), down to its minimal cubes ($\Stop(R)$ for us). 

To the authors' knowledge, this scheme first appeared in the work of David and Semmes \cite{david-semmes91} and of Leg\'{e}r \cite{leger1999menger}. For a neat presentation in the plane, we refer the reader to \cite{tolsa-book}. Our context is closest to that of \cite{david-semmes91}; for several proofs below we redirect the reader there, in particular to Chapter 8.  

\begin{remark}
 Sice $E$ is uniformly $d$-rectifiable, it admits  \textit{corona decomposition} in terms of Lipschitz graphs (by \cite[C4]{david-semmes91}). Then constructing a Lipschitz graph seems rather redundant. We will see below, however, that we will need to work \textit{with} this Lipschitz graph, and so using the aforementioned result as a black box would not be possible.
\end{remark}

\subsection{Smoothing procedure and Whitney cubes}
For $x\in \R^{n}$, let 
\begin{equation*}
d_{R}(x):= d_{\Tree(R)}(x):=\inf_{Q\in \Tree(R)}(\ell(Q)+\dist(x,Q)).
\end{equation*}
Without loss of generality we might (and will) assume that $P_{R}=\R^{d}\subset \R^n$. Let $\pi=\pi_{\R^{d}}$. Let $\dI$ be the family of dyadic cubes in $\R^d$. Elements of $\dI$ will generally be denoted using $I, J$. Define
\begin{equation*}
    D_{R}(x):= \inf_{y\in \pi^{-1}(x)} d_{R}(y) \, \, \mbox{ and } \, \, 
    D_{R}(I):=\inf_{y\in I}D_{S}(y) \mbox{ for } I \in \dI.
\end{equation*}
A fact that will play a role later is that both $d_R$ and $D_R$ are $1$-Lipschitz. Furthermore, by \cite[Lemma 8.21]{david-semmes91}, 
\begin{equation}
\label{e:dsimD}
d_{R}(x)\approx_M D_{S}\circ\pi (x) \;\; \mbox{ for all } x\in MB_{R}\cap E.
\end{equation}
\begin{remark}
Since we will cite \cite{david-semmes91} quite a bit, we provide a dictionary between their and our own notation. A stopping time region there is denoted by $S$. This corresponds to our $\Tree(R)$. Their $Q(S)$ (the root, or maximal, cube) is our $R$. They write $L = \diam(Q(S))$ (first paragraph, page 43). So, for us, $L \approx \ell(R)$. Their $k_0$ (\cite[Proposition 8.2]{david-semmes91}) plays the role of our $M$. They define $U_0= P \cap B(\pi(x_0), 2 k_0 L)$ (see \cite[pg. 45, line 7]{david-semmes91}), where $P=P_{Q(S)}$. For us $x_0$ will be $x_R$ (the center of $R$), which we may assume to be $x_R=0$. 
\end{remark}
\begin{remark}
In \cite[Chapter 8]{david-semmes91} they also define $Z$ to be the set of points that are contained in infinitely many cubes from  $\Tree(R)$ (see pg 43, last paragraph). However, we have defined our stopping-time region $\Tree(R)$ in such a way that every $x\in R$ is contained in a minimal cube, and so $Z=\emptyset$ in our situation. 
\end{remark}
\noindent
We will also need the following.
\begin{lemma}[{\cite[Lemma 8.4]{david-semmes91}}]\label{l:piperppi}
Let $\pi^{\perp}$ denote the projection into $\R^{n-d}$. Then for $x,y\in 2MB_{R}$ with $|x-y|\geq 10^{-3}\min\{d_{R}(x),d_{R}(y)\}$, 
\begin{equation}
\label{e:piperppi}
|\pi^{\perp}(x-y)|<2\delta_0 |\pi(x-y)|.
\end{equation}
\end{lemma}
\noindent
To prove this lemma, we will need the following easy fact.
\begin{lemma}\label{l:small-betainf}
 Let $Q \in \Tree(R)$. Then 
 \begin{align}\label{e:QinTree-beta-inf<ve}
     \beta_E^{d, \infty}(M^2B_Q, P_Q) \leq C \ve^{\frac{1}{d+1}},
 \end{align}
 where $C$ depends only on $C_0$, $M$ and $d$. 
\end{lemma}
\begin{proof}
Since $Q \in \Tree(R)$, then $\gammaq(M^2B_Q)< C_0 \ve \lambda$. This, by definition, implies that
\begin{align*}
    |\nabla A_Q| \beta_E^{d,q}(M^2B_Q) < C_0 \ve \lambda.
\end{align*}
Now, since $Q \in \Tree(R)$, then $|\nabla A_Q| \geq \alpha^{-1} |\nabla A_R|$, and since $R \in \cC_\lambda$, $|\nabla A_R| > \lambda$. Recalling also that $\alpha \leq 2$, we get
\begin{align*}
     \alpha_\sigma^d(M^2B_Q) < 2C_0 \ve.
\end{align*}
Now, by \eqref{e:beta<alpha} and \eqref{e:beta<beta}
\begin{align*}
    \beta_\Sigma^{d, \infty} (MB_Q) & \lesssim \beta_\Sigma^{d, 1}(MB_Q)^{\frac{1}{d+1}} \lesssim \alpha_\sigma^d(MB_Q)^{\frac{1}{d+1}} \\
    & \lesssim_{M,d} \alpha_\sigma^d(M^2B_Q)^{\frac{1}{d+1}} \lesssim_{M,d, C_0} \ve^{\frac{1}{d+1}}.
\end{align*}
\end{proof}
\begin{proof}[Sketch of the proof of Lemma \ref{l:piperppi}]
In \cite{david-semmes91}'s proof of this lemma they use the constant $K$ in place of $2M$, but it is easily adapted to our case: the proof relies on \cite[Equation (6.1)]{david-semmes91}, but Lemma \ref{l:small-betainf} takes care of this. However, for the proof in \cite{david-semmes91} to work, we need to choose $\ve$ small enough depending on $\delta_0$ (the angle parameter for $\BA$) (in fact, we need $\ve^{\frac{1}{d+1}}\ll \delta_0$). In David and Semmes' proof $\delta_0$ is in fact $\delta$.
\end{proof}

\noindent
Let $\dW_R \subset \dI$ be the maximal dyadic cubes in $\R^{d}$ so that  
\[
\diam (I)\leq \frac{1}{20}D_{R}(I).
\]
We index them with an index set $\cI= \cI(R)$, so that $\dW_R=\{I_j\}_{j \in \cI}$. 
\begin{lemma}[{\cite[Lemma 8.7]{david-semmes91}}]\label{l:DS1}
 Let $I, J \in \dW_R$. If $10I \cap 10 J \neq \emptyset$, then 
\begin{equation}
    C^{-1} \diam(I) \leq \diam(J) \leq C \diam(I),
\end{equation}
for some constant $C$ independent of $I, J$.
\end{lemma}
\begin{proof}[Sketch of the proof]
The lemma follows from the fact that if $10J \cap 10I \neq \emptyset$, then
\begin{equation}\label{e:Ij-sim-D-better}
    10 \diam(I) \leq D_R(y) \leq 60 \diam (I) \, \mbox{ for all }  y \in 10I. 
\end{equation}
That \eqref{e:Ij-sim-D-better} holds is proven in the paragaph below Equation (8.8) in \cite{david-semmes91}, page 44.
\end{proof}
\noindent
A corollary of Lemma \ref{l:DS1} is that if $10J \cap 10I \neq \emptyset$, then
\begin{align}\label{e:ellIapproxDRy}
    \ell(I) \approx \ell(J) \approx D_R(y) \, \mbox{ for all } y \in 10I. 
\end{align}
We will write 
\begin{equation}\label{e:IsimJ}
I\sim J \;\; \mbox{ when } \;\; 10I \cap 10J \neq\emptyset, \, \mbox{ for } I, J \in \dW_R.
\end{equation}
In \cite[Chapter 8]{david-semmes91} a subset of indices $\cI_0$ is introduced (see pg. 45, below Equation (8.9)). We alter its definition slightly to make some estimates more convenient later, but the results stay the same. Note that David and Semmes' $\cI$ corresponds to our $\cI$ and their $\cI_0$ to our $\cI_K$. For $K\geq 1$, set
\begin{equation}\label{e:def-U_0}
	U_K := \pi(2KB_R)
\end{equation}
and define
\begin{equation}\label{e:def-cI_0}
	\cI_K := \{i \in \cI \, |\, \mbox{ there is a } \, j \in \cI \, \mbox{ s.t. }\, I_i \sim I_j \,\,  \mbox{ and }\, \, I_j \cap U_K \neq \emptyset \}.
\end{equation}
We have the following lemma.
\begin{lemma}\label{l:I_0}
Let $K \geq 1$ some constant large enough. Then
\begin{equation}
\label{e:J0-in-R}
\bigcup_{i\in \cI_K}I_{i}\subseteq V_{K}:=\pi(K^2B_{R}) \subset B(\pi(x_R), K^2 \ell(R)).
\end{equation}
\end{lemma}

\begin{proof}
We only need to prove the first containment. Let $i \in \cI_K$. If $j \in \cI$ is so that $I_j \sim I_i$ and $I_j \cap U_K$, then
\[
\ell(I_{j})\approx D_{S}(I_{j})\lec K\ell(R)\] 
and so 
\begin{align*}
\dist(I_{i},\pi(R))
&  \leq \diam 10I_{i}+\diam 3I_{j} + \dist(I_{j},\pi(R))
\\ 
& \lec \ell(I_{j}) + D_{S}(I_j)\lec K\ell(R)
\end{align*}
Now \eqref{e:J0-in-R} follows for $K$ large enough with respect to the implicit constant, which is dimensional.
\end{proof}
\noindent
\begin{lemma}\label{l:Qj}
If $j\in \cI_K$, there is a cube in $\Tree(R)$, denoted by $Q_j$, so that
\begin{gather}
	\dist(\pi(Q_j), I_j) \lesssim \ell(I_j);\label{e:distQjIj} \\
\ell(I_{j})\approx \ell(Q_{j}). \label{e:QiIi}
\end{gather}
\end{lemma}
\noindent
The proof of this lemma can be found in \cite[Equation 8.10]{david-semmes91}. 
If $j\in \cI\backslash \cI_K$, we will let $Q_{j}=R$. \\

\subsection{Construction of the approximating Lipschtiz graph and approximating Lipschitz function}
\subsubsection{Construction of the Lipschitz graph}
Let $\{\phi_{i}\}$ be a partition of unity with respect to the $\dW_R=\{I_{i}\}_{i \in \cI}$. Each $\phi_i$ will satisfy
\begin{gather}
 \one_{I_i} \leq \phi_{i}\leq \one_{3I_{i}}, \label{e:phi_i-in-1/4I}\\
 \|\d^{\alpha} \phi_{i}\|_{\infty}\lec \ell(I_{i})^{-\alpha},\, \mbox{ for } |\alpha|\leq 2; \label{e:phi-properties-2}\\
 \sum_{i \in \cI} \phi_{i}\equiv 1. \label{e:phi-properties-3}
\end{gather}
We are ready to define the approximating Lipschitz graph promised at the beginning of this section. For $i \in \cI$, we let $B_{i}:\R^{d}\rightarrow \R^{n-d}$ be the affine map whose graph is $P_{Q_{i}}$. Now set
\begin{equation}\label{e:def-H}
H(x) = \sum_{i \in \cI} B_{i}(x)\phi_{i}(x).
\end{equation}
We will denote the graph of $H$ by $\Gamma_R \subset \R^n$. 
For $x\in \R^{d}$, let $b_i:\R^d \to \R^n$ be the affine map $x \mapsto x + B_i(x) \in \dM(d,n)$. Then set
\begin{equation}\label{e:def-h}
h(x) := x+ H(x)= x+ \sum_{i \in \cI} B_i(x) \phi_i(x) = \sum_{i \in \cI} b_i(x)\phi_i(x).
\end{equation}
where we view $\R^{n-d}$ as a subspace of $\R^{n}$, so $x+H(x)$ and $x+B_i(x)$ make sense\footnote{In \cite{david-semmes91}, our $H$ is denoted by $A$, see Equation (8.14).}.
To summarise,
\begin{equation}
	h(\R^d) = \{(x, H(x)) \in \R^d \times \R^{n-d} \, |\, x \in \R^d\} = \Gamma_R.
\end{equation} 
\subsubsection{H is indeed Lipschitz}
We want to show that $H$ is in fact Lipschitz, and with small constant. To this aim, we need the following.
\begin{lemma}[{\cite[Lemma 8.17]{david-semmes91}}]\label{l:QiQj}
 For $i, j \in \cI$, if $I_i \sim I_j$, then
 \begin{align}
 & \dist(Q_i, Q_j) \leq C \diam(I_i)\, \mbox{ and } \label{e:QiQj}\\
&     |b_i(x) -b_j(x)|=|B_i(x) - B_j(x)| \leq C \ve \ell(I_i) \, \mbox{ for all } x \in 100I_i.\label{e:bi-bj}
 \end{align}
\end{lemma}
\noindent
The proof of this lemma is just like in David and Semmes' monograph, except that we must also treat the case when either $i$ or $j$ are not in $\cI_K$; even the, the proof is very similar and we omit the details.
Lemma \ref{l:QiQj} is used to prove the following one, which is stated and proved in \cite{david-semmes91}, between Equations (8.15) and (8.19), page 46.
\begin{lemma}
 The restriction of $H$ to $2I_j$, $j \in \cI$, is $3 \delta_0$-Lipschitz.
 \end{lemma}

\begin{remark}\label{r:a=x}
By the penultimate paragraph of \cite[p. 47]{david-semmes91}, there is $C>0$ so that $H$ is $C\delta_0$-Lipschitz on $U_0$. This is where we diverge a bit from their construction: David and Semmes then do a Whitney extension of $H|_{U_{K}}$ to get a globally defined function $H$ that is still $C\delta_0$-Lipschitz. However, we have already defined $H$ on {\it all} of $\R^{d}$, and it can be shown that our extension is globally $C\delta_0$-Lipschitz as well. Indeed, in Lemma \ref{l:Qj} we chose $Q_i = Q_R$ when $i \in \cI \setminus \cI_K$, and that $P_{R}= \R^d \subset \R^n$. Hence $B_i=0$ for $i \in \cI\setminus \cI_K$. Thus by the definition of $H$ and \eqref{e:J0-in-R}, this means that 
\begin{equation}\label{e:H=0}
H(x)=0  \;\; \mbox{ for all }x\in \R^{d}\backslash V_K,
\end{equation}
\begin{equation}\label{e:h=0}
	h(x)=x \;\; \mbox{ for all }x\in \R^{d}\backslash V_K.
\end{equation}
\noindent
The set $V_K$ was defined in \eqref{e:J0-in-R}.
Using Pythagoras' theorem, one can show that $h:\R^n \to \R^n$ is $(1+C\delta_0^2)$-bi-Lipschitz and $b_i$ is $(1+4\delta^2_0)$-bi-Lipschitz. Also, by our construction, we have
\end{remark}

\begin{lemma}
	Notation as above. Then
	\begin{align} \label{e:distxGamma}
		\dist(x, \Gamma_R) \lesssim \ve d_R(x) \mbox{ for all } x \in R.
	\end{align}
\end{lemma}

\subsubsection{Construction of the approximating Lipschitz function}
Now, define a new function $F:\R^{d}\rightarrow \R$ by setting 
\begin{equation}\label{e:def-F}
F(x) = \sum_{i} A_{Q_{i}}\circ b_{i}(x) \phi_{i}(x).
\end{equation}
Here $A_{Q_i}$ are the affine maps $\R^n\to \R$ as in Lemma \ref{l:A_B}.
Again, because $A_{Q_{i}}\circ b_{i} = A_{R}\circ \pi_{P_{R}} = A_{R}$ for $i\in \cI\backslash \cI_K$, we have 
\begin{equation}\label{e:A=0}
F(x)= A_{R}(x) \;\; \mbox{ for all }x\in \R^{d}\backslash V_K.
\end{equation}

\subsection{Some properties of the Whitney cubes}
\begin{lemma}\label{l:QinBG}
	For $Q\in \BG(Q)$, let $x_{Q}$ be its center and let $I_{j} \in \dW_R$ be the cube containing $\pi(x_{Q})$. Then
	\begin{equation}
		\label{e:MBQ_j<M^2BQ}
		MB_{Q_{j}}\subseteq M^2 B_{Q},
	\end{equation}
	where $Q_j \in \Tree(R)$ is the Christ-David cube from Lemma \ref{l:Qj}.
\end{lemma}
\begin{proof}
	Note that necessarily we have $j\in \cI_K$, since $Q \in \BG(Q)$ implies that $Q \subset B_R$, and then $\pi(x_Q) \in U_0$, as defined in Lemma \ref{l:I_0}. First we claim that 
	\begin{equation}\label{e:dRapproxellQ}
		d_{R}(x_Q)\approx \ell(Q).
	\end{equation}
	The upper bound is immediate: 
	\begin{equation*}
		d_R(x_Q) = \inf_{Q' \in \Tree(R)} \left( \dist(x_Q, Q') + \ell(Q') \right) \leq \ell(Q).
	\end{equation*}
	We prove the lower bound. Let $T\in \Tree(R)$: if $T\supseteq Q$, then $\ell(T) +\dist(x_Q,T)= \ell(T) \geq \ell(Q)$; otherwise, $T\cap Q=\emptyset$, so in particular, $0.1 B_{Q}\cap 0.1 B_{T}=\emptyset$. This implies that $\ell(T)+\dist(x_{Q},T)\geq 0.1 \ell(Q)$. Infimizing over all such $T$'s gives $d_{R}(x_Q)\geq 0.1 \ell(Q)$, which proves \eqref{e:dRapproxellQ}. Recalling that, by hypothesis, $\pi(x_Q) \in I_j$ for some $j \in \cI_0$, we use this to conclude that
	\begin{equation}
		\label{e:QsimdsimDsimI_j}
		\ell(Q) \stackrel{\eqref{e:dRapproxellQ}}{\approx} d_{R}(x_{Q})\stackrel{\eqref{e:dsimD}}{\approx} D_{R}(\pi(x_{Q})) \stackrel{\eqref{e:ellIapproxDRy}}{\approx} \ell(I_{j})\stackrel{\eqref{e:QiIi}}{\approx}\ell(Q_{j}).
	\end{equation}
	To conclude the proof of the lemma we need to estimate $\dist(Q,Q_j)$. Note that for $x\in \Sigma \cap 2MB_R$, if $x'\in \Gamma_R$ is a closest point to $x$, then, recalling that $h$ is $(1+C\delta_0^2)$-Lipschitz (and hence $2$-Lipschitz for $\delta_0$ small enough),
	\begin{align}\label{e:form30}
		|x-h\circ \pi(x)| &
		\leq |x-x'|+|h\circ \pi(x')-h\circ \pi(x)|\nonumber \\ 
		& \leq 3|x-x'| \approx \dist(x, \Gamma_R) \stackrel{\eqref{e:distxGamma}}{\lec } \delta_0 d_{R}(x).
	\end{align}
	(By definition, if $x' \in \Gamma_R$, then $h\circ \pi(x')=x'$). In particular, if $x \in Q$ for some $Q \in \Tree(R)$, then
	\begin{align}\label{e:form31}
		|x-h\circ \pi(x)| \lesssim \delta_0 \ell(Q).
	\end{align}
	Thus
	\begin{align}
		\dist(Q,Q_{j})
		& \leq |x_{Q}-x_{Q_{j}}|\leq
		|h\circ \pi(x_{Q})-a\circ \pi(x_{Q_{j}})|+C\delta_0 \ell(Q) \notag  \\
		& \stackrel{\eqref{e:form31}}{\leq}  |\pi(x_{Q})- \pi(x_{Q_{j}})|+\delta_0  \ell(Q) \notag  \\
		& \stackrel{\pi(x_Q) \in I_j}{\lesssim} \ell(I_j) + \dist(I_j,\pi(x_{Q_{j}}))+\delta_0  \ell(Q) \notag  \\
		& \leq  \ell(I_j) + \dist(I_j,\pi(Q_{j}))+\ell(Q_j)+\delta_0 \ell(Q) \notag  \\
		&  \stackrel{\eqref{e:QiIi}}{\lec} \ell(I_{j})+\ell(Q_j) + \delta_0 \ell(Q)
		\stackrel{\eqref{e:QsimdsimDsimI_j}}{\lesssim} \ell(Q).
		\label{e:d(Q,Q_j)}
	\end{align}
	So for $M$ large enough, \eqref{e:MBQ_j<M^2BQ} follows.
\end{proof}

\begin{lemma}
	Let 
	\begin{equation}\label{e:def-alpha-1}
		\alpha_1=(1-\alpha_0^{-1})/2.
	\end{equation} 
	For $Q\in \BG(R)$,
	\begin{equation}
		\label{e:a<gradAQ-gradAR}
		|\grad A_{Q}-\grad A_{R}|\geq \alpha_1\lambda 
	\end{equation}
	
\end{lemma}

\begin{proof}
	If $Q\in \BG(R)$, the by definition it has a child $Q'$ for which either 
	\begin{enumerate}[(a)]
		\item $|\grad A_{Q'}|> \alpha_0|\grad A_{R}|$, or
		\item $|\grad A_{Q'}|< \alpha_0^{-1}|\grad A_{R}|$.
	\end{enumerate} 
	Consider first case (a). First, from Lemma \ref{l:A-A}  and \eqref{e:gamma-monotone} it follows that $|\nabla A_Q - \nabla A_{Q'}| \lesssim \gamma(MB_{Q})$. Then recall that for if $Q \subset R$, $\gamma(M^2 B_Q) \lesssim \ve \lambda$ (this is \eqref{e:gamma<ve}). Thus we get
	\begin{align*}
		|\grad A_{Q}-\grad A_{R}|
		&  \geq |\grad A_{Q'}-\grad A_{R}|-|\grad A_{Q}-\grad A_{Q'}|\\
		&  \stackrel{\eqref{e:gamma-monotone} \atop \eqref{e:gamma<ve}}{ \geq} |\grad A_{Q'}|-|\grad A_{R}| - C\ve\lambda.
	\end{align*}
	Recall also that $\lambda < |\nabla A_R| \leq (1+C \ve) \lambda$ (by \eqref{e:gradA-simLambda}). Thus, if $\ve>0$ is chosen sufficiently small with respect to $\alpha_0-1$ (and $M$, since $C$ here depends on it), we get 
	\begin{align*}
		|\nabla A_Q - \nabla A_R| \stackrel{\eqref{e:gradA-simLambda}}{\geq} (\alpha_0-1-C\ve) \lambda \geq\frac{(\alpha_0-1)\lambda}{2}.
	\end{align*}
	Given that $1-\alpha^{-1} = (\alpha_0-1)/\alpha_0 <\alpha_0-1$, we see \eqref{e:a<gradAQ-gradAR} holds in this case. In the latter case (b), for $\ve$ small enough
	\begin{align*}
		|\grad A_{Q}-\grad A_{R}|
		&  \geq |\grad A_{Q'}-\grad A_{R}|-|\grad A_{Q}-\grad A_{Q'}|\\
		& \stackrel{\eqref{e:gamma-monotone} \atop \eqref{e:gamma<ve}}{ \geq} |\grad A_{R}|-|\grad A_{Q'}| - C\ve\lambda \\
		& \stackrel{\eqref{e:gradA-simLambda}}{\geq} (1-\alpha_0^{-1} -C\ve) \lambda \geq\frac{(1-\alpha_0^{-1})\lambda}{2}
	\end{align*}
	and so have \eqref{e:a<gradAQ-gradAR} holds in this case as well.

\end{proof}

\section{Proof of Main Lemma \ref{l:m2} via a square function estimate}\label{s:tB-proof-main-lemma-via-square}
In this section, we prove Lemma \ref{l:m2} using the square function estimate below.

\begin{lemma}[Square function estimate]\label{l:OmegaF-sim}
	Let $E, R$ be as in Main Lemma \ref{l:m2} and $F$ be defined as in \eqref{e:def-F}. Then
\begin{equation}
\sum_{I\in \dI} \Omega_{F}^{1}(3B_I)^{2} |I|
\lec (\lambda \delta_0)^{2}.
\label{e:OmegaF-sim}
\end{equation}
Here $3B_I=B(x_I, 3 \ell(I))$, and $x_I \in \R^d$ is the center of $I$. 
\end{lemma}
Let us first prove a technical claim.
\begin{lemma}\label{l:gradF-gradAR}
Let $j\in \cI_K$ and $x\in I_j$. For an appropriate choice of $\delta_0$ and $\ve$ \textup{(}depending only on $\alpha_1$, as defined \eqref{e:def-alpha-1}\textup{)}, we have
\begin{equation}
\label{e:F-A_R>lambda}
|\grad F(x)- \grad A_{R}(x))|\geq \frac{\alpha_1}{4}\lambda . 
\end{equation}
\end{lemma}

\begin{proof}
	Once again, recall that $F(x)= \sum_{i \in \cI} A_{Q_i} \circ b_i(x) \phi_i(x)$, where $b_i(x) = x+ B_i$, $B_i$ is an affine map $\R^d\to \R^{n-d}$ whose graph is $P_{Q_i}$, and that $A_{Q_i}$ is the map defined in Lemma \ref{l:A_B}, that is, essentially the affine map that best approximates $f$ at $Q_i$.
	
	Let $x \in I_i$, $i \in \cI_K$. If we show that
	\begin{equation}\label{e:gradAQi-gradAR}
		|\grad (A_{Q_i} \circ b_i) (x) - \grad A_R(x)| \geq \frac{\alpha_1}{2}\lambda,
	\end{equation} 
	then we are done, since
	\begin{align*}
		& |\grad F(x) - \grad A_{R}(x))| \\
		& \geq | \grad (A_{Q_i} \circ b_i)(x) - \grad A_R(x)| - |\grad F(x) - \grad A_{Q_i} \circ b_i(x)| \\
		& \stackrel{\eqref{e:gradF-gradAQi} \atop \eqref{e:gradAQi-gradAR}}{\geq} \frac{\alpha_1}{2}\lambda - \frac{\alpha_1}{4}\lambda = \frac{\alpha_1}{4}\lambda. 
	\end{align*}
	We focus on proving \eqref{e:gradAQi-gradAR}. We compute
	\begin{align*}
		& |\grad (A_{Q_{i}}\circ b_{i})(x)-\grad A_{R}(x)| 
		= |\grad A_{Q_{i}}\cdot \grad b_{i}-\grad A_{R}|
		\\
		& =|(\grad  A_{Q_{i}} -\grad A_{R})\cdot \grad b_{i} + \grad A_{R}\cdot (\grad b_{i}-{\rm Id})| \\
		& \geq |\grad A_{Q_{i}} -\grad A_{R}|-|\grad A_{R}|\cdot |\grad B_{i}| \\
		& \stackrel{\eqref{e:gradA-gradA} \atop \eqref{e:MBQ_j<M^2BQ}}{\geq} (|\grad A_{Q} -\grad A_{R}|-\gamma(M^2B_{Q}))-2\delta_0 |\grad A_{R}| \\
		& \stackrel{\eqref{e:a<gradAQ-gradAR}\atop \eqref{e:gradA-simLambda},\eqref{e:gamma<ve}}{\geq} \alpha_1\lambda -C\ve\lambda -2\delta_0 (1+C\ve)\lambda 
		\geq \frac{\alpha_1}{2}\lambda.
	\end{align*}
\end{proof}

\begin{proof}[Proof of Main Lemma \ref{l:m2}]
Note that \eqref{e:QsimdsimDsimI_j} and \eqref{e:d(Q,Q_j)} (and the fact that disjoint cubes $Q$ and $T$ satisfy $c_{0}B_{Q}\cap c_{0}B_{T}=\emptyset$ by the definition of the Christ-David cubes) imply that there are only boundedly many $Q\in \BG(R)$ for which $\pi(x_Q)\in I_{j}$ for any given $j \in \cI_K$; that is, if we set
\begin{equation*}
	\BG_j (R) := \{Q\in \BG(R): \pi(x_Q)\in I_{j}\},
\end{equation*}
then
\begin{equation}
\label{e:QtoI_h}
\# \BG_j(R) \lec 1 \;\;\; \mbox{ for all }j\in \cI_K.
\end{equation} 
Conversely, if $Q \in \BG(R)$, then $B_Q \subset 2KB_R$, and therefore $\pi(x_Q) \in I_j$ for some $j \in \cI_K$. 
Note also that $F-A_R$ is compactly supported, since, by \eqref{e:A=0}, $F(x)=A_R(x)$ whenever $x \in \R^d \setminus V_K$. Then, using Dorronsoro's Theorem \ref{t:dorronsoro-classic}, the affine invariance of $\Omega$ \eqref{e:omega-affine-invariant} and Lemma \ref{l:OmegaF-sim}, we obtain
\begin{equation}\label{e:form700}
\|\grad F-\grad A_{R})\|_{2}^{2} \lec 
\sum_{I \in \dI}\Omega_{F-A_{R}}(3B_I)^{2}|I|
=\sum_{I \in \dI}\Omega_{F}(3B_I)^{2}|I|
\stackrel{\eqref{e:OmegaF-sim}}{\lec} (\lambda \delta_0 )^{2} \ell(R)^{d}.
\end{equation}
\noindent
We conclude that
\begin{align*}
\sum_{Q\in \BG(R)}|Q| & = \sum_{j \in \cI_K} \sum_{Q \in \BG_j(R)} |Q|\\
& \approx \sum_{j \in \cI_K} \sum_{Q \in \BG_j(R)} \ell(Q)^d \\
& \stackrel{\eqref{e:QsimdsimDsimI_j}}{\approx} \sum_{j \in \cI_0} \#\BG_j(R) \ell(I_j)^d \\
& \stackrel{\eqref{e:QtoI_h}}{\lesssim} \sum_{j \in \cI_K} \int_{I_j}dx \\ 
& \stackrel{\eqref{e:F-A_R>lambda}}{\lesssim} \lambda^{-2} \sum_{j \in \cI_K} \int_{I_{j}} |\grad (F(x)-A_{R}(x))|^2 dx\\
& \leq C \lambda^{-2} \|\grad F-\grad A_{R}\|_{2}^{2} 
\stackrel{\eqref{e:form700}}{\lec} \delta_0^{2}|R|.
\end{align*}
\noindent
For $\delta_0>0$ sufficiently small, this proves \eqref{e:m2}, and finishes the proof of Main Lemma \ref{l:m2}, and thus the proof of  Theorem \ref{t:dorronsoro-II}. 
\end{proof}

\def\I{\mathcal{I}}
\section{Proof of the square function estimate Lemma \ref{l:OmegaF-sim}.}\label{s:tB-proofLemmaOmegaF-sim}
This last section is devoted to the proof of Lemma \ref{l:OmegaF-sim}. For reader's sake, we report here its statement. 
\begin{lemma}
	Let $E, R$ be as in Main Lemma \ref{l:m2} and $F$ be defined as in \eqref{e:def-F}. Then
	\begin{equation}\label{e:OmegaF-sim2}
		\sum_{I\subseteq \R^{d}} \Omega_{F}^{1}(3B_I)^{2} |I|
		\lec (\lambda \delta_0)^{2}.
	\end{equation}
\end{lemma}
\noindent
We split the family of dyadic cubes $I$ in the sum of \eqref{e:OmegaF-sim2} intro three subfamilies; we will prove the estimate above for each one of them.

\begin{itemize}
\item Let $\Delta_K(R)=\Delta_K$ be those dyadic cubes $I$ for which $3I\cap 3I_{j}\neq\emptyset$ for some $j\in \cI_{K}$. 
\item Let $\Delta_{1}\subseteq \Delta_K$ be those cubes for which $I\subseteq I_{j}$ for some $j$. 
\item Let $\Delta_2\subseteq \Delta_K \backslash \Delta_{1}$ be those cubes $I$ with $\ell(I)\leq \ell(R)$.
\item Let $\Delta_3\subseteq \Delta_K \backslash (\Delta_{1}\cup \Delta_{2})$ be those cubes $I$ with $\ell(I)> \ell(R)$.
\end{itemize}

Over the next few subsections, we will prove 
\begin{lemma}
\begin{equation} \label{e:Deltai-bound}
\sum_{I\in \Delta_{i}}\Omega_{F}^{1}(3B_I)^{2}|I|\lec (\lambda\delta_0)^{2}|R| \;\; \mbox{ for }\;\; i=1,2,3.
\end{equation}
\end{lemma}
\begin{proof}[Proof of Lemma \ref{l:OmegaF-sim} assuming \eqref{e:Deltai-bound}]
Lemma \ref{l:OmegaF-sim} follows almost immediately. Indeed, notice that if $I$ is a dyadic cube so that $3I\cap 3I_j=\emptyset$ for all $j\in \cI_K$, then $F|_{3I}=A_{R}$ and so $\Omega_{F}^{1}(3B_I)=0$. Hence, we only need to compute the sum in Lemma \ref{l:OmegaF-sim} over cubes $I\in \Delta_K$, in which case
\begin{equation*}
\sum_{I}\Omega_{F}^{1}(3B_I)^{2}|I|
=\sum_{I\in \Delta_K}\Omega_{F}^{1}(3B_I)^{2}|I|
=\sum_{i=1}^{3}\sum_{I\in \Delta_{i}}\Omega_{F}^{1}(3B_I)^{2}|I|
\stackrel{\eqref{e:Deltai-bound}}{\lec} (\lambda\delta_0)^{2}|R|.
\end{equation*}
\end{proof}

\subsection{A technical lemma} In this subsection, we  prove the following lemma, which will be useful to estimates the sums in \eqref{e:Deltai-bound}.
\begin{lemma}\label{c:claim30}
	Let $j \in \cI$ and $x \in 3I_{j}$. We have
	\begin{equation}\label{e:gradF-gradAQi<lambdave}
		|\grad \left(F - A_{Q_{j}}\circ b_{j}\right)(x)| \lesssim (\delta_0+\ve) \lambda. 
	\end{equation}
	In particular, if we choose $\delta_0$ and $\ve$ sufficiently small with respect to $\alpha_1=\frac{1-\alpha^{-1}_0}{2}$, we obtain
	\begin{equation}\label{e:gradF-gradAQi}
		|\grad \left(F - A_{Q_{j}}\circ b_{j}\right)(x)| \lesssim \frac{\alpha_1}{4} \lambda.
	\end{equation}
\end{lemma}
\begin{proof}
	Since $\sum_{i \in \cI} \phi_i  \equiv 1$, then we can compute 
	\begin{align*}
		& \left|\grad F(x) - \grad (A_{Q_j}\circ b_j)\right| \\
		& = \Bigg| \sum_{i \in \cI} \left[\grad (A_{Q_i} \circ b_i) \right] \phi_i(x)  + \sum_{i \in \cI} A_{Q_i}\circ b_i (x) \grad \phi_i(x) \\
		& \qquad \qquad \qquad \qquad - \grad \left( A_{Q_j}\circ b_j(x)\sum_{i \in \cI}\phi_i(x) \right)\Bigg|\\
		& \leq \sum_{i \in \cI} |\grad (A_{Q_i}\circ b_i) - \grad( A_{Q_j} \circ b_j)| \\
		& \qquad \qquad \qquad + \left| \sum_{i \in \cI} \left( A_{Q_i}\circ b_i(x) - A_{Q_j} \circ b_j(x) \right) \grad \phi_i(x) \right| =: T_1 + T_2.
	\end{align*}
	Note that $\phi_i(x) \neq 0$ only whenever $3I_j \cap 3I_i \neq \emptyset$ and that, given a fixed $j \in \cI$, there exists boundedly many other $i$ such that $3I_i \cap 3I_j \neq \emptyset$. Hence, to estimate $T_1$, it suffices to estimate $|\grad A_{Q_i} \circ b_i - \grad A_{Q_j} \circ b_j|$ for some $i \in \cI$ with this property. To this aim, we first claim that 
	\begin{gather}
		\ell(Q_i) \approx \ell(Q_j) \, \mbox{ and }\label{e:ellQiellQj}\\
		\dist(Q_j, Q_i) \lesssim \ell(Q_i).
	\end{gather}
	The first one is immediate: we know that, since $3I_i \cap 3I_j \neq \emptyset$, then $\ell(I_i) \approx \ell(I_j)$ by \eqref{e:ellIapproxDRy} and that $\ell(Q_i) \approx \ell(I_i)$ whenever $i \in \cI$ by Lemma \ref{l:Qj}. We have
	\begin{align*}
		& \dist(Q_i, Q_i)  \leq |x_{Q_j} - x_{Q_i}|\\
		& \leq |x_{Q_i} - h \circ \pi (x_{Q_i})| +| h \circ \pi (x_{Q_i}) - h \circ \pi (x_{Q_j})| + |h \circ \pi (x_{Q_j})- x_{Q_j}|  
	\end{align*}
	As in \eqref{e:form31}, we have $|x_{Q_i} - h \circ \pi (x_{Q_i})| \lesssim \ell(Q_i)$ (for any $i$). Also, since $3I_i\cap 3I_j \neq \emptyset$, and $\dist(\pi(Q_i), I_i) \lesssim \ell(I_i)$ for any $i \in \cI_0$, then 
	\begin{align*}
		|h \circ \pi (x_{Q_i}) - h \circ \pi (x_{Q_j})| \lesssim \ell(I_j) \approx \ell(Q_j) \approx \ell(Q_i).
	\end{align*}
	All in all we see that $\dist(Q_i, Q_j) \lesssim \ell(Q_i)$. This implies that, if $Q^* \in \BG(R)$ is the cube given by Lemma \ref{l:QinBG} (either for $i$ or $j$, it doesn't matter), then $MB_{Q_i} \subset 3M^2B_{Q^*}$ and $MB_{Q_j} \subset 3M^2B_{Q^*}$, whenever $3I_i \cap 3I_j \neq \emptyset$. Using \eqref{e:gradA-gradA},  \eqref{e:gamma-monotone}, and \eqref{e:gamma<ve}, we get that 
	\begin{equation}\label{e:gradQj-gradQi}
		|\nabla A_{Q_j} - \nabla A_{Q_i}| \leq C \ve \lambda,
	\end{equation}
	whenever $3I_i \cap 3I_j$ and $C$ depends on $M$. Recall also that $|\nabla b_i| \leq (1+C\delta_0^2) \lambda$. Finally, it is easy to see from the fact that $Q_i, Q_j \in \Tree(R)$, that 
	\begin{align*}
		|\nabla b_i (x) - \nabla b_j(x)| = |\nabla B_i(x) - \nabla B_j(x)| \lesssim \delta_0. 
	\end{align*}
	Thus we get (recall $x \in 3I_j$):
	\begin{align}\label{e:claim-T1}
		& T_1 = \sum_{i \in \cI \atop 3I_i \cap 3I_j \neq \emptyset}| \grad (A_{Q_i}\circ b_i)- \grad( A_{Q_j} \circ b_j)| \notag \\
		&  \quad \quad \lesssim \sum_{3I_i \cap 3I_j \neq \emptyset} |\grad A_{Q_i}\cdot (\grad b_i(x) - \grad b_j(x))| + |(\grad A_{Q_i}- \grad A_{Q_j}) \cdot \grad b_j(x)| \notag \\
		& \quad \quad \leq C\alpha_0 \lambda \delta_0 + C\ve \lambda (1+C\delta_0^2).
	\end{align}
	Let us now estimate $T_2$. We compute
	\begin{align*}
		T_2 & \leq \sum_{i \in \cI \atop 3I_i \cap 3I_j \neq \emptyset} \frac{|A_{Q_i} \circ b_i(x) - A_{Q_i} \circ b_j(x)|}{\ell(I_j)} \\
		& \qquad + \sum_{i \in \cI \atop 3I_i \cap 3I_j \neq \emptyset} \frac{|A_{Q_i} \circ b_j(x) - A_{Q_j}\circ b_j(x)|}{\ell(I_j)}=: T_{2,1} + T_{2,2}.
	\end{align*}
	Note that for $x,y \in \R^n$, $A_{Q_i}x - A_{Q_i}y = \grad A_{Q_i}$, so 
	\begin{align*}
		T_{2,1} \leq \sum_{i \in \cI \atop 3I_i\cap 3I_j \neq \emptyset} |\grad A_{Q_i}|\frac{|b_i(x) - b_j(x)|}{\ell(I_j)}.
	\end{align*}
	Now, 
	\begin{align*}
		|b_i(x) - b_j(x)|= |B_i(x) - B_j(x)| \approx \angle(P_{Q_i}, P_{Q_j}) \ell(I_j) \lesssim \delta_0 \ell(I_j),
	\end{align*} 
	since $Q_i, Q_j \in \Tree(R)$. Moreover, $|\grad A_{Q_i}| \leq \alpha_0 |\grad A_R| \leq \alpha_0(1+ C \ve) \lambda$ from \eqref{e:gradA-simLambda} and again using that $Q_i \in \Tree(R)$. We then see that
	\begin{align}\label{e:claim-T21}
		T_{2,1} \lesssim \delta_0 \alpha_0(1+ C\ve)\lambda
	\end{align}
	We estimate $T_{2,2}$. From \eqref{e:AB-AB'} and the fact that $MB_{Q_i} \subset M^2B_{Q_j}$ (which can be deduced as in the paragraph above \eqref{e:gradQj-gradQi}) and viceversa, we deduce that
	\begin{align*}
		|A_{Q_i}(x) - A_{Q_j}(x)| \lesssim \gamma(M^2B_{Q_j})\left( \dist(x, M^2B_{Q_j}) + \ell(Q_j)\right)
	\end{align*}
	Now, \eqref{e:QiIi} tells us that $\ell(Q_j) \approx \ell(I_j)$; also recall that since $Q_j \in \Tree(R)$, $\gamma(M^2B_{Q_j}) \lesssim_M \ve \lambda$. We finally conclude that
	\begin{align}\label{e:claim-T22}
		T_{2,2} \lesssim \ve \lambda,
	\end{align}
	since $|b_i| \lesssim 1$ for any $i \in \cI$.
	Collecting \eqref{e:claim-T1}, \eqref{e:claim-T21} and \eqref{e:claim-T22}, we finish the proof of the lemma.
\end{proof}

\subsection{Estimates for $\Delta_{1}$}
Recall that $\Delta_1$ is the family of those dyadic cubes $I \in \dI$ so that $3I \cap 3I_j$, for some $j \in \cI_K$ and $I \subset I_j$. 
Our aim in this subsection is to prove 
\begin{equation}
\label{e:C1}
\sum_{I\in \Delta_{1}}\Omega_{F}^{1}(3B_I)^{2}|I|\lec (\lambda\delta_0)^{2}|R|.
\end{equation}

\begin{proof}[Proof of \eqref{e:C1}]
We claim that it suffices to prove that for all $i\in \cI$,
\begin{equation}\label{e:I<I_jsum}
\sum_{I\subseteq I_{i}} \Omega_{F}^{1}(3B_I)^{2} |I|
\lec  (\lambda\delta_0)^{2}|I_{j}|.
\end{equation}
Let us see why this claim is valid.
Let $\cI_1$ be the set of indexes $i\in \cI$ for which there exists at least one cube $I \in \Delta_1$ with $I \subset I_i$. By definition, if $I \in \Delta_1$, then $3I \cap 3I_j \neq \emptyset$ for some $j \in \cI_K$. This implies that if $i \in \cI_1$, then there exists a $j \in \cI_K$ so that $3I_i \cap 3I_j \neq \emptyset$. By \eqref{e:ellIapproxDRy} we have that $\ell(I_j)\approx \ell(I_i)$. Then, for $K \geq 1$ sufficiently large, $I_j \subset V_K= \pi(K^2B_R)$ (this is \eqref{e:J0-in-R}). We conclude that 
\begin{equation}\label{e:form34}
    I_{i}\subseteq \pi( C K^{2}B_{R}) \mbox{ whenever } i \in \cI_1, 
\end{equation}
where $C$ is come sufficiently large universal constant. We conclude that
\begin{align*}
\sum_{I\in \Delta_{1}}\Omega_{F}^{1}(3B_I)^{2}|I|
& =\sum_{i\in \cI_1} \sum_{I\subseteq I_{i}}\Omega_{F}^{1}(3B_I)^{2}|I| 
 \\ & \stackrel{\eqref{e:I<I_jsum}}{\lec} (\lambda\delta_0)^{2}\sum_{i\in \cI_1} |I_{i}| \stackrel{\eqref{e:form34}}{\lec}  (\lambda\delta_0)^{2}|R|.
\end{align*}
This proves \eqref{e:C1}, assuming \eqref{e:I<I_jsum}. We now verify this latter inequality.\\

\noindent
First, note that for any $i\in \cI$
\begin{equation}
\label{e:F-AQjcircBQ_jLinfty}
\|\grad F- \grad (A_{Q_{j}}\circ b_{Q_{i}})\|_{L^{\infty}(3I_{i})} \stackrel{\eqref{e:gradF-gradAQi<lambdave}}{\lesssim} \delta_0\lambda,
\end{equation}
since $\ve \ll \delta_0$. Further, from Theorem \ref{t:dorronsoro-classic} it follows easily that
\begin{equation}\label{e:form607}
	\sum_{I \subset I_i} \Omega_{F-A_{Q_i}\circ b_{Q_i}}(3B_I)^2 |I| \lesssim_d \|\grad(F-A_{Q_i}\circ b_{Q_i})\|_{L^\infty(3I_i)}^2 |I_i|.
\end{equation}
Then\footnote{Recall that $F:\R^d \to \R$, that $A_{Q_i}: \R^n \to \R$ and $b_i: \R^d \to \R^n$, and that both $A_{Q_i}$ and $b_i$ are affine, so that their composition is also affine.}, for each $i \in \cI$ we obtain
\begin{align*}
\sum_{I\subseteq I_{i}} \Omega_{F}^{1}(3B_I)^{2} |I| & 
\stackrel{\eqref{e:omega-affine-invariant}}{=} \sum_{I\subseteq I_{i}} \Omega_{F-A_{Q_i}\circ b_{Q_{j}}}^{1}(3B_I)^{2} |I| \\
& \stackrel{\eqref{e:form607}}{\lesssim} |I_i| \|\grad F- \grad (A_{Q_{j}}\circ b_{Q_{i}})\|_{L^{\infty}(3B_{I_{i}})}^2 \stackrel{\eqref{e:F-AQjcircBQ_jLinfty}}{\lesssim}  (\lambda \delta_0)^{2}|I_{i}|,
\end{align*}
which proves the lemma. 
\end{proof}
\vspace{1cm}

\subsection{Estimates for $\Delta_2$}
The goal of this section is to prove 
\begin{equation}
	\label{e:C2}
	\sum_{I\in \Delta_{2}}\Omega_{F}^{1}(3B_I)^{2}|I|\lec (\lambda\delta)^{2}|R|.
\end{equation}
\subsubsection{Preliminaries}
We establish first some preliminary facts.
Recall that $\Delta_2 \subset \Delta_R$ is the family of dyadic cubes $I \in \dI$ so that $3I\cap 3I_j$, $j \in \cI_K$, $\ell(I) \leq \ell(R)$ and also $I$ is not contained in any $I_j$, $j \in \cI$. Since $\dW_R$ covers $V_K$ (see \eqref{e:J0-in-R}), this implies that given $I \in \Delta_2$, there is a $j_0 \in \cI$ so that $I_{j_0} \subset I$. Now let $Q_{I}'\in \Tree(R)$ be a cube so that 
\begin{equation*}
\ell(Q_{I}')+\dist(I,\pi(Q_{I}'))\leq 2 D_{R}(I).
\end{equation*}
That is, $Q_I'$ `almost-minimises' $D_R(I)$. By definition, $\dW=\{I_i\}_{i \in \cI}$ is the maximal family of dyadic cubes in $\dI$ so that $\diam(I_i)< \frac{1}{20} D_R(I_i)$. Then, since $I_{j_0} \subsetneq I$, $I \in \Delta_2$, $\diam(I) \geq \frac{1}{20} D_R(I)$. 
This implies that 
\begin{equation}
\label{e:QI'distIQI'}
\ell(Q_{I}')+\dist(I,\pi(Q_{I}'))\leq 2 D_{R} (I) \leq 40\ell(I).
\end{equation}
Now we choose $Q_{I}\supseteq Q_{I}'$ in $\Tree(R)$ to be maximal so that $\ell(Q_{I})\leq 40\ell(I)$. Since $\ell(I)\leq \ell(R)$,
\begin{equation}\label{e:Q_IsimI}
\ell(Q_{I})\approx \ell(I).
\end{equation}
Our task is to estimate
\begin{align*}
    \Omega_{F}^{1}(3B_I) = \inf_{A\in \dM(d,1)} \avint_{3B_I} \frac{|F(x)-A(x)|}{\ell(I)} \, dx.
\end{align*}
To this end, let $A_I: \R^d \to \R^n$ be the affine map that minimizes $\Omega_{h}^{1}(C_1I)$, where $C_1$ is a large constant we will pick later. Recall that $h:\R^d \to \R^n$ was defined in \eqref{e:def-h}, and that it is given by
\begin{align*}
	h(x)= x + \sum_{i \in \cI} B_i(x) \phi_i(x)= \sum_{i\in \cI} b_i(x) \phi_i(x),
	\end{align*}
where $B_i:\R^d \to \R^{n-d}$ is the map whose graph is $P_{Q_i}$, $Q_i$ is the cube corresponding to $I_i \in \dW$, $i \in \cI$, as found in Lemma \ref{l:Qj}. Note also that if $i \in \cI\setminus \cI_K$, then $B_i \equiv 0$ and so $b_i(P_R)= P_R= \R^d$. Now let $x\in 3B_I$. Then
\begin{align}\label{e:def-T1T2T3}
F(x)-A_{Q_{I}}\circ A_{I}(x)
 & =\sum_{i\in \cI} A_{Q_{i}}\circ b_{i}(x)\phi_{i}(x) - A_{Q_{I}}\circ A_{I}(x)\notag \\
& = \sum_{i \in \cI \atop 3I_i} \left(A_{Q_{i}}\circ b_{i}(x)-A_{Q_{I}}\circ b_{i}(x))\phi_{i}(x)\right) \notag \\
&  + \sum_{i \in \cI}\left(A_{Q_{I}}\circ b_{i}(x)-A_{Q_{I}}\circ b_{Q_{I}}(x)\right)\phi_{i}(x)   \notag  \\
&  + \sum_{i \in \cI}\left(A_{Q_{I}}\circ b_{Q_{I}}(x) - A_{Q_{I}}\circ A_I(x)\right)\phi_{i}(x)\notag\\
 & =: T_{1}(x)+T_{2}(x)+T_{3}(x).
\end{align}
We will bound the integrals of each of these terms separately to estimate $\Omega_{F}^{1}=\Omega_{F-A_{Q_{I}}\circ A_{I}}^{1}$.

\subsubsection{Bounds on $T_{1}$}
The goal of this subsection is to prove the following estimate on $T_1$:
\begin{lemma}\label{l:T1-Delta2}
	For $I \in \Delta_2$,
\begin{equation}
\label{e:I1}
\int_{3B_I} |T_{1}(x)| \, dx
\lec  \delta_0\lambda  \sum_{3I_{i}\cap 3I\neq\emptyset}\ell(Q_{i})^{d+1}  +  \gamma(MB_{Q_{I}})\ell(Q_{I})^{d+1}.
\end{equation}
\end{lemma}
\noindent
Recall that $3B_I= B(x_I, 3 \ell(I))$, where $x_I$ is the center of $I$. For $i \in \cI$, set\footnote{Recall that $b_i: \R^d \to \R^n$ is the affine map whose image is $P_{Q_i}$. Then $b_i(x_{I_i})$ lies on $P_{Q_i}$. Recall also that $x_{I_i}$ is the center of $I_i$.}
\begin{align*}
    B(i) := B(b_i(x_{I_i}), 10 \diam(I_i))
\end{align*}
and let $\psi_{i}:\R^n \to \R$ be a smooth cut-off function such that
\begin{align*}
{\rm spt}(\psi_i) \subset 2B(i), \quad \psi_i \equiv 1 \mbox{ on } B(i) \quad \mbox{and } |\grad \psi_i| \lesssim \frac{1}{\ell(I_i)}.
\end{align*}
Set also 
\begin{align*}
    \Psi_{i}:=\psi_{i}\left(\phi_{i}\circ\pi\right), 
\end{align*}
which is easily seen to be $C/\ell(I_{i})$-Lipschitz. 
\begin{lemma}
\label{l:MBQiMBQI}
Let $I \in \Delta_2$. If $C_2\geq 3$ and $3I_i\cap C_2I\neq\emptyset$, then for $M$ large enough depending on $C_2$,
\begin{equation}
\label{e:MBQiMBQI}
\supp \Psi_i \subseteq M^{\frac{1}{2}}B_{Q_{i}}\subseteq M B_{Q_{I}}.
\end{equation}
\end{lemma}

\begin{proof}
For $Q \in \Tree(R)$, the computation in \eqref{e:a<2/lambdagamma} together with \eqref{e:gamma<ve} give $\alpha_\sigma^d(MB_Q) \lesssim \ve$; then
 $\beta_{\Sigma}^{1}(MB_{Q_{i}},P_{Q_{i}})\lec \ve$, since $Q_i \in \Tree(R)$. By Chebyhev's inequality there is $\xi\in Q_{i}$ with $\dist(\xi,P_{Q_{i}})\lec M\ve\ell(Q_i)$. Hence, there is $\zeta\in P_{Q_{i}}$ with $|\zeta-\xi|\lec M\ve\ell(Q_i)$. So we compute
\begin{align}\label{e:form40}
|b_i(x_{I_{i}})-x_{Q_{i}}| & \leq |b(x_{I_i}) - \xi| + |\xi - x_{Q_i}| \notag \\
& \lec |b_i(x_{I_{i}})-\xi| + \ell(Q_i)
\lesssim |b_i(x_{I_{i}})-\zeta| +  \ell(Q_i). 
\end{align}
Now recall that $P_R= \R^d$, and that $\angle(\R^d, P_{Q_i}) < \delta_0$ since $Q_i \in \Tree(R)$. This implies that $|b_i(x_{I_i}) - \zeta| \approx |\pi(b_i(x_{I_i})) - \pi(\zeta)| \approx |x_{I_i} - \pi(\zeta)|$. Moreover, since $|\zeta-\xi| \lesssim \ve \ell(Q_i)$ and $\xi \in Q_i$, then $|\pi(\zeta) - \pi(x_{Q_i})| \lesssim \ell(Q_i)$. Thus
\begin{align*}
\eqref{e:form40} \lesssim |x_{I_{i}}-\pi(x_{Q_{i}})| + \ell(Q_i)
\leq \dist(I_i,\pi(Q_{i}))+ \ell(Q_i)\stackrel{\eqref{e:distQjIj}}{\lec} \ell(Q_i),
\end{align*}
Recall also that $\diam(I_i) \approx \ell(I_i) \approx \ell(Q_i)$ whenever $i \in \cI$ by \eqref{e:QiIi}. Thus $2B(i) \subset M^{\frac{1}{2}} B_{Q_i}$ whenever $M$ is chosen sufficiently large; this proves the first containment in \eqref{e:MBQiMBQI}.

Let us prove the second one. Note that since $3I_{i}\cap C_2 I\neq\emptyset$ and $\dW_R$ covers $\R^d$, we can find $I_j\subseteq C_2 I$ with $I_{i}\sim I_{j}$ (recall this notation in \eqref{e:IsimJ}). Thus,
\begin{equation*}
\ell(Q_{i})\stackrel{\eqref{e:QiIi}}{\approx} \ell(I_{i}) \stackrel{I_j\sim I_i}{\approx} \ell(I_{j})\stackrel{I_j \subset C_2I}{\lesssim} C_2 \ell(I) \stackrel{\eqref{e:Q_IsimI}}{\approx} C_2 \ell(Q_I).
\end{equation*}
Moreover,
\begin{align*}
& \dist(\pi(Q_{i}),\pi(Q_{I}))\\
 & \lesssim \dist(\pi(Q_{i}),I_{i}) + \ell(I_{i})
 + \dist(I_{i},I)+\ell(I) + \dist(I,\pi(Q_{I})) \\
 & \stackrel{\eqref{e:distQjIj}}{\lesssim} \ell(I) + \dist(I, \pi(Q_I)) 
  \stackrel{\eqref{e:QI'distIQI'}}{\lec}  \ell(I).
 \end{align*}
The last two inequalities imply
\begin{align*}
|\pi(x_{Q_{i}})-\pi(x_{Q_{I}})|
& \leq \ell(Q_{i}) + \dist(\pi(Q_{i}),\pi(Q_{I}))+ \ell(Q_{I}) \\
& \lec \ell(I)\approx \ell(Q_{I})
\end{align*}
By \eqref{e:piperppi}, we can finally conclude that
\begin{align*}
|x_{Q_{i}}-x_{Q_{I}}|
& \leq \sqrt{|\pi(x_{Q_{i}})-\pi(x_{Q_{I}})|^{2}+|\pi^{\perp}(x_{Q_{i}})-\pi^{\perp}(x_{Q_{I}})|^{2}}\\
& \lec |\pi(x_{Q_{i}})-\pi(x_{Q_{I}})| \lec \ell(Q_I),
\end{align*}
which implies \eqref{e:MBQiMBQI} for $M$ large.
\end{proof}

\begin{lemma}\label{l:LipAQ-AQi}
\begin{equation}
\label{e:def-Li}
\Lip \ps{( A_{Q_{I}}-A_{Q_{i}})\Psi_{i}}
\lec \avint_{Q_{i}}\frac{|A_{Q_{I}}-f|}{\ell(Q_{i})}d\sigma + \lambda\delta_0 =: L_i,
\end{equation}
\begin{equation}
\label{e:gradA-gradA<L}
|\grad A_{Q_{I}}-\grad A_{Q_{i}}|\lec L_i,
\end{equation}
and for all $x$,
\begin{equation}
\label{e:A-A<L}
|A_{Q_{I}}(x)-A_{Q_{i}}(x)|\lec L_i(\dist(x,Q_i)+\ell(Q_i)).
\end{equation}
\end{lemma}

\begin{remark}It will come in handy later that, since 
\begin{equation}
    G_i(x) := \frac{(A_{Q_I}(x)-A_{Q_i}(x))\Psi_i(x)}{C L_i}
\end{equation}
is $1$-Lipschitz with support in $2B(i)$, then, if $P_{Q_i}, c_{Q_i}$ are a plane and a constant infimising $\alpha_\sigma^d(MB_{Q_i})$, 
\begin{align}\label{e:form52}
    \left| \int G_i(x) \, c_{Q_i} \dH^d|_{P_{Q_i}}(x) - \int G_i (x) d \sigma(x) \right| \lesssim  \alpha_\sigma^d(MB_{Q_i}) \ell(Q_i)^{d+1}.
\end{align}
\end{remark}

\begin{proof}
We first prove \eqref{e:gradA-gradA<L}. To begin with, note that
\begin{align*}
\avint_{Q_{i}}|A_{Q_{I}}-A_{Q_{i}}|d\sigma
& \leq \avint_{Q_{i}}|A_{Q_{I}}-f|d\sigma+\avint_{Q_{i}}|f-A_{Q_{i}}|d\sigma \\
& \lec \avint_{Q_{i}}|A_{Q_{I}}-f|d\sigma+\Omega_{f}^{1}(MB_{Q_{i}})\ell(Q_{i})\\
& \stackrel{\eqref{e:gamma-small}}{\lec}  \avint_{Q_{i}}|A_{Q_{I}}-f|d\sigma+\delta_0\lambda \ell(Q_{i})= L_{i}\ell(Q_{i}).
\end{align*}
Once again, we will need the following fact:
since $\Sigma$ is Ahlfors $d$-regular, there is a constant $c$ (depending only on the Ahlfors regularity data of $\Sigma$) so that for any ball $B$ centered on $\Sigma$ we can find balls\footnote{Not to be confused with the affine maps $B_i: \R^d \to \R^{n-d}$. The meaning will be clear from the context.} $B_{0},...,B_{d}$ centered on $\Sigma$ of radii $c\ell(Q_i)$ so that $2B_k\subseteq c_0B_{Q_{i}}$ and 
\begin{equation}
\label{e:d-dim-points}
\dist(x_{B_{k+1}},{\rm span} \{x_{0},...,x_{k}\})\geq 4c\ell(Q_i),
\end{equation}
where $x_k= x_{B_k}$ is the center of $B_k$.
By Chebyshev's inequality, there are $x_{k}\in B_{k}\cap \Sigma$ so that if $y_{k}=\pi_{P_{Q_{i}}}(x_k)$,
\begin{equation}\label{e:form50}
|x_{k}-y_{k}|  \lec_{c} \beta_{\Sigma}^{d,1}(MB_{Q_{i}})M\ell(Q_{i}) \lec_{M}\ve \ell(Q_{i})
\end{equation}
and
\begin{equation}\label{e:form51}
|A_{Q_{i}}(x_k)-A_{Q_{I}}(x_k)| \lec L_{i}\ell(Q_i).
\end{equation}
One may find them as in the proof of Claim \ref{c:balanced-points}.
If $\ve>0$ is small enough with respect to $c$ (so depending only on the Ahlfors regularity data of $\Sigma$), we still have that
\begin{equation*}
\dist(y_{k+1},{\rm span}\{y_{0},...,y_{k}\})\geq c\ell(Q_i) \, \mbox{ for } \, k=0,...,d-1.
\end{equation*}
This can be shown using \eqref{e:QinTree-beta-inf<ve}.
In particular, the vectors $\{u_{k}=y_{k}-y_{0}:k=1,...,d\}$ form a basis for $P_{Q_i}-y_{0}$. Hence, if $v \in B(0, r_{MB_{Q_I}})$, we can write $v= \alpha_1 v_1 + \alpha_2 v_2$, where $|\alpha_i| \lesssim 1$ and $v_1 \in P_{Q_i} - y_0$ and $v_2 \in P_{Q_i}^\perp -y_0$, and $|v_i|\lesssim \ell(Q_i)$ for $i=1,2$. We now compute
\begin{align*}
&\|\grad A_{Q_i} - \grad A_{Q_I} \|_{{\rm op}} = \sup_{v \in \bB} |(\grad A_{Q_i} - \grad A_{Q_I})v| \\
& \qquad \leq \sup_{v \in \bB \cap P_B -y_0}| (\grad A_{Q_i} - \grad A_{Q_I})v| + \sup_{v \in \bB \cap P_B^\perp -y_0}| (\grad A_{Q_i} - \grad A_{Q_I})v|. 
\end{align*}
So, we first estimate
\begin{align*}
 \sup_{v\in P_{Q_i}-y_{0}\atop |y|\lesssim \ell(Q_i)} & |(\grad A_{Q_i}-\grad A_{Q_{I}})v| \lec \max_{1\leq k\leq d}|(\grad A_{Q_{i}}-\grad A_{Q_{I}})(y_{k}-y_{0})|\\
& =\max_{1 \leq k \leq d}|(A_{Q_i}(y_{k})-A_{Q_i}(y_{0}))-(A_{Q_I}(y_{k})-A_{Q_I}(y_{0}))|\\
& \lec \max_{1 \leq k \leq d}|(A_{Q_i}(x_{k})-A_{Q_i}(x_{0}))-(A_{Q_I}(x_{k})-A_{Q_I}(x_{0}))|\\
&  \quad + (|\grad A_{Q_{i}}| +|\grad A_{Q_I}|)(|x_{k}-y_{k}|+|x_{0}-y_{0}|) \\
& \stackrel{\eqref{e:form50} \atop \eqref{e:form51}}{\lec} L_i\ell(Q_{i})+\lambda \ve \ell(Q_i)
\lec L_i\ell(Q_{i}),
\end{align*}
where we used the fact that $|\nabla A_Q| \lesssim \lambda$ for any $Q \in \Tree(R)$. On the other hand, if $v \in P_{Q_{i}}^\perp-y_0$, then $\pi_{P_{Q_i}^\perp - y_0}(v) =0$. Recall also that $\grad A_{Q_i} = \grad A \cdot \grad \pi_{P_{Q_i}}$. The kernel of $\grad \pi_{P_{Q_i}}$ is $P_{Q_i}^\perp -y_0$, and so $\grad A_{Q_i} \equiv 0$ on $P_{Q_i}^\perp -y_0$. Hence
\begin{align*}
|\pi_{P_{Q_{I}}-y_0}(v)|
& = |\pi_{P_{Q_{I}}-y_0}(v)-\pi_{P_{Q_{i}}-y_0}(v)|
\leq \|\pi_{P_{Q_{I}}-y_0}-\pi_{P_{Q_{i}}-y_0}\|_{{\rm op}}\cdot |v|\\
& \approx \angle(P_{Q_{I}},P_{Q_{i}})|v|
\lesssim \left[\angle(P_{Q_{I}},\R^{d}) + \angle(\R^{d},P_{Q_{i}})\right]|v|
\lec \delta_0 |v|,
\end{align*}
since $Q_i$, $Q_I \in \Tree(R)$. If moreover $|v|\lesssim \ell(Q_i)$, we obtain
\[
|(\grad A_{Q_i}-\grad A_{Q_{I}})v|
=|\grad A_{Q_I}v|
= |\grad A_{Q_I} \circ \pi_{P_{Q_I}-y_0}(v)|
\lec |\grad A_{Q_{I}}|\delta_0 \ell(Q_i).
\]
Combining the above estimates gives \eqref{e:gradA-gradA<L}. The proof of \eqref{e:A-A<L} is just like the proof of \eqref{e:AB-AB'}; we write it for convenience.
We have

Now we prove \eqref{e:def-Li}. Note that 
\begin{align*}
 \|(A_{Q_I} - A_{Q_i}) \grad \Psi_i\|_\infty \lesssim \sup_{x \in 2B(i)} \frac{|A_{Q_I}(x) - A_{Q_i}(x)|}{\ell(Q_i)} \stackrel{\eqref{e:A-A<L}}{\lesssim} L_i.
\end{align*}
All in all we get
\begin{align*}
\|\grad &   \left[ (A_{Q_{I}}-A_{Q_{i}})  \Psi_{i}\right]\|_{L^\infty(\R^n)}\\
&  \leq \|(\grad A_{Q_{I}}-\grad A_{Q_{i}})\Psi_{i}\|_{\infty}+\|(A_{Q_{I}}-A_{Q_{i}})\grad \Psi_{i}\|_{\infty}
\lec L_{i}
\end{align*}
This and \eqref{e:maximal-lip} imply  \eqref{e:def-Li}. 
\end{proof}

\begin{proof}[Proof of Lemma \ref{l:T1-Delta2}]
Recall that we are aiming for a bound on $\int_{3I} |T_1(x)| dx$, where $I \in \Delta_2 \subset \Delta_R$ and that 
\begin{equation*}
    T_1(x) =\sum_{i \in \cI} \left( A_{Q_i}\circ b_i(x) - A_{Q_I} \circ b_i(x)\right) \phi_i(x) = \sum_{i \in \cI \atop 3I_i \cap 3I} \left( A_{Q_i}\circ b_i(x) - A_{Q_I} \circ b_i(x)\right) \phi_i(x),
\end{equation*}
since that $x \in 3I$ implies that $\phi_i(x) \neq 0$ only when $3I_i \cap 3I$. Moreover $\angle(\R^d, P_{Q_i}) \leq \delta_0$. So we have
\begin{align*}
    \int_{3I} |T_1(x)| dx & \lesssim \sum_{i \in \cI \atop 3I_i \cap 3I \neq \emptyset} \int_{b_i(3I)}|A_{Q_I}(x) - A_{Q_i}(x)| (\phi_i\circ \pi)(x) \, d \dH^d|_{P_{Q_i}}\\
    & \lesssim \sum_{i \in \cI \atop 3I_i \cap 3I \neq \emptyset} \int_{b_i(\R^d)}|A_{Q_I}(x) - A_{Q_i}(x)| \Psi_i(x) \, d \dH^d|_{P_{Q_i}}
\end{align*}
By \eqref{e:form52}, we see that
\begin{align*}
    \int_{3I}|T_1(x)| \, dx & \lesssim \sum_{i \in \cI \atop 3I_i \cap 3I \neq \emptyset} \alpha_\sigma^d(MB_{Q_i})\ell(Q_i)^{d+1} L_i \notag \\
    &  + \sum_{i \in \cI \atop 3I_i \cap 3I \neq \emptyset} \int |A_{Q_I}(x) - A_{Q_i}(x)| \Psi_i(x) \, d\sigma(x) := T_{1,1} + T_{1,2}.
\end{align*}
By Lemma \ref{l:alpha-lemmas}, $\alpha_\sigma^d(B) \lesssim 1$ for any ball. We compute
\begin{align*}
	T_{1,1} & =
\sum_{3I_{i}\cap 3I\neq\emptyset } \alpha_{\sigma}^d(MB_{Q_i}) L_{i}\ell(Q_{i})^{d+1} \lesssim \sum_{3I_{i}\cap 3I\neq\emptyset }  L_{i}\ell(Q_{i})^{d+1} \\
& 
  \stackrel{\eqref{e:def-Li}}{\lec} \delta_0\lambda \sum_{3I_{i}\cap 3I\neq\emptyset }\ell(Q_{i})^{d+1}+\sum_{3I_{i}\cap 3I\neq\emptyset }\int_{Q_{i}}|A_{Q_{I}}-f|d\sigma \\
& \stackrel{\eqref{e:MBQiMBQI}}{\leq} \delta_0\lambda \sum_{3I_{i}\cap 3I\neq\emptyset }\ell(Q_{i})^{d+1}
+ \Omega_{f}^{1}(MB_{Q_{I}})\ell(Q_{I})^{d+1}.
\label{e:sumLiQ}
\end{align*}
Thus, since $\Omega_f^1(MB_{Q_I}) \leq \gamma(MB_{Q_I})$, we obtain the required estimate in Lemma \ref{l:T1-Delta2} for the term $T_{1,1}$.
We bound $T_{1,2}$ as follows:
\begin{align*}
T_{1,2}
& \leq   \sum_{3I_{i}\cap I\neq\emptyset }\int \ps{|A_{Q_{I}}(x)-f|+|f-A_{Q_i}(x)|}\Psi_{i}(x) d\sigma(x)\\
& \stackrel{\eqref{e:MBQiMBQI}}{\lec}  \sum_{3I_{i}\cap 3I\neq\emptyset } \ps{\int_{MB_{Q_{I}}}|A_{Q_{I}}(x)-f|\phi_{i}\circ \pi(x) d\sigma(x) +\gamma(MB_{Q_{i}})\ell(Q_{i})^{d+1}}\\
& \lec   \int_{MB_{Q_{I}}} |A_{Q_{I}}(x)-f|\phi_{i}\circ \pi(x) d\sigma(x) + \sum_{3I_i\cap 3I \neq \emptyset}\gamma(MB_{Q_{i}})\ell(Q_{i})^{d+1}\\
& \lec  \gamma(MB_{Q_{I}})\ell(Q_{I})^{d+1}+\sum_{3I_i\cap 3I \neq \emptyset}\gamma(MB_{Q_{i}})\ell(Q_{i})^{d+1}\\
& \stackrel{\eqref{e:gamma-small}}{\lec}  \gamma(MB_{Q_{I}})\ell(Q_{I})^{d+1}+\ve\lambda \sum_{3I_i\cap 3I \neq \emptyset}\ell(Q_{i})^{d+1}\\
\end{align*}
This gives \eqref{e:I1} for the term $T_{1,2}$ as well, and we are done. 
\end{proof}

\subsubsection{Bounds on $T_2$}

In this subsection we prove the following estimate. 

\begin{lemma}
	Let $I \in \Delta_2$. Then
\begin{equation}
\label{e:I2}
\int_{3B_I} \av{  T_{2}(x)} \, dx
  \lec \sum_{3I_{i}\cap 3I\neq\emptyset}  \lambda  \ve \ell(Q_{i})^{d+1} + \gamma(MB_{Q_{I}})\ell(Q_{I})^{d+1}.
\end{equation}
\end{lemma}
\noindent
Recall that 
\begin{align*}
	T_2(x)& = \sum_{i \in \cI} \left(A_{Q_I} \circ b_i (x) - A_{Q_I} b_{Q_I}(x) \right) \phi_i(x) \\
	& = \sum_{i \in \cI \atop 3I_i \cap 3I \neq \emptyset} \left(A_{Q_I} \circ b_i (x) - A_{Q_I} b_{Q_I}(x) \right) \phi_i(x).
\end{align*}
We compute
\begin{align*}
	|T_2(x)| \leq |\grad A_{Q_I}| \sum_{3I_i \cap 3I \neq \emptyset} |b_i(x) - b_{Q_I}(x)| \phi_i(x) \stackrel{\eqref{e:gradA-simLambda}}{\lesssim} \lambda \sum_{3I_i \cap 3I \neq \emptyset} |b_i(x) - b_{Q_I}(x)| \phi_i(x).
\end{align*}
Then, setting $b(x) := (b_i \circ \pi(x) - b_{Q_I}\circ \pi(x)) \Psi_i(x)$, we have
\begin{align*}
	\int_{3B_I} |T_2(x)| \, dx & \lesssim \lambda \sum_{3I_i \cap 3I \neq \emptyset} \int_{3B_I} |b_i (x)-b_{Q_I}| \phi_i(x) \, dx\\
	& \lesssim \lambda \sum_{3I_i \cap 3I \neq \emptyset} \int_{b_i(3B_I)}  |b_i \circ \pi(x) - b_{Q_I} \circ \pi (x)| \phi \circ \pi (x) \, d\dH^d|_{P_{Q_i}}(x) \\
	& \lesssim \lambda \sum_{3I_i \cap 3I \neq \emptyset} \int_{P_{Q_i}}  |b_i \circ \pi(x) - b_{Q_I} \circ \pi (x)| \Psi_i (x) \, d\dH^d|_{P_{Q_i}}(x)\\
	& \lesssim \lambda \sum_{3I_i \cap 3I \neq \emptyset}  \int |b(x)| \big( c_{Q_i}\, d\dH^d|_{P_{Q_i}} - d \sigma\big)(x) + \lambda \sum_{3I_i \cap 3I \neq \emptyset} \int |b(x)| \, d \sigma(x) \\
	& =: \dT_{2,1} + \dT_{2,2}.  
\end{align*}
\noindent
In order to bound $\dT_{2,1}$, we will need the following lemma, similar to Lemma \ref{l:LipAQ-AQi}.
 \begin{lemma}\label{l:Mi}
\begin{equation}
\label{e:def-Mi}
  \Lip((b_{i}\circ \pi-b_{Q_{I}}\circ\pi)\Psi_{i})\lec M_i := \ve+ \avint_{Q_{i}}\frac{|x-b_{Q_{I}}\circ\pi(x)|}{\ell(Q_{i}) }\Psi_{i}(x)d\sigma(x),
\end{equation}
\begin{equation}
\label{e:gradb-gradb<M}
|\grad (b_{i}\circ \pi)-\grad(  b_{Q_{I}}\circ\pi)| \lec M_i
\end{equation}
and for all $x$,
\begin{equation}
\label{e:b-b<M}
|b_{i}\circ \pi(x)-b_{Q_{I}}\circ\pi(x)|\lec M_i(\dist(x,Q_i)+\ell(Q_i)).
\end{equation}
In particular, for $C$ an appropriately chosen dimensional constant, 
\begin{equation}\label{e:form701}
	{\rm Lip}\left( \frac{ (b_{i}\circ \pi-b_{Q_{I}}\circ\pi)\Psi_{i}}{CM_i} \right) \leq 1.
\end{equation}
\end{lemma}

\begin{proof}
We start off by showing \eqref{e:gradb-gradb<M}. If $x \in \R^n$, then $b_i \circ \pi(x) \in P_{Q_i}$. Also, since $\angle(P_{Q_{i}},\R^{d})<\delta_0$, then it can be easily seen that 
\begin{equation}\label{e:form703}
 |x-b_i\circ \pi(x)|\lec \dist(x,P_{Q_{i}}) \;\; \mbox{ for }\;\; x\in\Sigma.
 \end{equation}
 Hence, since ${\rm spt}(\Psi_i) \subset MB_{Q_i}$ by \eqref{e:MBQiMBQI}, 
\begin{align}
\avint_{Q_{i}}|x- b_{i}\circ \pi(x)|\Psi_{i}(x)d\sigma(x)
 \stackrel{\eqref{e:MBQiMBQI}}{\lec} \avint_{MB_{Q_{i}}} \dist(x,P_{Q_{i}})d\sigma(x) \notag \\
\lec \beta_{\sigma}^{d,1}(MB_{Q_{i}})\ell(Q_{i})
 \stackrel{\eqref{e:jones-small}}{\lec} \ve \ell(Q_{i})
 \label{e:bpi-x<eell}
\end{align}
Thus,
  \begin{multline*}
  \avint_{Q_{i}} |b_{i}\circ \pi(x)-b_{Q_{I}}\circ\pi(x)|\Psi_{i}(x)d\sigma(x) \\
  \leq    \avint_{Q_{i}}\left( |b_{i}\circ \pi(x)-x|+|x-b_{Q_{I}}\circ\pi(x)|\right) \Psi_{i}(x) d\sigma(x) \\
   \stackrel{\eqref{e:bpi-x<eell}}{\lec} \ve\ell(Q_{i}) + \avint_{Q_{i}}|x-b_{Q_{I}}\circ\pi(x)|\Psi_{i}(x)d\sigma(x) = M_i\ell(Q_{i}).
  \end{multline*}
Since $E$ is Ahlfors $d$-regular, there is a constant $c$ (depending only of the Ahlfors regularity constants of $E)$ for which we can find balls $B_{0},...,B_{d}$ centered on $\Sigma $ and of radii $c\ell(Q_k)$ so that $2B_k\subseteq c_0 B_{Q_{i}}$ and 
\[
\dist(x_{k+1},{\rm span} \{x_{0},...,x_{k}\})\geq 4c\ell(Q_i).
\]
By Chebyshev's inequality, there are $x_{k}\in B_{k}\cap \Sigma$, $k=0,...,d$, so that
\begin{equation}\label{e:form56}
|b_i \circ \pi(x_k) - b_{Q_I}\circ \pi(x_k)| \lec M_{i}\ell(Q_i).
\end{equation}
Since $x_{k}\in B_{k}$, we have 
\begin{equation*}
\dist(x_{k+1}-x_0, \,{\rm span} \{x_{1}-x_0,...,x_{k}-x_0\})\geq 2c\ell(Q_i).
\end{equation*}
Let $y_k'=\pi_{P_{Q_i}}(x_k)$ and $y_k=\pi(x_k)$. By \eqref{e:QinTree-beta-inf<ve}, $|x_k-y_k'| \lesssim \ve^{\frac{1}{d+1}}\ell(Q_i)$. Also, since $Q_i \in \Tree(R)$, $\angle(\R^d, P_{Q_i})< \delta_0$. Thus 
\begin{align}
    |x_k - y_k| \leq |x_k - y_k'| + |y_k' - y_k| \leq  (\ve^{\frac{1}{d+1}} + \delta_0)\ell(Q_i) \ll c \ell(Q_i)
\end{align} 
for a sufficiently small choice of $\ve$ and $\delta_0$. This in particular implies that
\[
\dist(y_{k+1}-y_0,\, {\rm span} \{y_{1}-y_0,...,y_{k}-y_0\})\geq c\ell(Q_i).
\]
And therefore the vectors $\{u_{k}=y_{k}-y_{0}:k=1,...,d\}$ are linearly independent (with good constants), and form a basis for $\R^d$. Then, with $x \in B(0, 2r_{B(i)})$, 
\begin{align}
|\grad(b_i \circ \pi)(x) - \grad (b_{Q_I} \circ \pi)(x)| \lesssim \max_{1 \leq k \leq d} |(\grad b_i - \grad b_{Q_I})(y_0- y_k)|.\label{e:form57}
\end{align}
Since $b_i$ and $b_{Q_I}$ are affine, $b_i(y_0)- b_{i}(y_k)= \grad b_i(y_0-y_k)$, and the same holds for $b_{Q_I}$. Thus we get
\begin{align*}
    \eqref{e:form57} \lesssim \max_{1 \leq k \leq d} \left| \left(b_i(y_0)-b_i(y_k)\right) + \left(b_{Q_I}(y_0)-b_{Q_I}(y_k)\right) \right| \stackrel{\eqref{e:form56}}{\lesssim} M_i \ell(Q_i)
\end{align*}
This implies that $|\grad(b_i \circ \pi) - \grad (b_{Q_I} \circ \pi)|\lesssim M_i$ and thus proves \eqref{e:gradb-gradb<M}. The proof of \eqref{e:b-b<M} is just like the proof of \eqref{e:AB-AB'}. 

Finally, 
  \begin{multline*}
  \|\grad \ps{(b_{i}\circ \pi-b_{Q_{I}}\circ\pi_{i})\Psi_{i}}\|_{L^{\infty}(\R^n)}\\
   \leq  \|\grad (b_{i}\circ \pi-b_{Q_{I}} \pi_{i}))\Psi_{i}\|_{\infty}
  + \| (b_{i}\circ \pi-b_{Q_{I}}\circ\pi_{i})\grad \Psi_{i}\|_{\infty}\\  
  \lec M_i + M_i \ell(Q_i)\cdot \frac{1}{\ell(Q_i)}\approx M_i
  \end{multline*}
This imply \eqref{e:def-Mi}.
\end{proof}

\noindent
\underline{Estimate for $\dT_{2,1}$.} We have
\begin{align*}
	\dT_{2,1} & = \lambda \sum_{3I_i \cap 3I \neq \emptyset} \int  |b(x)|\frac{CM_i}{CM_i} (c_{Q_i} \, d \dH^d|_{P_{Q_i}} - d \sigma)(x)\\
	&  \lesssim \lambda \sum_{3I_i \cap 3I \neq \emptyset} M_i \alpha_\sigma^d(MB_{Q_i})\ell(Q_i)^{d+1}.
\end{align*}
By a similar argument as before with the $L_i$, we have that
\begin{align}
\sum_{3I_{i}\cap 3I\neq\emptyset} M_{i} \ell(Q_{i})^{d+1}
& 
\stackrel{\eqref{e:def-Mi}}{\leq} \sum_{3I_{i}\cap 3I\neq\emptyset} \ve\ell(Q_{i})^{d+1} + \sum_{3I_{i}\cap 3I\neq\emptyset}  \int_{Q_{i}}|x-b_{Q_{I}}\circ\pi|\Psi_{i}(x)d\sigma(x) \notag \\
& \lec \ve\sum_{3I_{i}\cap 3I\neq\emptyset} \ell(Q_{i})^{d+1}+ \sum_{3I_{i}\cap 3I\neq\emptyset}  \int_{Q_{i}}\dist(x,P_{Q_{I}})\Psi_{i}(x)d\sigma(x) \notag \\
&  \lec  \ve \sum_{3I_{i}\cap 3I\neq\emptyset} \ell(Q_{i})^{d+1} + \beta_{\sigma}^{d,1}(MB_{Q_{I}},P_{Q_{I}})\ell(Q_{I})^{d+1} \notag. 
\end{align}
Since $Q_I \in \Tree(R)$, then $|\grad A_{Q_I}| \approx \lambda$, by \eqref{e:gradA-simLambda}. In particular $\lambda \beta_{E}^{d,1} (MB_{Q_I}) \lesssim \gamma(MB_{Q_I})$. Hence
\begin{equation}\label{e:dT21}
	\dT_{2,1} \leq C \ve \lambda \sum_{3I_{i}\cap 3I\neq\emptyset} \ell(Q_{i})^{d+1} + \lambda \beta_{\sigma}^{d,1}(MB_{Q_{I}},P_{Q_{I}})\ell(Q_{I})^{d+1}
\end{equation}

\noindent
\underline{Estimates for $\dT_{2,2}$. }
We have
  \begin{align*}\label{e:sumx-bI<alphaQI}
  	\dT_{2,2} & = \lambda \sum_{3I_i \cap 3I \neq \emptyset} \int_E |b_i \circ  \pi(x) - b_{Q_I} \circ \pi(x)| \Psi_i(x)  \, d \sigma(x) \notag \\
  	& \lesssim \lambda \sum_{3I_i \cap 3I \neq \emptyset} \int_E |b_i \circ  \pi(x) - x | \Psi_i(x) \, d \sigma(x) + \lambda \sum_{3I_i \cap 3I \neq \emptyset} \int_E | x- b_{Q_I} \circ \pi(x)| \Psi_i(x)  \, d \sigma(x)\\
  	& =: \dT_{2,2,1} + \dT_{2,2,2}.
\end{align*}
Since the supports of $\{\phi_i\}_{i \in \cI}$ have bounded overlap, so do the supports of $\{\Psi_i\}_{i \in \cI}$. Thus we can estimate  
\begin{align*}
  	\dT_{2,2,2} & \lesssim  \lambda \sum_{3I_{i}\cap 3I\neq\emptyset}  \int_{\Sigma} |x-b_{Q_{I}}\circ \pi(x)|\Psi_{i}(x) d\sigma(x) \notag \\
  	& 
   \stackrel{\eqref{e:MBQiMBQI}}{\lec} \lambda \int_{MB_{Q_{I}}} |x-b_{Q_{I}}\circ \pi(x)|\Psi_{i}(x) d\sigma(x) \notag
  \\
   & \leq \lambda \beta_{\sigma}^{d,1}(MB_{Q_{I}},P_{Q_{I}})\ell(Q_{I})^{d+1},
\end{align*}
where for the last inequality we can argue as in \eqref{e:form703}, recalling also that $Q_I \in \Tree(R)$.
Similarly, arguing as in \eqref{e:bpi-x<eell}, we obtain
\begin{equation*}
	\dT_{2,2,1} \lesssim \lambda \sum_{3I_i \cap 3I \neq \emptyset} \beta_\sigma^{d,1}(MB_{Q_i}) \ell(Q_i)\ell(Q_i)^{d+1} \stackrel{\eqref{e:jones-small}}{\leq} C  \ve\lambda \sum_{3I_i \cap 3I \neq \emptyset} \ell(Q_i)^{d+1}.
\end{equation*}
This proves that 
\begin{align}\label{e:dT22}
	\dT_{2,2} \lesssim \lambda \ve \sum_{3I_i \cap 3I \neq \emptyset}\ell(Q_i)^{d+1} + \beta_{E}^{d,1} (MB_{Q_I}) \ell(Q_I)^{d+1}.
\end{align}

Finally, putting together the estimates \eqref{e:dT21} and \eqref{e:dT22}, we obtain that 
\begin{equation}
	\int_{3B_I} |T_2(x)|\,  dx \lesssim \dT_{2,1} + \dT_{2,2} \lesssim \ve \lambda \sum_{3I_i \cap 3I \neq \emptyset} \ell(Q_i)^{d+1} + \lambda \beta_{E}^{d,1} (MB_{Q_I}) \ell(Q_I)^{d+1}.
\end{equation}

\subsubsection{Bounds on $T_3$}

In this subsection we prove the following estimate. 

\begin{lemma}
\begin{equation}
\label{e:I3}
\int_{3I} \av{  T_{3}(x)}
  \lec \sum_{3I_{i}\cap C_1I\neq\emptyset}  \lambda  \ve \ell(Q_{i})^{d+1} + \gamma(MB_{Q_{I}})\ell(Q_{I})^{d+1} + \Omega_h(C_1 B_I) \ell(I)^{d+1}.
\end{equation}
\end{lemma}
\noindent
The term $T_3(x)$ is given by
\begin{align*}
	T_3(x)& =\sum_{i \in \cI}\left(A_{Q_{I}}\circ b_{Q_{I}}(x) - A_{Q_{I}}\circ A_I(x)\right)\phi_{i}(x)\\
	&  = \sum_{3I_i \cap 3I \neq \emptyset}\left(A_{Q_{I}}\circ b_{Q_{I}}(x) - A_{Q_{I}}\circ A_I(x)\right)\phi_{i}(x),
\end{align*}
since ${\rm spt}(\phi_i) \subset 3I_i$.
Recall also from the proof of Lemma \ref{l:MBQiMBQI} that if $3I_{i}\cap 3I\neq\emptyset$, then there is $I_j\subseteq I$ with $I_{i}\sim I_{j}$, so in particular 
\begin{equation*}
\ell(I_{i})\approx \ell(I_{j})\leq \ell(I)
\end{equation*}
and so for $C_1$ large enough, 
\begin{equation*}
\supp \phi_i\subseteq 3I_{i}\subseteq C_1B_I.
\end{equation*}
We compute
\begin{align*}
\int_{3B_I} |T_3(x)|
& =\int_{3B_I}\av{\sum_{3I_i \cap 3I \neq \emptyset} (A_{Q_{I}}\circ b_{Q_{I}}(x) - A_{Q_{I}}\circ A_I(x))\phi_{i}(x)}dx
\\
& \leq \int_{3B_I}\sum_{3I_i \cap 3I \neq \emptyset}|\grad A_{Q_{I}}|\cdot |b_{Q_{I}}(x)-A_I(x)|\phi_{i}(x)dx\\
& \stackrel{\eqref{e:AQ<alphaAR}}{\lec} \lambda  \int_{C_1B_I} |b_{Q_{I}}(x)-A_I(x)|dx\\
& \leq \lambda \int_{C_1B_I} |b_{Q_{I}}(x)-h(x)|+ \lambda \int_{C_1B_I} |h(x)-A_I(x)|dx:= \dT_{3,1} + \dT_{3,2}. 
\end{align*}
We estimate $\dT_{3,2}$ first.
The affine map $A_I$ was chosen to minimise $\Omega_h^1(C_1B_ I)$. Thus 
\begin{equation}\label{e:dT32}
\dT_{3,2} \lesssim 
\lambda \Omega_{h}(C_1B_I)\ell(I)^{d+1}.
\end{equation}
We focus on $\dT_{3,1}$. Compute
\begin{align*}
 \int_{C_1B_I} |b_{Q_{I}}(x)-h(x)|dx
 & \stackrel{\eqref{e:def-h}}{=}\int_{C_1B_I} \av{b_{Q_{I}}(x)-\sum_{i \in \cI}b_{i}(x)\phi_{i}(x)}dx\\
& =\int_{C_1B_I} \av{\sum_{i \in \cI}(b_{Q_{I}}(x)-b_{i}(x))\phi_{i}(x)}dx\\
& \leq \sum_{3I_i\cap C_1I \neq\emptyset} \int |b_{Q_{I}}(x)-b_{i}(x)|\phi_{i}(x)dx.
\end{align*}
Note that this latter integral is equal to $\dT_{2,2}$, and therefore it can be estimated in the same way. We thus conclude that
\begin{equation}\label{e:dT31}
	\dT_{3,1} \stackrel{\eqref{e:dT22}}{\lesssim} \lambda \ve \sum_{3I_i \cap C_1 I \neq \emptyset} \ell(Q_i)^{d+1} + \beta_{E}^{d,1}(MB_{Q_I}) \ell(Q_I)^{d+1}.
\end{equation}
Note again that since $Q_I \in \Tree(R)$, then $|\grad Q_I| \approx \lambda$, and therefore $\lambda \beta_{E}^{d,1}(MB_{Q_I}) \lesssim \gamma(MB_{Q_I})$.
All in all, we obtain
\begin{equation}\label{e:est-T3}
\int_{3B_I} |T_{3}(x)|\, dx
\lec \lambda \Omega_{h}(C_1B_I)\ell(I)^{d+1}+  \gamma(MB_{Q_{I}})\ell(Q_{I})^{d+1}+\sum_{3I_i\cap C_1I\neq\emptyset}    \lambda\ve \ell(Q_{j})^{d+1}
\end{equation}
which  implies \eqref{e:I3}.

\subsubsection{Putting together the estimates for $T_1, T_2$ and $T_3$}

We finally combine our estimates (and recall that $\ve\ll \delta_0$) to get

\begin{multline*}
\Omega_{F}(3B_I)\ell(I)^{d+1}  \leq \int_{3B_I}|F(x)-A_{Q_{I}}\circ A_{I}(x)|dx
 \stackrel{\eqref{e:def-T1T2T3}}{\leq} \sum_{i=1}^{3}\int  |T_i(x)|\, dx \\
  \lec  \lambda \Omega_{h}(C_1B_I)\ell(I)^{d+1} + \delta_0\lambda  \sum_{3I_{i}\cap C_1I\neq\emptyset}\ell(Q_{i})^{d+1}  +  \gamma(MB_{Q_{I}})\ell(Q_{I})^{d+1}
\end{multline*}

Recalling that $\ell(I)\approx \ell(Q_I)$ and $\ell(Q_i)\approx \ell(I_i)$, we have 
\begin{align*}
\sum_{I\in \Delta_{2}}\Omega_{F}(3B_I)^{2}|I|
& \lec \sum_{I\in \cC_{2}}\lambda \Omega_{h}(C_1I)^{2}|I|\\
& \qquad + \sum_{I\in \Delta_{2}} (\delta_0\lambda)^2 \ps{\sum_{3I_i\cap 3I\neq\emptyset}\frac{\ell(I_{i})^{d+1}}{\ell(I)^{d+1}}}^{2}|I| \\
& \qquad \qquad+ \sum_{I\in \Delta_{2}}\gamma(MB_{Q_{I}})^2|I|
=: S_1+S_2+S_3.
\end{align*}
First we estimate $S_1$. Recall that $h(x)=x+H(x)$ and $h$ is a $C\delta_0$-Lipschitz function supported in $V_K$. Since the identity map is affine, $\Omega_{h}^{1}=\Omega_{H}^{1}$, and so by Dorronsoro's Theorem,
\[
S_{1}  = \sum_{I\in \Delta_{2}}\lambda \Omega_{H}(C_1B_I)^{2}\ell(I)^{d}
\lec \int|\grad H|^{2}\, dx \lec \delta_0^2 |V_{K}|\approx \delta_0^2|R|.
\]

Next we estimate $S_2$: By Jensen's inequality,

\begin{align*}
S_{2}
& \lec   (\delta\lambda)^2\sum_{I\in \Delta_{2}}  \sum_{3I_i\cap C_1I\neq\emptyset}\frac{\ell(I_{i})^{d+2}}{\ell(I)^{d+2}}\ell(I)^{d}\\
&  \lesssim (\delta_0\lambda)^2\sum_{3I_{i}\cap C_1I\neq\emptyset} \ell(I_{i})^{d+2} \sum_{I\in \Delta_{2}\atop  3I_{i}\cap C_1I\neq\emptyset} \ell(I)^{-2}\\
& \lec (\delta_0\lambda)^2\sum_{3I_{i}\cap C\pi(B_{R})\neq\emptyset} \ell(I_{i})^{d}
\lec (\delta_0\lambda)^2\ell(R)^{d}
\end{align*}
\noindent
Here we used the fact that $\ell(I_i)\lec \ell(I)$ whenever 
\begin{equation}
\label{e:3IjCI}
3I_i\cap C_1I\neq\emptyset
\end{equation}
 and there are boundedly many dyadic cubes of any given side length satisfying this property and so the second sum in the second line is essentially a geometric series. We also used the fact that $I\in \Delta_{2}$ implies $3I\cap 3I_i$ for some $i\in \cI_K$, but those were cubes that intersected $U_K$ and they have side lengths $\lec \ell(R)$, and moreover $I\in \Delta_{2}$ implies $\ell(I)\leq \ell(R)$, and so any $I_i$ satisfying \eqref{e:3IjCI} must have $\ell(I_i)\lec \ell(I)\lec \ell(R)$ and must be contained in $C\pi(B_{R})$ for some large enough constant $C$.
\\

Lastly, we handle $S_3$. For this, we just observe that, given $Q\in \Tree(R)$, there can only be boundedly many dyadic cubes $I$ for which $Q_I=Q$, and so 
\[
S_{3} \lec \sum_{Q\in \Tree(R)} \gamma(MB_{Q})^{2}|Q|\lec \ve^{2}|R|.
\]

Combining these estimates together gives \eqref{e:C2}.

\subsection{Estimates for $\Delta_{3}$}

Finally, the goal of this section is to prove

\begin{equation}
\label{e:C3}
\sum_{I\in \Delta_{3}}\Omega_{F}^{1}(3B_I)^{2}|I|\lec (\lambda_0\delta)^{2}|R|.
\end{equation}
\noindent
Recall that $\Delta_3\subset \Delta_K$ is the family of dyadic cubes so that each $I \in \Delta_3$ has 
\begin{itemize}
	\item $3I \cap 3I_i \neq \emptyset$ for some $i \in \cI_K$;
	\item each $I \in \Delta_3$ is not contained in any $I_i$, for $i \in \cI$; and
	\item $\ell(I)> \ell(R)$.
\end{itemize} 
For $I \in \Delta_3$, let $I_i \in \{I_j\}_{j \in \cI_K}$ be so that $3I \cap 3I_i \neq \emptyset$. We claim that $\ell(I_i)\lesssim \ell(I)$. Indeed, note that if $I_i \in \dW$ and $I_i \cap V_K \neq \emptyset$, then 
\begin{equation*}
	\ell(I_i) < \frac{1}{20} D_R(I_i) \leq \frac{1}{20} (\ell(R) + \dist(I, \pi(R)) \lesssim \ell(R) \lesssim \ell(I).
\end{equation*}
Clearly then there are only boundedly many cubes $I \in \dI$ of some given sidelength satisfying this, that is
\begin{align}\label{e:form707}
	\# \{ I \in \Delta_3 \, |\, \ell(I) \approx 2^j \ell(R)\} \lesssim 1 \, \, \mbox{ for all } \, \, j \geq 0.
\end{align} 
Recall from \eqref{e:A=0}, that $F(x)=A_R(x)$ for all $x \in \R^n \setminus V_K = K^2\pi(B_R)$. Thus, 
\begin{align*}
	\Omega_F^1(3B_I)^2 & \lesssim \left( \frac{1}{\ell(I)^d} \int_{V_K} \frac{|F(x)-A_R(x)|}{\ell(I)} \, dx \right)^2 \\
	& \lesssim \frac{\ell(R)^{2d+2}}{\ell(I)^{2d+2}} \Omega_F^1(M^2B_R)^2.
\end{align*}
Hence, keeping in mind also \eqref{e:form707}, 
\begin{align*}
	\sum_{I \in \Delta_3} \Omega_F^1(3B_I)^2 |I| \lesssim \Omega_F^1(M^2 B_R) \, \ell(R)^{2d+2}\sum_{I \in \Delta_3} \ell(I)^{-d-1} \lesssim \Omega_F^1(M^2B_R)^2\ell(R)^d. 
\end{align*}
Since by the way we constructed $F$ we have that $\Omega_F^1(M^2B_R) \lesssim \delta_0$, we obtain \eqref{e:C3}.

\bibliography{bibliography}

@article{AMV2,
	title={Smooth extensions for domains with uniformly rectifiable boundary},
	author={Azzam, Jonas and Mourgoglous, Mihalis and Villa, Michele},
	journal={Preprint},
	year={2025}
}

@article{azzam2018analyst,
  title={An {Analyst’s} traveling salesman theorem for sets of dimension larger than one},
  author={Azzam, Jonas and Schul, Raanan},
  journal={Math. Ann.},
  volume={370},
  number={3-4},
  pages={1389--1476},
  year={2018},
  publisher={Springer},
  doi={10.1007/s00208-017-1609-0},
  url={https://www.dropbox.com/s/gtw7uzrcewl9mh0/dTSP-Updating.pdf?dl=0}
}

@article{david-semmes91,
	title={Singular integrals and rectifiable sets in {$\mathbb{R}^n$}: {Au-del\`{a} des graphes lipschitziens}},
	author={David, Guy and Semmes, Stephen},
	journal="Ast{\'e}risque",
	volume={193},
	year={1991},
	doi={10.24033/ast.68},
}

@book{david-semmes93,
	title={{Analysis of and on Uniformly Rectifiable Sets}},
	author={David, Guy and Semmes, Stephen},
	volume={38},
	year={1993},
	series={Mathematical Surveys and Monographs},
	publisher={American Mathematical Society},
	doi={10.1090/surv/038},
}

@article{jones90,
	title={Rectifiable sets and the traveling salesman problem},
	author={Jones, Peter W},
	journal="Invent. Math.",
	volume={102},
	number={1},
	pages={1--15},
	year={1990},
	publisher={Springer},
	doi={10.1007/BF01233418},
}

@book{heinonen2005lectures,
  title={Lectures on Lipschitz analysis},
  author={Heinonen, Juha},
  number={100},
  year={2005},
  publisher={University of Jyv{\"a}skyl{\"a}}
}

@book{mattila,
	title={Geometry of sets and measures in {Euclidean} spaces: fractals and rectifiability},
	author={Mattila, Pertti},
	series={Cambridge studies in advanced mathematics},
	volume={44},
	year={1995},
	publisher={Cambridge University Press},
	doi={10.1017/CBO9780511623813},
}

@book{tolsa-book,
	title={Analytic capacity, the {Cauchy transform, and non-homogeneous Calder{\'o}n-Zygmund theory}},
	author={Tolsa, Xavier},
	volume={307},
	year={2014},
	publisher={Birkh\"{a}user},
	doi={10.1007/978-3-319-00596-6},
	series={Progress in Mathematics},
}

@article{villa2020tangent,
  title={Tangent points of lower content d-regular sets and $\beta$ numbers},
  author={Villa, Michele},
  journal={J. Lond. Math. Soc.},
  volume={101},
  number={2},
  pages={530--555},
  year={2020},
  publisher={Wiley Online Library},
  doi={10.1112/jlms.12275}
}

@article{leger1999menger,
  title={{Menger} curvature and rectifiability},
  author={L{\'e}ger, Jean-Christophe},
  journal={Ann. Math.},
  volume={149},
  number={3},
  pages={831--869},
  year={1999},
  publisher={JSTOR}
}

@article{tolsa2009uniform,
  title={Uniform rectifiability, {Calder{\'o}n--Zygmund} operators with odd kernel, and quasiorthogonality},
  author={Tolsa, Xavier},
  journal={Proc. Lond. Math. Soc.},
  volume={98},
  number={2},
  pages={393--426},
  year={2009},
  publisher={Oxford University Press}
}

@article{villa2019square,
  title={A square function involving the center of mass and rectifiability},
  author={Villa, Michele},
  journal={Math. Zeitschrift (to appear)},
  year={2022},
  eprint={1910.13747},
}

@book{ambrosio2000functions,
  title={{Functions} of bounded variation and free discontinuity problems},
  author={Ambrosio, Luigi and Fusco, Nicola and Pallara, Diego},
  year={2000},
  publisher={Courier Corporation}
}

@book{simon1983lectures,
  title={{Lectures} on geometric measure theory},
  author={Simon, Leon},
  year={1983},
  publisher={The Australian National University, Mathematical Sciences Institute}
}

@article{hajlasz2000sobolev,
  title={Sobolev met poincar{\'e}},
  author={Haj{\l}asz, Piotr and Koskela, Pekka},
  year={2000},
  journal={Mem. Amer. Math. Soc.},
  publisher={American Mathematical Soc.}
}

@book{delellis2008rectifiable,
    title={Rectifiable sets, densities and tangent measures},
    author={De Lellis, Camillo},
    publisher={European Mathematical Society},
    year={2008}
}

@book{federer2014geometric,
  title={Geometric measure theory},
  author={Federer, Herbert},
  year={2014},
  publisher={Springer}
}

@article{dorronsoro1985characterization,
  title={A characterization of potential spaces},
  author={Dorronsoro, Jos{\'e} R},
  journal={Proc. Amer. Math. Soc.},
  volume={95},
  number={1},
  pages={21--31},
  year={1985}
}

@article{semmes2001real,
  title={Real analysis, quantitative topology, and geometric complexity},
  author={Semmes, Stephen},
  journal={Publ. Mat.},
  pages={265--333},
  year={2001},
  publisher={JSTOR}
}

@article{heinonen1998quasiconformal,
	title={Quasiconformal maps in metric spaces with controlled geometry},
	author={Heinonen, Juha and Koskela, Pekka},
	journal={Acta Math.},
	volume={181},
	number={1},
	pages={1--61},
	year={1998},
	publisher={Springer}
}

@article{azzam2021poincare,
	title={{Poincar{\'e}} inequalities and uniform rectifiability},
	author={Azzam, Jonas},
	journal={Rev. Mat. Iberoam.},
	year={2021}
}

@book{maggi2012sets,
	title={Sets of finite perimeter and geometric variational problems: an introduction to {Geometric Measure Theory}},
	author={Maggi, Francesco},
	number={135},
	year={2012},
	publisher={Cambridge University Press}
}

@article{mourgoglou2021regularity,
	title={The regularity problem for the {Laplace} equation in rough domains},
	author={Mourgoglou, Mihalis and Tolsa, Xavier},
	journal={arXiv preprint arXiv:2110.02205},
	year={2021}
}

@article{mattila2021rectifiability,
    title={Rectifiability: a survey},
    author={Mattila, Pertti},
    journal={arXiv preprint},
    eprint={2112.00540v1},
    year={2021}
}

@article{azzam2016bi,
  title={{Bi-Lipschitz} parts of quasisymmetric mappings},
  author={Azzam, Jonas},
  journal={Rev. Mat. Iberoam.},
  volume={32},
  number={2},
  pages={589--648},
  year={2016}
}

@article{hytonen2016quantitative,
  title={Quantitative affine approximation for {UMD} targets},
  author={Hyt{\"o}nen, Tuomas and Li, Sean and Naor, Assaf},
  journal={Disc. Anal.},
  volume={2016},
  number={6},
  pages={1--37},
  year={2016},
  publisher={Scholastica}
}

@article{fassler2020dorronsoro,
  title={Dorronsoro's theorem in {Heisenberg} groups},
  author={F{\"a}ssler, Katrin and Orponen, Tuomas},
  journal={Bull. Lond. Math. Soc.},
  volume={52},
  number={3},
  pages={472--488},
  year={2020},
  publisher={Wiley Online Library}
}

@article{orponen2021integralgeometric,
  title={An integralgeometric approach to {Dorronsoro} estimates},
  author={Orponen, Tuomas},
  journal={Int. Math. Res. Not.},
  volume={2021},
  number={21},
  pages={17170--17200},
  year={2021},
  publisher={Oxford University Press}
}

@book{stein2016singular,
  title={{Singular Integrals} and {Differentiability Properties of Functions}},
  author={Stein, Elias M},
  year={2016},
  publisher={Princeton university press}
}

@article{alabern2012new,
  title={A new characterization of {Sobolev} spaces on {$\R^d$}},
  author={Alabern, Roc and Mateu, Joan and Verdera, Joan},
  journal={Math. Ann.},
  volume={354},
  number={2},
  pages={589--626},
  year={2012},
  publisher={Springer}
}

@article{barcelo2020characterization,
  title={Characterization of {Sobolev} spaces on the sphere},
  author={Barcel{\'o}, Juan Antonio and Luque, Teresa and P{\'e}rez-Esteva, Salvador},
  journal={J. Math. Anal. App.},
  volume={491},
  number={1},
  pages={124240},
  year={2020},
  publisher={Elsevier}
}

@inproceedings{hofmann2020uniform,
  title={{Uniform} rectifiability and {$\varepsilon$}-approximability of harmonic functions in {$L^p$}},
  author={Hofmann, Steve and Tapiola, Olli},
  booktitle={Ann. Inst. Fourier},
  volume={70},
  number={4},
  year={2020},
  organization={Centre Mersenne; l'Institut Fourier.}
}

@article{varopoulos1977bmo,
  title={{BMO} functions and the {$\partial$}-equation},
  author={Varopoulos, Nicholas},
  journal={Pacific J. Math.},
  volume={71},
  number={1},
  pages={221--273},
  year={1977},
  publisher={Mathematical Sciences Publishers}
}

@article{varopoulos1978remark,
  title={A remark on functions of bounded mean oscillation and bounded harmonic functions},
  author={Varopoulos, N},
  journal={Pacific J. Math},
  volume={74},
  number={1},
  pages={257--259},
  year={1978}
}

@book{garnett2007bounded,
  title={Bounded analytic functions},
  author={Garnett, John},
  volume={236},
  year={2007},
  publisher={Springer Science \& Business Media}
}

@article{hofmann2016uniform,
  title={Uniform rectifiability, {Carleson} measure estimates, and approximation of harmonic functions},
  author={Hofmann, Steve and Martell, Jos{\'e} Mar{\'\i}a and Mayboroda, Svitlana},
  journal={Duke Math. J.},
  volume={165},
  number={12},
  pages={2331--2389},
  year={2016},
  publisher={Duke University Press}
}

@article{garnett2018uniform,
  title={Uniform rectifiability from {Carleson} measure estimates and {$\varepsilon$-approximability} of bounded harmonic functions},
  author={Garnett, John and Mourgoglou, Mihalis and Tolsa, Xavier},
  journal={Duke Math. J.},
  volume={167},
  number={8},
  pages={1473--1524},
  year={2018},
  publisher={Duke University Press}
}

@article{bortz2019approx,
  title={{$\varepsilon$}-approximability of harmonic functions in {$L^p$} implies uniform rectifiability},
  author={Bortz, Simon and Tapiola, Olli},
  journal={Proc. Amer. Math. Soc.},
  volume={147},
  number={5},
  pages={2107--2121},
  year={2019}
}

@article{hytonen2018bounded,
  title={{Bounded} variation approximation of {$L^p$} dyadic martingales and solutions to elliptic equations},
  author={Hyt{\"o}nen, Tuomas and Ros{\'e}n, Andreas},
  journal={J. European Math. Soc.},
  volume={20},
  number={8},
  pages={1819--1850},
  year={2018}
}

@article{hofmann2021uniform,
  title={Uniform rectifiability implies Varopoulos extensions},
  author={Hofmann, Steve and Tapiola, Olli},
  journal={Advances in Mathematics},
  volume={390},
  pages={107961},
  year={2021},
  publisher={Elsevier}
}

@book{duoandikoetxea2001fourier,
	title={Fourier analysis},
	author={Duoandikoetxea, Javier and Zuazo, Javier Duoandikoetxea},
	volume={29},
	year={2001},
	publisher={American Mathematical Soc.}
}

@book{christ1991lectures,
	title={Lectures on singular integral operators},
	author={Christ, Francis Michael},
	volume={77},
	year={1991},
	publisher={American Mathematical Soc.}
}

@article{BHS23,
	title={Uniformly rectifiable metric spaces: Lipschitz images, {Bi-Lateral Weak Geometric} {Lemma} and {Corona Decompositions}},
	author={Bate, David and Hyde, Matthew and Schul, Raanan},
	journal={Preprint},
	year={2023}
}

@article{mourgoglou2023varopoulos,
	title={{Varopoulos'} extensions of boundary functions in {$L^p$} and BMO in domains with {Ahlfors}-regular boundaries},
	author={Mourgoglou, Mihalis and Zacharopoulos, Thanasis},
	journal={arXiv preprint arXiv:2303.10717},
	year={2023}
}
\bibliographystyle{halpha-abbrv}

\end{document}